\documentclass[]{amsart}

\usepackage{amsmath}
\usepackage{amssymb,amsfonts}
\usepackage{xypic}
\xyoption{all}
\usepackage[pdfborder={0 0 0}]{hyperref}

\newtheorem{theorem}{Theorem}[section]

\newtheorem{lemma}[theorem]{Lemma}
\newtheorem{proposition}[theorem]{Proposition}
\newtheorem{remark}[theorem]{Remark}
\newtheorem{example}[theorem]{Example}
\newtheorem{corollary}[theorem]{Corollary}

\def\B{\ensuremath{\mathcal{B}}}
\def\A{\ensuremath{\mathcal{A}}}
\def\C{\ensuremath{\mathcal{C}}}
\def\v{\ensuremath{\mathrm{v}}}
\newcommand{\x}{\mathbf{x}}
\newcommand{\aso}{\boldsymbol{a}}
\newcommand{\br}{\boldsymbol{r}}
\newcommand{\bl}{\boldsymbol{l}}
\newcommand{\brep}{\ensuremath{\mathbb{L}\mathrm{ ax}}}
\newcommand{\brepg}{\ensuremath{\mathbb{G}\mathrm{raph}}}
\def\ner{\ensuremath{\underline{\mathrm{N}}}}
\def\Set{\ensuremath{\mathbf{Set}}}
\def\Cat{\ensuremath{\mathbf{Cat}}}

\newcommand{\BB}{\ensuremath{\mathrm{B}}}
\newcommand{\Ob}{\mathrm{Ob}}
\newcommand{\N}{\mathcal{N}}
\newcommand{\M}{\mathcal{M}}

\newcommand{\D}{\mathcal{D}}

\newcommand{\PP}{\mathcal{P}}
\newcommand{\QQ}{\mathcal{Q}}
\newcommand{\HH}{\mathcal{H}}
\newcommand{\GG}{\mathcal{G}}

\title{\em Bicategorical homotopy pullbacks}
\author{A.M. Cegarra, B.A. Heredia, J.Remedios}
\address{
\newline
Departamento de \'{A}lgebra, Universidad de
Granada, 18071 Granada, Spain, 
\newline acegarra@ugr.es
\newline \newline
Departamento de \'{A}lgebra, Universidad de
Granada,
18071 Granada, Spain, 
\newline  baheredia@ugr.es\newline \newline
Departamento de Matem\'{a}tica Fundamental, Universidad de La
Laguna, 38271 La Laguna, Spain, 
\newline jremed@ull.es } 

\thanks{This work has been supported by DGI of Spain, Project
  MTM2011-22554. Also, the second author by FPU grant FPU12-01112. The
  third author acknowledges and thanks the support and hospitality of
  the Algebra Department at the University of Granada}


\begin{document}
\maketitle
\begin{abstract} The homotopy theory of higher categorical structures has become a relevant
part of the machinery of algebraic topology and algebraic K-theory,
and this paper contains contributions to the study of the
relationship between B\'enabou's bicategories and the homotopy types
of their classifying spaces. Mainly, we state and prove an extension
of Quillen's Theorem B by showing, under reasonable necessary
conditions, a bicategory-theoretical interpretation of the
homotopy-fibre product of the continuous maps induced on classifying
spaces by a diagram of bicategories $\A\to\B\leftarrow \A'$.
Applications are given for the study of homotopy pullbacks of
monoidal categories and of crossed modules.

\end{abstract}
\section{Introduction and summary} If $A\overset{\phi}\to B\overset{\ \phi'}\leftarrow A'$
are continuous maps
between topological spaces, its
 {\em homotopy-fibre product} $A\times^{_\mathrm{h}}_BA'$ is the subspace of
the product $A\times B^I\!\times A'$, where $I=[0,1]$ and $B^{I}$ is
taken with the compact-open topology, whose points are triples
$(a,\gamma,a')$ with
 $a\in A$, $a'\in A'$, and $\gamma: \phi a\to \phi'a'$ is a path in $B$ joining  $\phi a$  and $\phi' a'$, that is
$\gamma:I\to B$ is a path starting at  $\gamma 0=\phi a$  and ending
at  $\gamma 1=\phi'a'$. In particular, the {\em homotopy-fibre} of a
continuous map $\phi:A\to B$ over a base point $b\in B$ is
$\mathrm{Fib}(\phi,b)=A\times^{_\mathrm{h}}_B\{b\}$, the
homotopy-fibre product of $\phi$ and the constant inclusion map
$\{b\}\hookrightarrow B$. That is, $\mathrm{Fib}(\phi,b)$ is the
space of pairs $(a,\gamma)$, where $a\in A$, and $\gamma:\phi a \to
b$ is a path in $B$ joining $\phi a$  with the base point $b$.

If $\A\overset{F}\to \B\overset{\ \,F'}\leftarrow \A'$ are now
functors between (small) categories, its
 {\em homotopy-fibre product category} is the comma category $F\!\downarrow \!F'$ consisting of triples
$(a,f,a')$ with $f:Fa\to F'a'$  a morphism in $\B$, in which a
morphism from $(a_0,f_0,a'_0)$ to $(a_1,f_1,a'_1)$ is a pair of
morphisms $u:a_0\to a_1$ in $\A$ and $u':a'_0\to a'_1$ in $A'$ such
that
 $F'u'\circ f_0=f_1\circ Fu$.
In particular, the {\em homotopy-fibre category} $F\!\downarrow \!b$
of a functor $F:\A\to \B$, relative to an object $b\in \Ob\B$, is
the homotopy-fibre product category of $F$ and the constant functor
$\{b\}\hookrightarrow \B$. These naive categorical emulations of the
topological constructions are, however, subtle. Let $\BB:\Cat \to
\mathbf{Top}$ be the classifying space functor. The homotopy-fibre
product category  $F\!\downarrow \!F'$ comes with a canonical map
from its classifying space to the homotopy-fibre product space of
the induced maps $\BB F:\BB\A\to \BB\B$ and $\BB F':\BB\A'\to
\BB\B$,
 and Barwick and Kan ~\cite{BK2011, BK2013} have proven that {\em this canonical map
$\BB(F\!\downarrow \!F') \to \BB\A\times^{_\mathrm h}_{\BB
\B}\BB\A'$ is a homotopy equivalence whenever the maps $\BB
(F\!\downarrow \!b_0)\to \BB (F\!\downarrow \!b_1)$, induced by the
different morphisms $b_0\to b_1$ of $\B$, are homotopy
equivalences}. This result extends the well-known Quillen's Theorem
B, which asserts that {\em under such an hypothesis, the  canonical
maps $\BB(F\!\downarrow \!b)\to \mathrm{Fib}(\BB F,\BB b)$ are
homotopy equivalences}. Let us stress that Theorem B and its
consequent Theorem A have been fundamental for higher algebraic
K-theory since the early 1970s, when Quillen ~\cite{Quillen1973}
published his seminal paper, and they are now two of the most
important theorems in the foundation of homotopy theory.

Similar categorical lax limit constructions have been used to
describe homotopy pullbacks in many settings of enriched categories,
where a homotopy theory has been established (see Grandis
~\cite{Gran94}, for instance). Here, we focus on bicategories. Like
categories, small B\'{e}nabou bicategories ~\cite{Benabou1967} and,
in particular, 2-categories and Mac Lane's monoidal categories, are
closely related to topological spaces through the classifying space
 construction, as shown by
Carrasco, Cegarra, and Garz\'on in ~\cite{CCG2010}. This assigns to
each  bicategory $\B$ a CW-complex $\BB \B$, whose cells give a
natural geometric meaning to the cells of the bicategory. By this
assignment, for example, bigroupoids correspond to homotopy 2-types,
that is, to CW-complexes whose $n^{\mathrm{th}}$ homotopy groups at
any base point vanish for $n\geq 3$ (see Duskin ~\cite[Theorem
8.6]{Duskin2002}), and homotopy regular monoidal categories  to
delooping spaces of the classifying spaces of the underlying
categories (see Jardine ~\cite[Propositions 3.5 and
3.8]{Jardine1991}).

In the preparatory Section \ref{Preliminaires} of this paper, for
any diagram $\xymatrix@C=18pt{\A\ar[r]^-{F} & \B
&\A'\ar[l]_(.4){F'}}$,
 where $\A$, $\B$, and $\A'$ are bicategories, $F$ is a lax
functor, and $F'$ is an oplax functor (for instance, if $F$ and $F'$
are both homomorphisms),  we present  a {\em homotopy-fibre product
bicategory} $F\!\downarrow\!F'$, whose 0-cells, or objects, are
triples $(a,f,a')$ with $f:Fa\to F'a'$ a $1$-cell in $\B$ as in the
case when $F$ and $F'$ are functors between categories. But now, a
$1$-cell from $(a_0,f_0,a'_0)$ to $(a_1,f_1,a'_1)$ is a triple
$(u,\beta,u')$
 consisting of $1$-cells $u:a_0\to a_1$ in $\A$ and
 $u':a'_0\to a'_1$ in $\A'$, together with a $2$-cell $\beta: F'u'\circ f_0\Rightarrow f_1\circ Fu$ in $\B$.
And $F\!\downarrow\!F'$ has  $2$-cells
$(\alpha,\alpha'):(u,\beta,u')\Rightarrow (v,\gamma,v')$, which are
given by $2$-cells $\alpha:u\Rightarrow v$ in $\A$ and
 $\alpha':u'\Rightarrow v'$ in $\A'$ such that  $(1_{f_1}\circ F\alpha)\cdot \beta =(\gamma\circ F'\alpha')\circ
1_{f_0}$. In particular, for any object $b\in \B$, we have the {\em
homotopy-fibre bicategories} $F\!\downarrow\!b$ and
$b\!\downarrow\!F'$, in
 terms of which we state and prove our main results  of  the paper. These are exposed in Section \ref{mainSection}, and
they can be summarized as follows (see Theorem \ref{mainTheorem} and
Corollary \ref{hfth}):

$\bullet$ {\em For any diagram of bicategories
$\xymatrix@C=14pt{\A\ar[r]^-{F} & \B &\A'\ar[l]_(.4){F'}}$, where
$F$ is a lax functor and $F'$ is an oplax functor, there is a
canonical map $\BB(F\!\downarrow \!F') \to \BB\A\times^{_{\mathrm
h}}_{\BB \B}\BB\A'$, from the classifying space of the
homotopy-fibre product bicategory to the homotopy-fibre product
space of the induced maps $\BB F:\BB\A\to \BB\B$ and $\BB
F':\BB\A'\to \BB\B$.}

 \vspace{0.15cm}$\bullet$ {\em For a given lax functor $F:\A\to \B$,
the following properties are equivalent:

- For any oplax functor $F':\A'\to \B$, the map $\BB(F\!\downarrow
\!F') \to \BB\A\times^{_{\mathrm h}}_{\BB \B}\BB\A'$ is a homotopy
equivalence.

-  For any 1-cell $b_0\to b_1$ of $\B$, the map $\BB (F\!\downarrow
\!b_0)\to \BB (F\!\downarrow \!b_1)$ is a homotopy equivalence.

- For any 0-cell $b$ of $\B$, the  map $\BB(F\!\downarrow \!b)\to
\mathrm{Fib}(\BB F,\BB b)$ is a homotopy equivalence. }

\vspace{0.15cm} $\bullet$ {\em For a given oplax functor $F':\A'\to
\B$, the following properties are equivalent:

-For any lax functor $F:\A\to \B$, the map $\BB(F\!\downarrow \!F')
\to \BB\A\times^{_{\mathrm h}}_{\BB \B}\BB\A'$ is a  homotopy
equivalence.

-  For any 1-cell $b_0\to b_1$ of $\B$, the map $\BB
(b_1\!\downarrow \!F')\to \BB (b_0\!\downarrow \!F)$ is a homotopy
equivalence.

- For any 0-cell $b$ of $\B$, the  map $\BB(b\!\downarrow \!F')\to
\mathrm{Fib}(\BB F',\BB b)$ is a homotopy equivalence. }

\vspace{0.15cm} Let us remark that,  if the map $\BB(F\!\downarrow
\!F') \to \BB\A\times^{\mathrm h}_{\BB \B}\BB\A'$ is a homotopy
equivalence, then, by Dyer and Roitberg ~\cite{DR1980}, there are
Mayer-Vietoris type long exact sequences on homotopy groups $$
\xymatrix@C=12pt{ \cdots\to \pi_{n+1}\BB \B\ar[r]
&\pi_n\BB(F\!\downarrow\!F')\ar[r]& \pi_n\BB\A\times\pi_n\BB\A'
\ar[r]&\pi_n\BB \B\to \cdots} .$$

The above results include the aforementioned results by Barwick and
Kan, but also the extension of Quillen's Theorems A and B to lax
functors between bicategories stated by Calvo, Cegarra, and Heredia
in ~\cite[Theorem 5.4]{CCH2013}, as  well as the generalized Theorem
A for lax functors from categories  into  2-categories by del Hoyo
in ~\cite[Theorem 6.4]{Hoyo2012} (see Corollaries \ref{hfth} and
\ref{tagen}). Related to this, an interesting relative Theorem A for
lax functors between 2-categories has recently been proven by Chiche
in ~\cite{Chiche2012}.

We also study conditions on a bicategory $\B$ in order to ensure that the space
$\BB(F\!\downarrow \!F')$ is always homotopy equivalent to the
homotopy-fibre product of the induced maps $\BB F:\BB \A\to \BB\B$
and $\BB F':\BB \A'\to \BB\B$. Thus, in Theorem \ref{xbs}, we prove

\vspace{0.15cm} $\bullet$ {\em For a bicategory $\B$, the following
properties are equivalent:

- For any diagram $\xymatrix@C=14pt{\A\ar[r]^-{F} & \B
&\A'\ar[l]_(.4){F'}}$, where $F$ is a lax functor and $F'$ is an
oplax functor, the map $\BB(F\!\downarrow \!F') \to
\BB\A\times^{_{\mathrm h}}_{\BB \B}\BB\A'$ is a homotopy equivalence

-  For any object $b$ and $1$-cell $b_0\to b_1$ in $\B$, the induced
map $\BB\B(b,b_0) \to \BB\B(b, b_1)$ is a homotopy equivalence.

-  For any object $b$ and $1$-cell $b_0\to b_1$ in $\B$, the induced
map $\BB\B(b_1,b) \to \BB\B(b_0, b)$ is a homotopy equivalence.

- For any two objects $b,b'\in\B$, the canonical map $$
\BB\B(b,b')\to \{\gamma:I\to \BB\B \mid \gamma(0)=\BB
b,\gamma(1)=\BB b'\}\subseteq \BB\B^{I}
$$
is a homotopy equivalence.
}

\vspace{0.15cm}For a bicategory $\B$ satisfying the conditions
above, we conclude the existence of a canonical homotopy equivalence
$$\BB\B(b,b) \simeq\Omega(\BB \B,\BB b)$$
between the loop space of the classifying space of the bicategory
with base point $\BB b$ and the classifying space of the category of
endomorphisms of $b$ in $\B$ (see Corollary \ref{corome}). This
result for $\B$ a 2-category should be attributed to Tillmann
~\cite[Lemma 3.3]{Tillmann1997}, but it has been independently
proven by both the first author ~\cite[Example 4.4]{Cegarra2011}
and by Del Hoyo ~\cite[Theorem 8.5]{Hoyo2012}.

Since any monoidal category can be regarded as a bicategory with
only one 0-cell, our results are applicable to them. Thus, any
diagram of monoidal functors and monoidal categories,
$(\N,\otimes)\overset{F}\to (\M,\otimes)\overset{\
F'}\leftarrow(\N',\otimes)$, gives rise to a {\em homotopy-fibre
product bicategory} $F\!\overset{_{\otimes}}\downarrow\!F'$, whose $0$-cells are the
objects $m\in \M$, whose 1-cells $(n,f,n'):m_0\to m_1$  consist of
objects $n\in \N$ and $n'\in \N'$, and a morphism $f: F'n'\otimes
m_0 \to m_1\otimes Fn$ in $\M$, and whose 2-cells
$(u,u'):(n,f,n')\Rightarrow (\bar{n},\bar{f},\bar{n}')$ are given by
a pair of morphisms, $u:n\to \bar{n}$ in $\N$ and $u':n'\to
\bar{n}'$ in $\N'$, such that $(1\otimes Fu)\cdot f=\bar{f}\cdot
(F'u'\otimes 1)$. In particular, for any monoidal functor $F$ as
above, we have the {\em homotopy-fibre bicategory} $F\!\overset{_{\otimes}}\downarrow
\!\text{I}$, where $\text{I}:([0],\otimes) \to (\M,\otimes)$ denotes
the monoidal functor from the trivial (one-arrow) monoidal category
$[0]$ to $\M$ that carries its unique object $0$ to the unit object
$\text{I}$ of the monoidal category $\M$. Then, our main conclusions
concerning monoidal categories, which are presented throughout
Section \ref{monoidalCategories}, are summarized as follows (see
Theorems \ref{maintheomon}, \ref{qmab}, and \ref{xbsmon}).

\vspace{0.15cm} $\bullet$ {\em The following properties on a
monoidal functor $F:(\N,\otimes)\to (\M,\otimes)$  are equivalent:

- For any monoidal functor  $F':(\N',\otimes)\to (\M,\otimes)$, the
canonical map $$\BB(F\!\overset{_{\otimes}}\downarrow \!F') \to
\BB(\N,\otimes)\times^{_{\mathrm h}}_
{\BB(\M,\otimes)}\BB(\N',\otimes)$$ is a  homotopy equivalence.

-  For any object $m\in \M $, the homomorphism
  $m\otimes -: F\!\overset{_{\otimes}}\downarrow\!\mathrm{I} \to
  F\!\overset{_{\otimes}}\downarrow\!\mathrm{I}$ induces a homotopy
  autoequivalence on $\BB(F\!\overset{_{\otimes}}\downarrow\!\mathrm{I})$.

- The canonical map $\BB(F\!\overset{_{\otimes}}\downarrow \!\mathrm{I})\to
\mathrm{Fib}(\BB F,\BB \mathrm{I})$ is a homotopy equivalence. }

 \vspace{0.15cm}$\bullet$ {\em The following properties on a monoidal category
$(\M,\otimes)$  are equivalent:

- For any diagram of  monoidal
functors$\xymatrix@C=14pt{(\N,\otimes)\ar[r]^-{F} & (\M,\otimes)
&(\N',\otimes)\ar[l]_(.4){F'}}$, the canonical
 map $\BB(F\!\overset{_{\otimes}}\downarrow \!F') \to
 \BB(\N,\otimes)\times^{_{\mathrm h}}_
{\BB(\M,\otimes)}\BB(\N',\otimes)$ is a  homotopy equivalence.

- For any object $m\in \M$, the functor $m\otimes-:\M \to
  \M$ induces a homotopy autoequivalence on the classifying
  space $\BB\M $.

- For any object $m\in \M$, the functor $-\otimes m:\M \to
  \M$ induces a homotopy autoequivalence on the classifying
  space $\BB\M$.

- The canonical map from the classifying space of the underlying
category into the loop space of the classifying space of the
monoidal category is a homotopy equivalence, $ \BB\M\simeq
\Omega\BB(\M,\otimes) $. }

\vspace{0.15cm}The equivalence between the two last statements in
the first result above might be considered as a version of Quillen's
Theorem B for monoidal functors. A monoidal version of Theorem A
follows: {\em If the homotopy-fibre bicategory of a monoidal functor
$F:(\N,\otimes)\to (\M,\otimes)$ is contractible, that is,
 $\BB(F\!\overset{_{\otimes}}\downarrow\!\mathrm{I})\simeq \mathrm{pt}$, then the
induced map $\BB F:\BB(\N,\otimes)\to \BB(\M,\otimes)$ is a homotopy
equivalence}. The equivalence of the three last statements in the
second one are essentially due to Stasheff ~\cite{Stasheff1963}.

Thanks to the  equivalence between the category of crossed modules
and the category of 2-groupoids, by Brown and Higgins ~\cite[Theorem
4.1]{BH1981equivalence}, our results on bicategories also find
application in the setting of crossed modules, what we do in Section
\ref{crossedModules}. Briefly, for any diagram of crossed modules
$(\GG,\PP,\partial)\overset{(\varphi,F)} \longrightarrow
(\HH,\QQ,\partial) \overset{(\varphi',F')}\longleftarrow
(\GG',\PP',\partial)$, we construct its {\em homotopy-fibre product
crossed module} $(\varphi,F)\!\downarrow\!(\varphi',F')$, and we
prove as the main result here (see Theorem \ref{mthpcm}) the
following:

 $\bullet$ {\em There is a canonical homotopy equivalence
$$\BB\big((\varphi,F)\!\downarrow\!(\varphi',F')\big) \simeq \BB(\GG,\PP,\partial)\times^{_{\mathrm h}}_
{\BB(\HH,\QQ,\partial)}\BB(\GG',\PP',\partial)$$ between the classifying
space of the homotopy-fibre product crossed module and the
homotopy-fibre product space of the induced maps $\BB
(\varphi,F):\BB(\GG,\PP,\partial)\to \BB(\HH,\QQ,\partial)$ and $\BB
(\varphi',F'):\BB(\GG',\PP',\partial)\to \BB(\HH,\QQ,\partial)$. }

 \noindent (Here, $(\GG,\PP,\partial)\mapsto
\BB(\GG,\PP,\partial)$ denotes the classifying space of crossed
modules functor by Brown and Higgins ~\cite{BH1991}.) Recalling that
the category of crossed complexes has a closed model structure, as
shown by Brown and Golasinki in ~\cite{BG1989}, we also prove that
the constructed homotopy-fibre product crossed module
$(\varphi,F)\!\downarrow\!(\varphi',F')$ occurs in a homotopy
pullback in this model category. More precisely, in Theorem
\ref{weakEquivalenceofHPB},  we prove that

 $\bullet$ {\em If one of the morphisms $(\varphi,F)$ or
$(\varphi',F')$ is a fibration, then
  the canonical morphism $$
(\GG,\PP,\partial)\times_{(\HH,\QQ,\partial)}(\GG',\PP',\partial)\to
(\varphi,F)\!\downarrow\!(\varphi',F'),$$ from the pullback crossed
module to the homotopy-fibre product crossed module  induces a
homotopy
  equivalence on classifying spaces. }

 The paper also includes some new results concerning classifying
spaces of bicategories, which are needed here to obtain the
aforementioned results on homotopy-fibre products.
 On the one hand, although in ~\cite[\S 4]{CCG2010} it was
proven that the classifying space construction is a functor from the
category of bicategories and homomorphisms to the category
$\mathbf{Top}$ of spaces, in this paper we need to extend that fact
as given below (see Lemma \ref{fact1}).

 $\bullet$ {\em  The assignment $\B\mapsto \BB\B$ is the
function on objects of two functors
$$\mathbf{Lax}\overset{\BB}\longrightarrow\mathbf{Top}\overset{\BB}\longleftarrow\mathbf{opLax},$$
where $\mathbf{Lax}$  is the category of bicategories  and lax
functors,  and $\mathbf{opLax}$  the category of bicategories
and oplax  functors.}

\noindent On the other hand, we also need to work with Duskin and
Street's {\em geometric nerves} of bicategories ~\cite{Duskin2002,
Street1996}. That is, with the simplicial sets
$\Delta^{\hspace{-2pt}\mathrm{u}}\B$, $\Delta\B$,
$\nabla_{\hspace{-2pt}{\mathrm{u}}}\B$, and $ \nabla\B$, whose
respective $p$-simplices are the normal lax, lax, normal oplax, and
oplax functors from the category $[p]=\{0<\cdots<p\}$ into the
bicategory $\B$. Although in  ~\cite[Theorem 6.1]{CCG2010}  the
existence of homotopy equivalences
$$
  |\Delta^{\hspace{-2pt}\mathrm{u}}\B|\,\simeq \,|\Delta\B|\,\simeq\,
\BB\B\,\simeq \,|\nabla \B|\,\simeq\,
|\nabla_{\hspace{-2pt}{\mathrm{u}}}\B|
$$
was proved, their natural behaviour is not studied. Then, in Lemma
\ref{facts2} we state the following:

 $\bullet$ {\em For any bicategory $\B$, the homotopy
equivalence $|\Delta^{\hspace{-2pt}\mathrm{u}}\B|\,\simeq
\,|\Delta\B|$ is natural on normal lax functors, the homotopy
equivalence $|\Delta\B|\,\simeq\, \BB\B$ is homotopy natural on lax
functors, the homotopy equivalence $\BB\B\,\simeq \,|\nabla \B|$ is
homotopy natural on oplax functors, and the homotopy equivalence
$|\nabla \B|\,\simeq\, |\nabla_{\hspace{-2pt}{\mathrm{u}}}\B|$ is
natural on normal oplax functors.}

 The proofs of these results are quite long and technical. Therefore, to avoid hampering
 the flow of the paper, we have put most
 of them into an appendix, comprising Section \ref{appendix}.

\section{Preparation: The constructions involved}
\label{Preliminaires} This section aims to make this paper as
self-contained as possible; therefore, at the same time as fixing
notations and terminology, we also review some necessary aspects and
results about homotopy pullbacks of topological spaces, comma
bicategories, and classifying spaces of small bicategories that are
used throughout the paper. However, some results, mainly those in
Lemmas \ref{simlem}, \ref{fact1}, and \ref{facts2}, are actually new.
For a detailed study of the definition of homotopy pullback of
continuous maps we refer the reader to Mather's original paper
~\cite{Mather1976} and to the more recent approach by Doeraene
~\cite{Doeraene1998}. For a general background on simplicial sets
and homotopy pullbacks in model categories, we recommend the books
by Goerss and Jardine ~\cite{GJ1999} and Hirschhorn
~\cite{Hirschhorn2009}. For a complete description of bicategories,
lax functors, and lax transformations, we refer the reader to the
papers by B\'enabou ~\cite{Benabou1967,Benabou1973} and Street
~\cite{Street1996}.

\subsection{Homotopy pullbacks.} Throughout this paper,
{\em all topological spaces have the homotopy type of CW-complexes},
so that a continuous map is a homotopy equivalence if and only if it
is a weak homotopy equivalence.

If $X\overset{f}\to B\overset{\ g}\leftarrow Y$ are continuous maps,
recall that its {\em homotopy-fibre product} is the space
$$X\times^{_\mathrm{h}}_BY=X\times_B B^I\!\times_B Y$$ consisting of
triples $(x,\gamma,y)$ with $x$ a point of $X$, $y$ a point of $Y$,
and $\gamma:I\to B$ a path of $B$ joining $f(x)$ and $g(y)$. This
space occurs in the so-called \emph{standard homotopy pullback} of
$f$ and $g$, that is, the homotopy commutative square
$$
\xymatrix@C=20pt@R=20pt{
X\times^{_\mathrm{h}}_BY \ar[r]^-{f'} \ar[d]_{g'} \ar@{}@<25pt>[d]|(.4){\overset{F}\Rightarrow}
& Y\ar[d]^{g}
\\ X\ar[r]^{f} & B}
$$
where $f'$ and $g'$ are the evident projection maps, and
$F:(X\times^{_\mathrm{h}}_BY)\times I\to B$ is the homotopy from
$fg'$ to $gf'$ given by $F(x,\gamma,y,t)=\gamma(t)$. In particular,
for any continuous map $g:Y\to B$ and any point $b\in B$, we have
the standard homotopy pullback
$$
\xymatrix@C=20pt@R=20pt{
\mathrm{Fib}(g,b) \ar[r]\ar[d] \ar@{}@<25pt>[d]|(.4){\overset{F}\Rightarrow}
&Y \ar[d]^{g}
\\ \mathrm{pt}\ar[r]^{b} & B,}
$$
where $\mathrm{Fib}(g,b)=\mathrm{pt}\times^{_\mathrm{h}}_BY$ is the
{\em homotopy-fibre} of $g$ over $b$. (We use pt to denote a
one-point space.) For any $y\in g^{-1}(b)$, one has the exact
homotopy sequence
$$
\cdots\to\pi_{n+1}(B,b)\to\pi_n(\mathrm{Fib}(g,b),(\mathrm{Ct}_b,y))\to \pi_n(Y,y)\to\pi_n(B,b)\to\cdots,
$$
from which $g$ is a homotopy equivalence if and only if all its
homotopy fibres are contractible.

More generally, following Mather's definition in ~\cite{Mather1976},
a homotopy commutative square
\begin{equation}\label{hps}
\begin{array}{c}
\xymatrix@C=20pt@R=20pt{
Z \ar[r]^{f'} \ar[d]_{g'} \ar@{}@<20pt>[d]|(.4){\overset{H}\Rightarrow}
& Y\ar[d]^{g}
\\ X\ar[r]^{f} & B,}
\end{array}
\end{equation}
where $H:fg'\Rightarrow gf'$ is a homotopy, is called a {\em
homotopy pullback} whenever the induced {\em whisker} map below is a
homotopy equivalence.$$w: Z\to
X\times^{_\mathrm{h}}_BY,\hspace{0.3cm} z\mapsto
(g'(z),H|_{z\times\!I},f'(z))$$

Throughout the paper, we use only basic well-known properties of
homotopy pullbacks.  For instance, the {\em homotopy-fibre
characterization of homotopy pullback squares}: The homotopy
commutative square $(\ref{hps})$
 is a homotopy pullback if and only if, for any point
$x\in X$, the composite square
$$ \xymatrix@C=20pt@R=20pt{
\mathrm{Fib}(g',x)\ar[r]\ar[d]\ar@{}@<22pt>[d]|{\Rightarrow}&Z
\ar[d]_{g'} \ar@{}@<18pt>[d]|{\Rightarrow}
\ar[r]^{f'}&Y\ar[d]^{g}\\
\mathrm{pt}\ar[r]^{x}&X\ar[r]^{f}&B}
$$
is a homotopy pullback. That is, if and only if the induced whisker
maps on homotopy fibres are homotopy equivalences, $ w:
\mathrm{Fib}(g',x) \overset{\sim\ }\to \mathrm{Fib}(g,f(x))$; or the
{\em two out of three property} of homotopy pullbacks: Let
$$
\xymatrix@C=20pt@R=20pt{
\bullet \ar[r] \ar[d]\ar@{}[rd]|{\Rightarrow}
& \bullet \ar[r] \ar[d] \ar@{}[rd]|{\Rightarrow}
& \bullet \ar[d]
\\ X' \ar[r] & X \ar[r] & \bullet}
$$
be a diagram of homotopy commutative squares. If the right square is
a homotopy pullback, then the left square is a homotopy pullback if
and only if the composite square is as well. If $\pi_0X'\to\pi_0X$
is onto and the left and composite squares are homotopy pullbacks,
then the right-hand square is a homotopy pullback.

Many other properties are easily deduced from the above ones. For
example, the square $(\ref{hps})$ is a homotopy pullback whenever
both maps $g$ and $g'$ are homotopy equivalences. If the square is a
homotopy pullback and the map
 $g$ is a homotopy equivalence, then so is $g'$. If the square is a homotopy pullback, $g'$
 is a homotopy equivalence, and the map $\pi_0X\to\pi_0B$ is surjective, then $g$
 is also a homotopy equivalence.

In ~\cite[Proposition 5.4 and Corollary 5.5]{CPS2006}, Chach\'olski,
Pitsch, and Scherer characterize continuous maps that always
produces homotopy pullback squares when one pulls back with them.
Along similar lines, we prove the needed lemma below for maps
induced on geometric realizations by simplicial maps. More
precisely, we characterize those simplicial maps $g:Y\to B$ such
that, for any simplicial map $f:X\to B$, the pullback square of
simplicial sets
\begin{equation}\label{pbsqs}\begin{array}{c}
\xymatrix@C=20pt@R=20pt{X\times_BY\ar[d]_{g'}\ar[r]^-{f'}&Y\ar[d]^{g}\\
X\ar[r]^-{f}&B}
\end{array}
\end{equation}
induces, by taking geometric realizations, a homotopy pullback
square of spaces.
 To do so, recall the
canonical homotopy colimit decomposition of a simplicial map, which
allows the source of the map to be written as the homotopy colimit
of its fibres over the simplices of the target: for a simplicial set
$B$, we can consider its category of simplices
$\Delta\!\downarrow\!B$ whose objects are the simplicial maps
$\Delta[n]\to B$ and whose morphisms are the obvious commutative
triangles. For a simplicial map $g:Y\to B$, we can then associate a
functor from $\Delta\!\downarrow\!B$ to the category of spaces by
mapping a simplex $x:\Delta[n]\to B$ to the geometric realization
$|g^{-1}(x)|$ of the simplicial set $g^{-1}(x)$ defined by the
pullback square
$$
\xymatrix@C=20pt@R=20pt{
g^{-1}(x)\ar[r]\ar[d]&Y\ar[d]^{g}\\ \Delta[n]\ar[r]^{x}&B.
}
$$
By ~\cite[Lemma IV.5.2]{GJ1999}, in the induced commutative diagram
of spaces,
$$
\xymatrix@C=20pt@R=20pt{ \underset{x:\Delta[n]\to B}{\mathrm{hocolim}}|g^{-1}(x)|\ar[d]
\ar@{}[rd]|{(a)}
\ar[r]^-{\sim}&|Y|\ar[d]^{|g|}\\
\underset{x:\Delta[n]\to B}{\mathrm{hocolim}}|\Delta[n]|\ar[r]^-{\sim}
&|B|
}
$$
the horizontal maps are both homotopy equivalences.

\begin{lemma} \label{simlem} For any given simplicial map  $g:Y\to B$, the following
statements are equivalent:

\vspace{0.2cm} $(i)$ For any simplex of $B$,  $x:\Delta[n]\to B$,
and for any simplicial map $\sigma:\Delta[m]\to \Delta[n]$, the
induced map
 $|g^{-1}(x\sigma)|\to |g^{-1}(x)|$
is a homotopy equivalence.

\vspace{0.2cm} $(ii)$ For any simplex $x:\Delta[n]\to B$, the
induced pullback square of spaces
$$
\xymatrix{|g^{-1}(x)|\ar[d]\ar[r]&|Y|\ar[d]^{|g|}\\
|\Delta[n]|\ar[r]^-{|x|}&|B|}
$$
 is a homotopy pullback.

$(iii)$ For any simplicial map $f:X\to B$, the pullback square of
spaces
$$
\xymatrix@C=20pt@R=25pt{|X\times_BY|\ar[d]_{|g'|}\ar[r]^-{|f'|}&|Y|\ar[d]^{|g|}\\
|X|\ar[r]^-{|f|}&|B|,}
$$
induced by $(\ref{pbsqs})$, is a homotopy pullback.
\end{lemma}
\begin{proof}
$(i)\Rightarrow (ii)$: Let $x:\Delta[n]\to B$ be any simplex of $B$.
We have the diagram
$$
\xymatrix@C=20pt@R=20pt{|g^{-1}(x)|\ar[r]\ar[d]\ar@{}[rd]|{(b)}&
\underset{x:\Delta[n]\to B}{\mathrm{hocolim}}|g^{-1}(x)|\ar[d]\ar@{}[rd]|{(a)}
\ar[r]^-{\sim}&|Y|\ar[d]^{|g|}\\
|\Delta[n]|\ar[r]^-{|x|}\ar[d]_{\wr}\ar@{}[rd]|{(c)}&
\underset{x:\Delta[n]\to B}{\mathrm{hocolim}}|\Delta[n]|\ar[r]^-{\sim}\ar[d]^{\wr}
&|B|\\
\mathrm{pt}\ar[r]^{x}&\underset{x:\Delta[n]\to B}{\mathrm{hocolim}}\,\mathrm{pt},
&
}
$$
where $\underset{x:\Delta[n]\to B}{\mathrm{hocolim}}\,\mathrm{pt}=
\BB(\Delta\!\downarrow\!B)$ is the classifying space of the simplex
category. Since, by Quillen's Lemma ~\cite[page 14]{Quillen1973},
the composite square $(b)+(c)$ is a homotopy pullback, it follows
that $(b)$ is a homotopy pullback. Therefore,  the composite
$(b)+(a)$ is as well.

$(ii)\Rightarrow (i)$:  For any simplicial map
$\sigma:\Delta[m]\to\Delta[n]$ and any simplex $x:\Delta[n]\to B$,
 the right side and the large square in the diagram of spaces
$$
\xymatrix@C=20pt@R=20pt{|g^{-1}(x\sigma)|\ar[r]\ar[d]&|g^{-1}(x)|\ar[r]\ar[d]&|Y|\ar[d]^{|g|}\\
|\Delta[m]|\ar[r]^{|\sigma|}&|\Delta[n]|\ar[r]^{|x|}&|B|
}
$$
are both homotopy pullback, and therefore so is the left-hand one.
 As  $|\Delta[m]|$ and $|\Delta[n]|$ are both contractible, the map
$|\sigma|$ is a homotopy equivalence, and therefore the map
$|g^{-1}(x\sigma)|\to |g^{-1}(x)|$ is a homotopy equivalence.

$(i)\Rightarrow (iii)$: Suppose we have the pullback square of
simplicial sets $(\ref{pbsqs})$. Then,   for any simplex
$x:\Delta[n]\to X$ of $X$, we have a natural isomorphism of fibres
$g'^{-1}(x)\cong g^{-1}(fx)$, and it follows that the map $g'$ also
satisfies the same condition $(i)$ as $g$ does. Then, by the already
proven part $(i)\Leftrightarrow (ii)$, we know that, for any vertex
$x:\Delta[0]\to X$, both the left side and the composite square in
the diagram
$$
\xymatrix@C=20pt@R=20pt{|g'^{-1}(x)|\cong|g^{-1}(fx)|\ar[r]\ar[d]&|X\times_BY|\ar[r]\ar[d]_{|g'|}&|Y|\ar[d]^{|g|}\\
\mathrm{pt}=|\Delta[0]|\ar[r]^{|x|}&|X|\ar[r]^{|f|}&|B|
}
$$
are homotopy pullbacks. Therefore, from the diagram on whisker maps
$$
\xymatrix{|g'^{-1}(x)|\ar[r]^{\sim}\ar[d]_{\wr}&\mathrm{Fib}(|g'|,|x|)\ar[d]^{w}\\
|g^{-1}(fx)|\ar[r]^{\sim}&\mathrm{Fib}(|g|,|fx|),}
$$
we conclude that the map $\mathrm{Fib}(|g'|,|x|)\to
\mathrm{Fib}(|g|,|fx|)$ is a homotopy equivalence. Since the
homotopy fibres of any map over points connected by a path are
homotopy equivalent, and any point of $|X|$ is path-connected with a
0-cell $|x|$ defined by some 0-simplex $x:\Delta[0]\to X$ as above,
the result follows from the homotopy fibre characterization.

$(iii)\Rightarrow (ii)$: This is obvious.
\end{proof}

\subsection{Some bicategorical conventions.} For bicategories, we use the same conventions and
notations as Carrasco, Cegarra, and Garz\'on in
~\cite[\S2.4]{CCG2010} and ~\cite[\S 2.1]{CCG2011}. Given any
bicategory $\B$, its set of {\em objects} or {\em $0$-cells} is
denoted by $\Ob\B$. For each ordered pair of objects $(b_0,b_1)$ of
$\B$, $\B(b_0,b_1)$ denotes its hom-category whose objects $f:b_0\to
b_1$ are called the {\em $1$-cells} in $\B$ with source $b_0$ and
target $b_1$, and whose morphisms ${\beta:f\Rightarrow g}$ are
called 2-{\em cells} of $\B$. The composition in each hom-category
$\B(b_0,b_1)$, that is, the {\em
  vertical composition} of 2-cells, is denoted by the symbol
$``\cdot"$, while the symbol $``\circ"$ is used to denote the {\em
horizontal composition} functors:
$$
\xymatrix@C=30pt{ b_0 \ar@/^1.2pc/[r]^{f}
  \ar@/_1.2pc/[r]_{h}\ar[r]|(.6){g} \ar@{}@<6pt>[r]|{\Downarrow\beta}
  \ar@{}@<-6pt>[r]|{\Downarrow\gamma} &b_1 }\ \overset{\cdot}\mapsto \
\xymatrix{ b_0 \ar@/^0.8pc/[r]^{f} \ar@/_0.8pc/[r]_{h}
  \ar@{}[r]|{\Downarrow \gamma\cdot\beta} &b_1, }\hspace{0.3cm}
\xymatrix @C=20pt {b_0 \ar@/^0.7pc/[r]^{f_1} \ar@/_0.7pc/[r]_{g_1}
  \ar@{}[r]|{\Downarrow\beta_1} &b_1 \ar@/^0.7pc/[r]^{f_2}
  \ar@/_0.7pc/[r]_{g_2} \ar@{}[r]|{\Downarrow\beta_2} &b_2}\
\overset{\circ}\mapsto\ \xymatrix @C=30pt {b_0
  \ar@/^0.8pc/[r]^{f_2\circ f_1} \ar@/_0.8pc/[r]_{g_2\circ g_1}
  \ar@{}[r]|{\Downarrow\beta_2\circ \beta_1} &b_2.}
$$
Identities are denoted as $1_f:f\Rightarrow f$, for any 1-cell $f$,
and $1_b:b\to b$, for any $0$-cell $b$.  The {\em associativity
  constraints} of the bicategory are denoted by
$$\aso_{f_3,f_2,f_1}:(f_3\circ f_2)\circ f_1
\cong f_3\circ (f_2\circ f_1),
$$
which are natural in $(f_3,f_2,f_1)\in
\B(b_2,b_3)\times\B(b_1,b_2)\times\B(b_0,b_1)$. The {\em left and
  right unity constraints} are denoted by
$\bl_f: 1_{b_1}\circ f \cong f$ and $\br_f:f\circ 1_{b_0}\cong f$.
These are natural in $f\in \B(b_0,b_1)$. These constraint 2-cells
must satisfy the well-known pentagon and triangle coherence
conditions.

A bicategory in which all the constraints are identities is a 2-{\em
  category}. It is the same as a category enriched in the category
$\Cat$ of small categories. As each category $\B$ can be considered
as a 2-category in which all deformations are identities, that is,
in which each category $\B(b_0,b_1)$ is discrete, several times
throughout the paper, categories are considered as special
bicategories.

A {\em lax functor} is written as a pair $ F=(F,\widehat{F}):\B \to
\C$, since we will generally denote its structure constraints by
$$\widehat{F}_{f_2,f_1}:Ff_2\circ Ff_1\Rightarrow F(f_2\circ
f_1),\hspace{0.4cm} \widehat{F}_b:1_{Fb}\Rightarrow F1_b,$$ for each
pair of composable 1-cells, and each object of $\B$. Recall that the
structure 2-cells $\widehat{F}_{f_2,f_1}$ are natural in
$(f_2,f_1)\in\B(b_1,b_2)\times\B(b_0,b_1)$ and they satisfy the
usual coherence conditions. Replacing the constraint 2-cells above
by $\widehat{F}_{f_2,f_1}:F(f_2\circ f_1)\Rightarrow Ff_2\circ Ff_1$
and $\widehat{F}_b: F(1_b)\Rightarrow 1_{Fb}$, we have the notion of
{\em
  oplax functor} $F=(F,\widehat{F}): \B\to \C$.  Any lax or oplax
functor $F$ is termed a {\em pseudo-functor} or {\em homomorphism}
whenever all the structure constraints $\widehat{F}_{f_2,f_1}$ and
$\widehat{F}_b$ are invertible. When these 2-cells are all
identities, then $F$ is called a 2-{\em functor}. If all the unit
constraints $\widehat{F}_b$ are identities, then the lax or oplax
functor is qualified as (strictly) {\em unitary} or {\em normal}.

If $F,F':\B \to \C$ are lax functors, then a {\em lax
transformation} ${\alpha=(\alpha,\widehat{\alpha}):F\Rightarrow F'}$
consists of morphisms ${\alpha b:Fb\to F'b}$, $b\in \mbox{Ob}\B$,
and 2-cells
$$
\xymatrix@C=20pt@R=20pt{Fb_0\ar[d]_{\alpha b_0}\ar[r]^{Ff}
\ar@{}@<25pt>[d]|(.38){\widehat{\alpha}_f}|(.55){\Rightarrow}
&Fb_1\ar[d]^{\alpha b_1}\\
F'b_0\ar[r]_{F'\!f}&F'b_1}
$$
which are natural on the 1-cells $f:b_0\to b_1$ of $\B$, subject to
the usual coherence axioms.  Replacing the structure deformation
above by $\widehat{\alpha}_f: \alpha b_1\circ Ff\Rightarrow
F'\!f\circ \alpha b_0$, we have the notion of {\em oplax
transformation} $\alpha:F\Rightarrow F'$. Any lax or oplax
transformation $\alpha$ is termed a {\em pseudo-transformation}
whenever all the naturality 2-cells $\widehat{\alpha}_{f}$ are
invertible.  Similarly, we have the notions of  lax, oplax, and
pseudo transformation between oplax functors.

\subsection{Homotopy pullback bicategories.}
\label{HPB} We present a bicategorical comma construction in some
detail, since it is fundamental for the results of this paper.
However, we are not claiming much originality since variations of
the quite ubiquitous `comma category' construction have been
considered (just to define `homotopy pullbacks') in many general
frameworks of enriched categories (where a homotopy theory has been
established); see for instance Grandis ~\cite{Gran94}.

Let $\xymatrix@C=18pt{\A\ar[r]^-{F} & \B &\A'\ar[l]_(.4){F'}}$ be a
diagram where $\A$, $\B$, and $\A'$ are bicategories, $F$ is a lax
functor, and $F'$ is an oplax functor. The ``{\em homotopy pullback
bicategory}"
\begin{equation}\label{ff'}F\!\downarrow\!F'\end{equation} is defined
as follows:

\vspace{0.2cm} $\bullet$ {\em The $0$-cells} of $F\!\downarrow\!F'$
are triples $(a,f,a')$ with $a$ a $0$-cell of $\A$, $a'$ a $0$-cell
of $\A'$, and $f:Fa\to F'a'$ a $1$-cell in $\B$.

\vspace{0.2cm} $\bullet$ {\em A $1$-cell}
$(u,\beta,u'):(a_0,f_0,a'_0)\to (a_1,f_1,a'_1)$ of
$F\!\downarrow\!F'$ consists of a $1$-cell $u:a_0\to a_1$ in $\A$, a
$1$-cell $u':a'_0\to a'_1$ in $\A'$, and $2$-cell $\beta: F'u'\circ
f_0 \Rightarrow f_1\circ Fu$  in $\B$,
$$
\xymatrix@C=20pt@R=20pt{
Fa_0\ar[r]^{Fu}\ar[d]_{f_0}\ar@{}@<25pt>[d]|(.5){\Rightarrow}|(.35){\beta}
  & Fa_1\ar[d]^{f_1}
\\ F'a'_0\ar[r]^{F'\!u'} & F'a'_1.}
$$

$\bullet$ {\em A $2$-cell} in $F\!\downarrow\!F'$,
$\xymatrix@C=30pt{
(a_0,f_0,a'_0)\ar@/^1pc/[r]^{(u,\beta,u')}\ar@{}[r]|(.5){\Downarrow
(\alpha,\alpha')}
\ar@/_1pc/[r]_{(\bar{u},\bar{\beta},\bar{u}')}&(a_1,f_1,a'_1),} $ is
given by a $2$-cell $\alpha:u\Rightarrow \bar{u}$ in $\A$ and a
$2$-cell $\alpha':u'\Rightarrow \bar{u}'$ in $\A'$ such that the
 diagram below commutes.
$$
\xymatrix@C=35pt{F'u'\circ f_0\ar@2[r]^{F'\!\alpha'\circ 1}\ar@2[d]_{\beta}
  & F'\bar{u}'\circ f_0\ar@2[d]^{\bar{\beta}}
\\ f_1\circ Fu\ar@2[r]^{1\circ\, F\!\alpha} & f_1\circ F\bar{u}}
$$

$\bullet$ {\em The vertical composition of $2$-cells} in
$F\!\downarrow \!F'$ is  induced by the vertical composition laws in
$\A$ and $\A'$, thus $ (\bar{\alpha},\bar{\alpha}')\cdot
(\alpha,\alpha')= (\bar{\alpha}\cdot
\alpha,\bar{\alpha}'\cdot\alpha')$. The identity at a $1$-cell is
given by $1_{(u,\beta,u')}=(1_u,1_{u'})$.

\vspace{0.2cm} $\bullet$ {\em  The horizontal composition of two
$1$-cells} in $F\!\downarrow \!F'$,
\begin{equation}\label{tc1cs}
  \xymatrix@C=40pt{
(a_0,f_0,a'_0)\ar[r]^{(u_1,\beta_1,u'_1)}
&(a_1,f_1,a'_1)\ar[r]^{(u_2,\beta_2,u'_2)}
& (a_2,f_2,a'_2)},
\end{equation}
is the 1-cell $ (u_2,\beta_2,u'_2)\circ (u_1,\beta_1,u'_1) =
(u_2\circ u_1, \beta_2\circledcirc \beta_1,u'_2\circ u'_1 ) $, where
$\beta_2\circledcirc \beta_1$ is the 2-cell pasted of the diagram in
$\B$
\begin{equation}\begin{array}{c}
\xymatrix@C=30pt{
\ar@{}@<-50pt>[d]|{\textstyle \beta_2\circledcirc \beta_1 =}
Fa_0\ar[r]^{Fu_1}\ar[d]_{f_0}
               \ar@{}@<25pt>[d]|(.3){\beta_1}|(.45){\Rightarrow}
\ar@{}@<14pt>[rr]|{\widehat{F}\,\Uparrow}
\ar@/^22pt/[rr]^{F(u_2\circ u_1)}
&Fa_1\ar[r]^{Fu_2}\ar[d]_{f_1}
\ar@{}@<25pt>[d]|(.3){\beta_2}|(.45){\Rightarrow}
&Fa_2\ar[d]^{f_2}\\
F'\!a'_0\ar[r]^{F'\!u'_1}
\ar@{}@<-12pt>[rr]|{\widehat{F}'\Uparrow}
\ar@/_22pt/[rr]_{F'(u'_2\circ u'_1)}
&F'\!a'_1\ar[r]^{F'\!u'_2}&F'\!a'_2,
}
\end{array}
\end{equation}
$
\begin{array}{cl}\text{that is, }
\beta_2\circledcirc \beta_1=\Big(&\hspace{-0.3cm}F'(u'_2\circ u'_1)\circ f_0
\overset{\widehat{F}'\!\circ 1}\Longrightarrow
 (F'u'_2\circ F'u'_1)\circ f_0 \overset{\aso}\Longrightarrow
 F'u'_2\circ (F'u'_1\circ f_0) \overset{1\circ \beta_1}\Longrightarrow\\[6pt]
 &\hspace{-0.3cm}F'u'_2\circ (f_1\circ Fu_1)\overset{\aso^{-1}}\Longrightarrow
 (F'u'_2\circ f_1)\circ Fu_1  \overset{\beta_2\circ 1}\Longrightarrow
 (f_2\circ Fu_2)\circ Fu_1  \overset{\aso}\Longrightarrow\\[6pt]
 &\hspace{-0.3cm}f_2\circ (Fu_2\circ Fu_1)
\overset{1\circ \widehat{F}}\Longrightarrow
   f_2\circ F(u_2\circ u_1)\Big).
\end{array}
$

$\bullet$ {\em The horizontal composition of $2$-cells} in
$F\!\downarrow \!F'$ is given by composing horizontally the 2-cells
in $\A$ and $\A'$, thus $ (\alpha_2,\alpha'_2)\circ
(\alpha_1,\alpha'_1)= (\alpha_2\circ
\alpha_1,\alpha'_2\circ\alpha'_1) $.

$\bullet$ {\em The identity $1$-cell} in $F\!\downarrow \!F'$, at an
object $(a,f,a')$, is $(1_a,\overset{_\circ}{1}_{(a,f,a')},1_{a'})
$, where $\overset{_\circ}{1}_{(a,f,a')}$ is the 2-cell in $\B$
obtained by pasting the diagram
$$
\xymatrix@C=30pt{
\ar@{}@<-50pt>[d]|{\textstyle \overset{_\circ}{1}_{(a,f,a')} =}
Fa\ar[rr]^(.3){1_{Fa}}\ar[d]_{f}
\ar@{}@<14pt>[rr]|{\widehat{F}\,\Uparrow}
\ar@/^22pt/[rr]^{F1_a}
&\ar@{}[d]|(.4){\br^{-1}\cdot \bl}|(.55){\cong}
&Fa\ar[d]^{f}\\
F'\!a'\ar[rr]_(.3){1_{F'\!a'}}
\ar@{}@<-12pt>[rr]|{\widehat{F}'\Uparrow}
\ar@/_22pt/[rr]_{F'1_{a'}}
&&F'\!a',
}
$$
that is, $ \overset{_\circ}{1}_{(a,f,a')}=\Big(F'1_{a'}\circ
f\overset{\widehat{F}'\!\circ 1}\Longrightarrow 1_{F'\!a'}\circ
f\overset{\bl}\Longrightarrow f \overset{\br^{-1}}\Longrightarrow
f\circ 1_{Fa}\overset{1\circ \widehat{F}}\Longrightarrow f\circ
F1_a\Big)$.

$\bullet$ {\em The associativity, right and left unit constraints}
of the bicategory  $F\!\downarrow \!F'$ are provided by those of
$\A$ and $\A'$ by the formulas
$$
\aso_{(u_3,\beta_3,u'_3),(u_2,\beta_2,u'_2),(u_1,\beta_1,u'_1)}\!=
\!(\aso_{u_3,u_2,u_1},\aso_{u'_3,u'_2,u'_1}),\ \
\bl_{(u,\beta,u')}\!=\!(\bl_u,\bl_{u'}),\ \ \br_{(u,\beta,u')}\!=\!(\br_u,\br_{u'}).
$$

\subsubsection{The main square.} There is a  (non-commutative!) square, which is of fundamental
interest for the discussions below:
\begin{equation}\label{pqsquare}\begin{array}{c}
    \xymatrix@C=20pt@R=20pt{
      F\!\downarrow\!F'\ar[r]^-{P'}\ar[d]_{P}
&\A'\ar[d]^{F'}\\ \A\ar[r]^{F}&\B}
\end{array}
\end{equation}
where $P$ and $P'$ are projection 2-functors, which act on cells of
$F\!\downarrow\!F'$ by
\begin{equation}\label{PP'}
\xymatrix{
a_0\ar@/^8pt/[r]^{u}\ar@/_8pt/[r]_{\bar{u}}\ar@{}[r]|{\Downarrow\alpha}&a_1}
\xymatrix@C=12pt{&\ar@{|->}[l]_{P}}
 \xymatrix@C=3pc{
(a_0,f_0,a'_0)\ar@/^12pt/[r]^{(u,\beta,u')}
             \ar@/_12pt/[r]_{(\bar{u},\bar{\beta},\bar{u}')}
            \ar@{}[r]|{\Downarrow(\alpha,\alpha')}
&(a_1,f_1,a'_1)}\xymatrix@C=12pt{\ar@{|->}[r]^{P'}&}
\xymatrix{
a'_0\ar@/^8pt/[r]^{u'}\ar@/_8pt/[r]_{\bar{u}'}\ar@{}[r]|{\Downarrow\alpha'}
&a'_1.}
\end{equation}

\subsubsection{Two pullback squares.} We consider here three  particular cases of the
    construction $(\ref{ff'})$:
\begin{itemize}

\item[-]  For any lax functor $F:\A\to \B$, the bicategory
    $F\!\downarrow\!\B:= F\!\downarrow\!1_\B.$

\vspace{0.2cm}
\item[-] For any oplax functor $F':\A'\to \B$, the bicategory
    $\B\!\downarrow\!F':= 1_\B\!\downarrow\!F'$.

\vspace{0.2cm}
\item[-]  For any bicategory $\B$, the
    bicategory $\B\!\downarrow\!\B:=1_\B\!\downarrow\!1_\B$.
\end{itemize}

There are commutative squares
\begin{equation}\label{pullbackSquares}\begin{array}{cc}
\xymatrix@C=20pt@R=20pt{
F\!\downarrow\!F'\ar[r]^-{\bar{F}}\ar[d]_{P} &\B\!\downarrow\!F'\ar[d]^{P}\\
\A\ar[r]^{F}&\B,}
&
\xymatrix@C=20pt@R=20pt{F\!\downarrow\!F'\ar[r]^-{P'}\ar[d]_{\bar{F}'} &\A'\ar[d]^{F'}\\
 F\!\downarrow\!\B\ar[r]^{P'}&\B,}
\end{array}
\end{equation}
where the first one is in the category of bicategories and lax
functors, and the second one in the category of oplax functors. The
lax functor $\bar{F}: F\!\downarrow\!F'\to \B\!\downarrow\!F'$ in
the first square is given on cells by applying $F$ to the first
components
$$
\xymatrix@C=3pc{
(a_0,f_0,a'_0)\ar@/^12pt/[r]^{(u,\beta,u')}
             \ar@/_12pt/[r]_{({\bar{u}},\bar{\beta},\bar{u}')}
              \ar@{}[r]|{\Downarrow(\alpha,\alpha')}
  &(a_1,f_1,a'_1)}\overset{\bar{F}}\mapsto
\xymatrix@C=3pc{
(Fa_0,f_0,a'_0)\ar@/^12pt/[r]^{(Fu,\beta,u')}
               \ar@/_12pt/[r]_{(F\bar{u},\bar{\beta},\bar{u}')}
               \ar@{}[r]|{\Downarrow(F\alpha,\alpha')}
  &(Fa_1,f_1,a'_1),}
$$
while the oplax functor $\bar{F}':F\!\downarrow\!F'\to
F\!\downarrow\!\B$ in the second one acts on cells through the
application of $F'$ to the last components
$$
\xymatrix@C=3pc{
(a_0,f_0,a'_0)\ar@/^12pt/[r]^{(u,\beta,u')}
             \ar@/_12pt/[r]_{(\bar{u},\bar{\beta},\bar{u}')}
  \ar@{}[r]|{\Downarrow(\alpha,\alpha')}
  &(a_1,f_1,a'_1)}\overset{\bar{F}'}\mapsto
\xymatrix@C=3pc{
(a_0,f_0,F'\!a'_0)\ar@/^12pt/[r]^{(u,\beta,F'\!u')}
  \ar@/_12pt/[r]_{(\bar{u},\bar{\beta},F'\!\bar{u}')}
   \ar@{}[r]|{\Downarrow(\alpha,F'\!\alpha')}
  &(a_1,f_1,F'\!a'_1).}
$$
At any pair of composable 1-cells in $F\!\downarrow\!F'$ as in
$(\ref{tc1cs})$, their respective structure constraints for the
composition are the 2-cells
$$\begin{array}{l}
(\widehat{F}_{u_2,u_1}, 1_{u'_2\circ u'_1}):
\bar{F}(u_2,\beta_2,u'_2)\circ \bar{F}(u_1,\beta_1,u'_1)
\Rightarrow \bar{F}\big((u_2,\beta_2,u'_2)
\circ (u_1,\beta_1,u'_1)\big),\\[5pt]
(1_{u_2\circ u_1},\widehat{F}'_{u'_2,u'_1}):
\bar{F}'\big((u_2,\beta_2,u'_2)\circ (u_1,\beta_1,u'_1)\big)
\Rightarrow \bar{F}'(u_2,\beta_2,u'_2)
\circ\bar{F}'(u_1,\beta_1,u'_1),
\end{array}
$$
and, at any object $(a,f,a')$ of $F\!\downarrow\!F'$, their
respective constraints for the identity are
$$
(\widehat{F}_a,1_{1_{a'}}): 1_{\bar{F}(a,f,a')}\Rightarrow
\bar{F}1_{(a,f,a')} ,
\hspace{0.4cm}
(1_{1_a},\widehat{F}'_{a'}):\bar{F}'1_{(a,f,a')}\Rightarrow 1_{\bar{F}'(a,f,a')}.
$$

Although neither the category of bicategories and lax functors nor
the category of bicategories and oplax functors have pullbacks in
general, the following fact holds.
\begin{lemma}\label{pullbackLemma}
  $(i)$ The first square in $(\ref{pullbackSquares})$ is a pullback in
  the category of bicategories and lax functors.

  $(ii)$ The second square in $(\ref{pullbackSquares})$ is a pullback
  in the category of bicategories and oplax functors.
\end{lemma}
\begin{proof} $(i)$
Any pair of lax functors $L:\D \to \A$ and $M:\D \to
\B\!\downarrow\!F'$ such that $FL=PM$ determines a unique oplax
functor $N:\D\to F\!\downarrow\!F'$
$$
\xymatrix{
\D\ar@/^5pt/[rrd]^{M}\ar@/_5pt/[ddr]_{L}\ar@{.>}[rd]
                 \ar@{}@<-3pt>[rd]^{N}
\\&F\!\downarrow\!F'\ar[r]^{\bar{F}}\ar[d]^{P}
& \B\!\downarrow\!F'\ar[d]^{P}
\\ &\A\ar[r]^{F} & \B}
$$
such that $PN=L$ and $\bar{F}N=M$, which is defined as follows: The
lax functor $M$ carries any object $d\in\Ob\D$ to an object of
$\B\!\downarrow\!F$ which is necessarily written in the form
$M(d)=(FL(d),f(d),a'(d))$, with $a'(d)$ an object of $\A'$ and
 $f(d):FL(d)\to F'a'(d)$ a 1-cell in $\B$. Similarly, for any
1-cell $h:d_0\to d_1$ in $\D$, we have $M(h)=(FL(h),
\beta(h),u'(h))$ for some 1-cell $u'(h):a'(d_0)\to a'(d_1)$ in $\A'$
and some 2-cell
$$
\xymatrix{
FL(d_0)\ar[r]^{FL(h)}\ar[d]_{f(d_0)}
      \ar@{}@<36pt>[d]|(.57){\Rightarrow}|(.42){\beta(h)}
  & FL(d_1)\ar[d]^{f(d_1)}
\\ F'a'(d_0)\ar[r]_{F'\!u'(h)} & F'a'(d_1),}
$$
in $\B$, and for any 2-cell $\gamma:h_0\Rightarrow h_1$ in $\D$, we
have $M(\gamma)=(FL(\gamma),\alpha'(\gamma))$, for some 2-cell
$\alpha'(\gamma):u'(h_0)\Rightarrow u'(h_1)$ in $\A'$. Also, for any
object $d$ and any pair of composable 1-cells $h_1:d_0\to d_1$ and
$h_2:d_1\to d_2$ in $\D$, the attached structure 2-cells of $M$ can
be respectively written in a similar form as
$$\begin{array}{l}
\widehat{M}_d=
(F(\widehat{L}_d)\cdot \widehat{F}_{L(d)},\widehat{\alpha}'_d): 1_{M(d)}
\Rightarrow M(1_d) ,\\[4pt]
\widehat{M}_{h_2,h_1}=(F(\widehat{L}_{h_2,h_1})\cdot\widehat{F}_{L(h_2),L(h_1)},
\widehat{\alpha}'_{h_2,h_1}):
M(h_2)\circ M(h_1)\Rightarrow M(h_2\circ h_1),
\end{array}
$$
for some 2-cells $\widehat{\alpha}'_{h_2,h_1}$ and
$\widehat{\alpha}'_{d}$ in $\A'$. Then, the claimed $N:\D\to
F\!\downarrow\!F'$ is the lax functor which acts on cells by
$$
\xymatrix@C=2pc{
d_0\ar@/^10pt/[r]^{h}\ar@/_10pt/[r]_{\bar{h}}\ar@{}[r]|{\Downarrow\gamma}
&d_1} \ \overset{N}\mapsto\
\xymatrix@C=5pc{
(L(d_0),f(d_0),a'(d_0))\ar@/^14pt/[r]^{(L(h),\beta(h),u'(h))}
  \ar@/_14pt/[r]_{(L(\bar{h}),\beta(\bar{h}),u'(\bar{h}))}
\ar@{}[r]|{\Downarrow(L(\gamma),\alpha'(\gamma))}&
(L(d_1),f(d_1),a'(d_1))}
$$
and its respective structure 2-cells, for any object $d$  and any
pair of composable 1-cells $h_1:d_0\to d_1$ and $h_2:d_1\to d_2$ in
$\D$, are
$$\begin{array}{ll}
\widehat{N}_d=( \widehat{L}_{d},\widehat{\alpha}'_d): 1_{N(d)}
\Rightarrow N(1_d) ,&
\widehat{N}_{h_2,h_1}=(\widehat{L}_{h_2,h_1},\widehat{\alpha}'_{h_2,h_1}):
N(h_2)\circ N(h_1)\Rightarrow N(h_2\circ h_1).
\end{array}
$$

The proof of $(ii)$ is parallel to that given above for part $(i)$,
and it is left to the reader.
\end{proof}

\subsection{The homotopy-fibre bicategories.}
 For any 0-cell $b\in \B$, we also denote by $b:[0]\to
    \B$ the normal homomorphism such that $b(0)=b$, and whose structure isomorphism is $\bl:1_b\otimes 1_b\cong 1_b$. Then, we have the bicategories
\begin{itemize}
\item[-]  $F\!\downarrow\!b$, for any lax functor $F:\A\to \B$.

\vspace{0.2cm}
\item[-] $b\!\downarrow\!F'$, for any oplax functor
    $F':\A'\to \B$.

\vspace{0.2cm}
\item[-] $b\!\downarrow\!\B:=b\downarrow\!1_\B$, and
    $\B\!\downarrow\!b:=1_\B\!\downarrow\!b $.

    \end{itemize}

Given $F$ and $F'$ as above, any $1$-cell $p:b_0\to b_1$ in $\B$
determines $2$-functors
\begin{equation}\label{ps}
p_*:F\!\downarrow\! b_0\to F\!\downarrow\!b_1, \hspace{0.4cm}
 p^*:b_1\!\downarrow\!F'\to b_0\!\downarrow\!F',
\end{equation}
respectively given on cells by
$$
\xymatrix@C=2.3pc{
(a_0,f_0)\ar@/^11pt/[r]^{(u,\beta)}\ar@/_11pt/[r]_{(\bar{u},\bar{\beta})}
\ar@{}[r]|{\Downarrow \alpha}
&(a_1,f_1)}
\xymatrix@C=16pt{\ar@{|->}[r]^{p_*}&}
\xymatrix@C=2pc{(a_0,p\circ f_0)\ar@/^11pt/[r]^{(u, p\odot\beta)}
\ar@/_11pt/[r]_{(\bar{u},p\odot \bar{\beta})}
\ar@{}[r]|{\Downarrow \alpha}&
(a_1,p\circ f_1),}
$$
$$
\xymatrix@C=2.3pc{
(f_0,a'_0)\ar@/^11pt/[r]^{(\beta,u')}
         \ar@/_11pt/[r]_{(\bar{\beta},\bar{u}')}
\ar@{}[r]|{\Downarrow \alpha'}
&(f_1,a'_1)}
\xymatrix@C=16pt{\ar@{|->}[r]^{p^*}&}
\xymatrix@C=2pc{
(f_0\circ p,a'_0)\ar@/^11pt/[r]^{(\beta\odot p, u')}
\ar@/_11pt/[r]_{(\bar{\beta}\odot p,\bar{u}')}
\ar@{}[r]|{\Downarrow \alpha'}&
(f_1\circ p, a'_1),}
$$
where, for any $(u,\beta):(a_0,f_0)\to (a_1,f_1)$ in
$F\!\downarrow\!b_0$ and $(\beta,u'):(f_0,a'_0)\to (f_1,a'_1)$ in
$b_1\!\downarrow\!F'$, the 2-cells $p\odot\beta$ and $\beta \odot p$
are respectively obtained by pasting the diagrams $$\begin{array}{c}
\xymatrix{Fa_0 \ar@{}@<-33pt>[dd]|{\textstyle p\odot\beta:}
\ar@{}@<25pt>[d]|(.35){\beta}|(.5){\Rightarrow}
\ar[d]_{f_0}\ar[r]^{Fu}&Fa_1\ar[d]^{f_1}\\
b_0
\ar@{}@<25pt>[d]|(.5){\cong}
\ar[d]_{p}\ar[r]^{1}&b_0\ar[d]^{p}\\b_1\ar[r]^{1}&b_1
}
\hspace{1cm}
\xymatrix{b_0
\ar@{}@<-33pt>[dd]|{\textstyle \beta\odot p:}
\ar@{}@<25pt>[d]|(.5){\cong}
\ar[d]_{p}\ar[r]^{1}&b_0\ar[d]^{p}\\
b_1
\ar@{}@<25pt>[d]|(.35){\beta}|(.5){\Rightarrow}
\ar[d]_{f_0}\ar[r]^{1}&b_1\ar[d]^{f_1}\\F'\!a'_0\ar[r]^{F'\!u'}&F'\!a'_1
}
\end{array}
$$
that is,

\vspace{0.2cm}
 $p\odot\beta=\Big(1\circ
(p\circ f_0)\overset{\bl}\Longrightarrow p\circ f_0\overset{1\circ
\bl^{-1}}\Longrightarrow p\circ (1\circ f_0) \overset{1\circ
\beta}\Longrightarrow p\circ (f_1\circ
Fu)\overset{\aso^{-1}}\Longrightarrow (p\circ f_1)\circ Fu\Big),$

\vspace{0.2cm}

$\beta\odot p=\Big(F'u'\circ (f_0\circ
p)\overset{\aso^{-1}}\Longrightarrow (F'u'\circ f_0)\circ
p\overset{\beta\circ 1}\Longrightarrow (f_1\circ 1 )\circ p
\overset{\br\circ 1}\Longrightarrow f_1\circ
p\overset{\br^{-1}}\Longrightarrow (f_1\circ p)\circ 1\Big). $

\subsection{Classifying spaces of bicategories.}\label{grb}
Briefly, let us recall from ~\cite[Definition 3.1]{CCG2010} that the
{\em
  classifying space} $\BB\B$ of a (small) bicategory $\B$ is defined
  as the geometric realization of the {\em Grothendieck
nerve} or {\em pseudo-simplicial nerve} of the bicategory, that is,
the pseudo-functor from $\Delta^{\!op}$ to the 2-category $\Cat$ of
small categories
\begin{equation}\label{GrothendieckNerve}
\ner\B: \Delta^{\!op}\to \Cat,
\ \ [p]\mapsto  \bigsqcup_{{(b_{0},\ldots,b_p)}}\hspace{-0.3cm}
  \B(b_{p-1},b_{p})\times\B(b_{p-2},b_{p-1})\times\cdots\times\B(b_0,b_1),
\end{equation}
whose face and degeneracy functors are defined in the standard way
by using the horizontal composition and identity morphisms of the
bicategory, and the natural isomorphisms $d_id_j\cong d_{j-1}d_i$,
etc., being given from the associativity and unit constraints of the
bicategory (see Theorem \ref{groner} in the Appendix, for more
details). Thus,
$$\xymatrix{\BB\B=\BB\int_\Delta\!\ner\B}$$
is the classifying space of the category $\int_\Delta\!\ner\B$
obtained by the Grothendieck construction ~\cite{Grothendieck1971}
on the pseudofunctor $\ner\B$. In other words, $\BB\B=
|\ner\int_\Delta\!\ner\B|$ is the geometric realization of the
simplicial set nerve of the category $\int_\Delta\!\ner\B$. When
$\B$ is a 2-category, then $\BB\B$ is homotopy  equivalent to
Segal's classifying space ~\cite{Segal1968} of the topological
category obtained from $\B$ by replacing the hom-categories
$\B(x,y)$ by their classifying spaces $\BB\B(x,y)$, see
~\cite[Remark 3.2]{CCG2010}.

In ~\cite[\S 4]{CCG2010}, it is proven that the classifying space
construction, $\B\mapsto \BB\B$, is a functor $\BB:\mathbf{Hom}\to
\mathbf{Top}$, from the category of bicategories and homomorphisms
to the category $\mathbf{Top}$ of spaces (actually of CW-complexes).
In this paper, we need the extension of this fact stated in  part
$(i)$ of the lemma below.

\begin{lemma}\label{fact1} $(i)$ The assignment
$\B\mapsto \BB\B$ is the function on objects of two functors into
the category of spaces
$$\xymatrix{\mathbf{Lax}\ar[r]^{\BB}&\mathbf{Top}&\mathbf{opLax}\ar[l]_{\BB}},$$
where $\mathbf{Lax}$ (resp. $\mathbf{opLax}$) is the category of
bicategories with lax (resp. oplax) functors between them as
morphisms.

\vspace{0.2cm} $(ii)$ If $F,G:\B\to\C$ are two lax or oplax functors
between bicategories, then any lax or oplax transformation between
them $\alpha:F\Rightarrow G$ determines a homotopy, $\BB\alpha:\BB
F\Rightarrow \BB G:\BB\B\to \BB\C$, between the induced maps on
classifying spaces.
\end{lemma}
\begin{proof} It is given in the Appendix, Corollaries
\ref{pfact1} and \ref{proffi}.\end{proof}

Other possibilities for defining $\BB\B$ come from  the {\em
geometric nerves} of the bicategory, first defined by Street
~\cite{Street1996} and studied, among others,  by Duskin
~\cite{Duskin2002}, Gurski ~\cite{Gurski2009} and Carrasco, Cegarra,
and Garz\'on ~\cite{CCG2010}; that is, the simplicial sets
\begin{equation}\label{eqgeoner}
\begin{array}{lll}
  \Delta^{\hspace{-2pt}\mathrm{u}}\B:\Delta^{\!op} \to \ \Set,
& &[p]\mapsto \mathrm{NorLax}([p],\B),\\[4pt]
  \Delta\B:\Delta^{\!op} \to \ \Set,
& &[p]\mapsto \mathrm{Lax}([p],\B),\\[4pt]
  \nabla_{\hspace{-2pt}{\mathrm{u}}}\B:\Delta^{\!op} \to \ \Set,
& &[p]\mapsto \mathrm{NorOpLax}([p],\B),\\[4pt]
  \nabla\B:\Delta^{\!op} \to \ \Set,
& &[p]\mapsto \mathrm{OpLax}([p],\B),
\end{array}
\end{equation}
whose respective $p$-simplices are the normal lax, lax, normal
oplax, and oplax functors from the category $[p]$ into the
bicategory $\B$.  In the Homotopy Invariance Theorem ~\cite[Theorem
6.1]{CCG2010} the existence of homotopy equivalences
\begin{equation}\label{4h}
  |\Delta^{\hspace{-2pt}\mathrm{u}}\B|\,\simeq \,|\Delta\B|\,\simeq\,
\BB\B\,\simeq \,|\nabla \B|\,\simeq\,
|\nabla_{\hspace{-2pt}{\mathrm{u}}}\B|,
\end{equation}
 it is proven, but their natural behaviour is not studied. Since, to establish the
results in this paper, we need to know that all the homotopy
equivalences above are homotopy natural, we state the following

\begin{lemma}\label{facts2}
For any bicategory $\B$, the first homotopy equivalence in
$(\ref{4h})$ is  natural on normal lax functors, the second one is
homotopy natural on lax functors, the third one is homotopy natural
on oplax functors, and the fourth one is natural on normal oplax
functors.
\end{lemma}
\begin{proof} By ~\cite[Theorem 6.2]{CCG2010}, the homotopy equivalence
 $|\Delta^{\hspace{-2pt}\mathrm{u}}\B|\,\simeq \,|\Delta\B|$ is
 induced on geometric realizations by the inclusion map
 $\Delta^{\hspace{-2pt}\mathrm{u}}\B\hookrightarrow \Delta\B$.
 Therefore, it is clearly natural on normal lax functors between
 bicategories. Similarly, the homotopy equivalence
 $|\nabla_{\hspace{-2pt}{\mathrm{u}}}\B|\,\simeq\,|\nabla \B|$ is natural on normal oplax
functors. The proof for the other two is more complicated and  is
given in the Appendix, Corollary \ref{pfact2}.\end{proof}

\section{Inducing homotopy pullbacks on classifying spaces}
\label{mainSection}

Quillen's Theorem B ~\cite{Quillen1973} provides a sufficient
condition on a functor between small categories $F:\A\to \B$ for the
classifying space $\BB(F\!\downarrow\!b)$  to be a homotopy-fibre
over the 0-cell $|b|\in\BB \B$ of the induced map $\BB F:\BB \A
\to\BB \B$, for each object $b\in\Ob \B$. The condition is that the
maps $\BB p_*:\BB (F\!\downarrow \!b) \to \BB (F\!\downarrow \!b')$
are homotopy equivalences for every morphism $p:b\to b'$ in the
category $\B$. That condition was referred to by Dwyer, Kan, and
Smith in ~\cite[\S 6]{DKS1989} by saying that ``the functor $F$ has
the property B" (see also Barwick, and Kan in
~\cite{BK2011,BK2013}). To state our theorem below, we shall adapt
that terminology to the bicategorical setting, and we will say that
\begin{itemize}
\item[(B$_l$)] {\em a lax functor} between bicategories $F:\A\to \B$
  {\em has the property} B$_l$ if, for any $1$-cell $p:b_0\to
  b_1$ in $\B$, the $2$-functor $p_*:F\!\downarrow\! b_0 \to
  F\!\downarrow\! b_1$ in $(\ref{ps})$ induces a homotopy
  equivalence on classifying spaces, $\BB(F\!\downarrow\! b_0)
  \simeq \BB(F\!\downarrow\! b_1)$.

\item[(B$_o$)] {\em an oplax functor} between bicategories $F':\A'\to
  \B$ {\em has the property} $\mathrm{B}_o$ if, for any $1$-cell
  $p:b_0\to b_1$ in $\B$, the $2$-functor
  $p^*:b_1\!\downarrow\!F' \to b_0\!\downarrow\! F'$ in
  $(\ref{ps})$ induces a homotopy equivalence on classifying
  spaces, $\BB(b_1\!\downarrow\! F') \simeq
  \BB(b_0\!\downarrow\! F')$.
\end{itemize}

The main result in this paper can be summarized as follows:
\begin{theorem}\label{mainTheorem} Let $\xymatrix@C=18pt{\A\ar[r]^-{F} & \B &\A'\ar[l]_(.4){F'}}$
  be a diagram of bicategories, where $F$ is a lax functor and $F'$ is
  an oplax functor (for instance, if $F$ and $F'$ are any two homomorphisms).

  $(i)$ There is a homotopy $\BB F\,\BB P\Rightarrow \BB F'\,\BB
P'$, so that the square below, which is induced by
  $(\ref{pqsquare})$ on classifying spaces, is homotopy commutative.\begin{equation}\label{hpqsquare}\begin{array}{c}
\xymatrix{
\BB(F\!\downarrow\!F')\ar@{}@<30pt>[d]|(.45){\Rightarrow}\ar[r]^-{\BB
P'}\ar[d]_{\BB P}
&\BB \A'\ar[d]^{\BB F'}\\
\BB \A\ar[r]^{\BB F}&\BB\B}
\end{array}
\end{equation}

\vspace{0.2cm} $(ii)$ Suppose that $F$ has the property
$\mathrm{B}_l$ or $F'$ has the
  property $\mathrm{ B}_o$. Then,  the square $(\ref{hpqsquare})$
is a homotopy pullback.

Therefore, by Dyer and Roitberg ~\cite{DR1980}, for each $a\in
\Ob\A$ and $a'\in\Ob\A'$ such that $Fa=F'\!a'$ there is an induced
Mayer-Vietoris type long exact sequence on homotopy groups based at
the $0$-cells $\BB a$ of $\BB\A$, $\BB Fa$ of $\BB \B$, $\BB a'$ of
$\BB\A'$, and $\BB(a,1,a')$ of $\BB(F\!\downarrow\!F')$,
$$
\xymatrix@C=12pt{
\cdots\to \pi_{n+1}\BB \B\ar[r]
&\pi_n\BB(F\!\downarrow\!F')\ar[r]&
\pi_n\BB\A\times\pi_n\BB\A'
\ar[r]&\pi_n\BB \B\to \cdots}
$$
$$
\xymatrix@C=12pt{
\cdots\to \pi_1\BB(F\!\downarrow\!F')\ar[r]
&\pi_1\BB\A\times\pi_1\BB\A'
\ar[r]&\pi_1\BB \B\ar[r]
&\pi_0\BB(F\!\downarrow\!F')\ar[r]&\pi_0(\BB\A\times\BB\A').}
$$

$(iii)$ If the square $(\ref{hpqsquare})$ is a homotopy pullback for
every $F'=b:[0]\to \B$, $b\in \Ob\B$, then $F$ has the property
$\mathrm{B}_l$. Similarly, if  the square $(\ref{hpqsquare})$ is a
homotopy pullback for any $F=b:[0]\to \B$, $b\in \Ob\B$, then $F'$
has the property $\mathrm{B}_o$.
\end{theorem}

The remainder of this section is devoted to the proof of this
theorem. We shall start by recalling from ~\cite[Lemma 5.2]{CCH2013}
the following lemma.

\begin{lemma}\label{cont}
  For any object $b$ of a bicategory $\B$, the classifying spaces of
  the comma bicategories $\B\!\!\downarrow{\!b}$ and
  $b\!\!\downarrow{\!\B}$ are contractible, that is,
$\BB(\B\!\!\downarrow{\!b})\simeq \mathrm{pt}\simeq
\BB(b\!\!\downarrow{\!\B})$.\end{lemma}

We also need the  auxiliary result below. To state it, we  use that,
for any given diagram $F:\A\to\B \leftarrow\A':F'$, with $F$ a lax
functor and $F'$  an oplax functor, and  for each objects $a$ of
$\A$ and $a'$ of $\A'$, there are normal homomorphisms
\begin{equation}\label{JF}
\xymatrix{
Fa\!\downarrow\!F'\ar[r]^{J}
& F\!\downarrow \!F'&F\!\downarrow\! F'\!a'\ar[l]_{J'},}
\end{equation}
where $J$ acts  on cells  by 
$$
\xymatrix@C=2.3pc{
(f_0,a'_0)\ar@/^11pt/[r]^{(\beta,u')}\ar@/_11pt/[r]_{(\bar{\beta},\bar{u}')}
\ar@{}[r]|{\Downarrow \alpha'}
&(f_1,a'_1)}
\xymatrix@C=16pt{\ar@{|->}[r]^{J}&}
\xymatrix@C=4.5pc{(a,f_0,a'_0)\ar@/^12pt/[r]^{(1_a,\imath(\beta, u'),u')}
\ar@/_12pt/[r]_{(1_a,\imath(\bar{\beta},\bar{u}'),\bar{u}')}
\ar@{}[r]|{\Downarrow (1,\alpha')}&
(a,f_1, a'_1),}
$$
where, for any 1-cell $(\beta,u'):(f_0,a'_0)\to (f_1,a'_1)$ in
$Fa\!\downarrow\!F'$, the 2-cell $\imath(\beta,u')$ is  defined as the composite 
$$
\begin{array}{ll}
\imath(\beta,u')=\big(F'u'\circ f_0\overset{\beta}\Longrightarrow
f_1\circ 1_{Fa}\overset{1\circ \widehat{F}_a}\Longrightarrow
f_1\circ F1_a\big),
\end{array}
$$
and whose  constraints, at pairs of 1-cells
$\xymatrix@C=20pt{
(f_0,a'_0)\ar[r]^{(\beta_1,u'_1)}&(f_1,a'_1)\ar[r]^{(\beta_2,u'_2)}
& (f_2,a'_2)}$ in $Fa\!\downarrow\!F'$, are the 2-cells of
$F\!\downarrow\!F'$
$$\begin{array}{l}
(\bl_{1_{a}},1_{u'_2\circ u'_1}):
(1_a\circ 1_a, \imath(\beta_2,u'_2)
\circledcirc\imath(\beta_1,u'_1),u'_2\circ u'_1)\cong
(1_a,\imath(\beta_2\circledcirc\beta_1,u'_2\circ u'_1),u'_2\circ u'_1).
\end{array}
$$

Similarly, $J'$ acts by
$$
\xymatrix@C=2.3pc{
(a_0,f_0)\ar@/^11pt/[r]^{(u,\beta)}\ar@/_11pt/[r]_{(\bar{u},\bar{\beta})}
\ar@{}[r]|{\Downarrow \alpha}
&(a_1,f_1)}
\xymatrix@C=16pt{\ar@{|->}[r]^{J'}&}
\xymatrix@C=4.5pc{
(a_0,f_0,a')\ar@/^12pt/[r]^{(u, \imath'\!(u,\beta),1_{a'})}
\ar@/_11pt/[r]_{(\bar{u}, \imath'\!(\bar{u},\bar{\beta}),1_{a'})}
\ar@{}[r]|{\Downarrow (\alpha,1)}&
(a_1,f_1,a'),}
$$
where, for any 1-cell  $(u,\beta):(a_0,f_0)\to (a_1,f_1)$ in
$F\!\downarrow\! F'\!a'$, the 2-cell 
$\imath'\!(u,\beta)$ is defined as the composites
$$
\begin{array}{ll}
\imath'\!(u,\beta)=
\big(F'1_{a'}\circ f_0\overset{\widehat{F}'\circ 1}\Longrightarrow
1_{F'\!a'}\circ f_0\overset{\beta}\Longrightarrow f_1\circ Fu\big),
\end{array}
$$
and whose constraints, at pairs of 1-cells
$\xymatrix@C=20pt{
(a_0,f_0)\ar[r]^{(u_1,\beta_1)}&(a_1,f_1)\ar[r]^{(u_2,\beta_2)} &
(a_2,f_2)}$ in $F\!\downarrow\!F'\!a'$, are the 2-cells of
$F\!\downarrow\!F'$
$$\begin{array}{l}
(1_{u_2\circ u_1},\bl_{1_{a'}}):
(u_2\circ u_1,\imath'\!(u_2,\beta_2)
\circledcirc\imath'\!(u_1,\beta_1),1_{a'}\circ 1_{a'})\cong
(u_2\circ u_1,\imath'\!(u_2\circ u_1,\beta_2\circledcirc\beta_1),
1_{a'}).
\end{array}
$$

\begin{lemma}\label{lemma1}  Let
$\xymatrix@C=18pt{\A\ar[r]^-{F} & \B &\A'\ar[l]_(.4){F'}}$ be any
diagram of bicategories, where $F$ is a lax functor and $F'$ is an
oplax functor.

$(i)$ If $\A$ is a category with an initial object {\em{\scriptsize
    0}}, then the homomorphism $J$ in $(\ref{JF})$ induces a homotopy
equivalence on classifying spaces, $\BB(F\text{\em{\scriptsize
    0}}\!\downarrow\!F')\simeq \BB(F\!\downarrow\!F')$.

$(ii)$  If $\A'$ is a  category with a terminal object
{\em{\scriptsize 1}}, then the homomorphism $J'$ in $(\ref{JF})$
induces a homotopy equivalence on classifying spaces,
$\BB(F\!\downarrow\!F'\!\text{\em{\scriptsize 1}})\simeq
\BB(F\!\downarrow\!F')$.

\end{lemma}
\begin{proof}
We only prove $(i)$ since the proof of $(ii)$ is parallel. Let
$\langle a\rangle :\text{{\scriptsize 0}}\to a$ be the unique
morphism in $\A$ from the initial object to $a$. There is a
2-functor $ L:F\!\downarrow \!F'\to F\text{{\scriptsize
0}}\!\downarrow\! F' $ given on cells by
$$
\xymatrix@C=3pc{
(a_0,f_0,a'_0)\ar@/^12pt/[r]^{(u,\beta,u')}
             \ar@/_12pt/[r]_{({u},\bar{\beta},\bar{u}')}
\ar@{}[r]|{\Downarrow(1,\alpha')}
&(a_1,f_1,a'_1)}\overset{L}\mapsto
 \xymatrix@C=3pc{
(f_0\circ F\langle a_0\rangle,\!a'_0)\ar@/^12pt/[r]^{(\ell(u,\beta, u'),u')}
 \ar@/_12pt/[r]_{(\ell({u},\bar{\beta},\bar{u}'),\bar{u}')}
     \ar@{}[r]|{\Downarrow\alpha'}
&(f_1\circ F\langle a_1\rangle, a'_1),}
$$
where, for any 1-cell $(u,\beta,u'):(a_0,f_0,a'_0)\to
(a_1,f_1,a'_1)$ of $F\!\downarrow \!F'$, $\ell(u,\beta,u')$ is the
2-cell of $\B$ obtained by pasting the diagram
$$\xymatrix{F\text{{\scriptsize
0}}
\ar@{}@<-58pt>[dd]|{\textstyle \ell(u,\beta,u'):}
\ar@{}@<25pt>[d]|(.33){\br^{-1}\!\cdot \widehat{F}}|(.55){\cong}
\ar[d]_{F\langle a_0\rangle}\ar[r]^{1}&F\text{{\scriptsize
0}}\ar[d]^{F\langle a_1\rangle}\\
Fa_0
\ar@{}@<25pt>[d]|(.35){\beta}|(.5){\Rightarrow}
\ar[d]_{f_0}\ar[r]^{Fu}&Fa_1\ar[d]^{f_1}\\F'\!a'_0\ar[r]_{F'\!u'}&F'\!a'_1
}
$$
that is,
$$
\begin{array}{cl}
\ell(u,\beta,u')=\Big(&\hspace{-0.3cm}F'u'\circ (f_0\circ
F\langle a_0\rangle)\overset{\aso^{-1}}\Longrightarrow
( F'u'\circ f_0)\circ
F\langle a_0\rangle \overset{\beta\circ 1}\Longrightarrow
(f_1\circ Fu)\circ F\langle a_0\rangle \overset{\aso}\Longrightarrow\\
 &\hspace{-0.3cm}f_1\circ (Fu\circ F\langle a_0\rangle)
\overset{1\circ \widehat{F}}\Longrightarrow
 f_1\circ F(u\circ \langle a_0\rangle)  = f_1\circ F\langle a_1\rangle
   \overset{\br^{-1}}\Longrightarrow(f_1\circ F\langle a_1\rangle)\circ 1\Big).
\end{array}
$$

In addition, there are two pseudo-transformations
$$
1_{F\text{{\tiny 0}}\downarrow F'}\Rightarrow LJ,
\hspace{0.4cm}JL\Rightarrow 1_{F\downarrow F'}.
$$
The first one has as a component, at any object $(f,a')$ of
$F\text{{\scriptsize 0}}\!\downarrow\!F'$, the 1-cell
$$\begin{array}{c}
(\eta(f,a'),1_{a'}):(f,a')\to (f\circ F1_{\text{{\tiny 0}}},a'),\\
\eta(f,a')=\xymatrix@C=15pt{\big(F'1_{a'}\circ f\ar@2[r]^{\widehat{F}'\circ 1}
 & 1_{F'a'}\circ f\ar@2[r]^-{\bl}
 & f\ar@2[r]^-{\br^{-1}}
& f\circ 1_{F\text{{\tiny 0}}} \ar@2[r]^(.45){1\circ \widehat{F}}
 & f\circ F1_{\text{{\tiny 0}}}\ar@2[r]^-{\br^{-1}}
&(f\circ F1_{\text{{\tiny 0}}})\circ
1_{F\text{{\tiny 0}}}\big)}
\end{array}
$$
while its naturality component, at any 1-cell
$(\beta,u'):(f_0,a'_0)\to (f_1,a'_1)$ of $F\text{{\scriptsize
0}}\!\downarrow\!F'$, is given
 by the canonical isomorphism $\bl^{-1}\cdot \br: u'\circ 1_{a'_0}\cong 1_{a'_1}\circ u'
 $,
 $$
\xymatrix@C=7pc{
(f_0,a'_0)\ar[d]_-{(\eta,1)}\ar[r]^{(\beta,u')}
  \ar@{}@<80pt>[d]|(.55){\cong}|(.42){\bl^{-1}\cdot \br}
 & (f_1,a'_1)\ar[d]^{(\eta,1)}
\\ (f_0\circ F1_{\text{{\tiny 0}}},a'_0)
 \ar[r]_{(\ell(1_{\text{{\tiny 0}}},\imath(\beta,u'),u'),u')}
 & (f_1\circ F1_{\text{{\tiny 0}}},a'_1).}
$$

As for the pseudo-transformation $JL\Rightarrow 1_{F\downarrow F'}$,
it associates to an object $(a,f, a')$ in $F\!\downarrow\!F'$ the
1-cell
$$\begin{array}{c}
(\langle a\rangle,\epsilon(a,f,a'), 1_{a'}):
(\text{{\scriptsize 0}},f\circ F\langle a\rangle,a')\to (a,f,a')
\\
\epsilon(a,f,a')=\xymatrix@C=15pt{\big(F'1_{a'}\circ (f\circ F\langle
a\rangle)\ar@2[r]^{\widehat{F}'\circ 1}
 & 1_{F'a'}\circ (f\circ F\langle a\rangle) \ar@2[r]^-{\bl}
 & f \circ F\langle a\rangle\big)}
 \end{array}
$$
while its naturality component, at a 1-cell
$(u,\beta,u'):(a_0,f_0,a'_0)\to (a_1,f_1,a'_1)$ of
$F\!\downarrow\!F'$, is
$$ \xymatrix@C=9pc{
(\text{{\scriptsize 0}},f_0\circ F\langle a_0
\rangle,a'_0)\ar[d]_-{(\langle a_0\rangle,\epsilon,1)}
        \ar[r]^{(1_{\text{{\tiny 0}}},\imath(\ell(u,\beta,u'),u'),u')}
  \ar@{}@<100pt>[d]|(.55){\cong}|(.42){(1,\bl^{-1}\cdot \br)}
 & (\text{{\scriptsize 0}},f_1\circ F\langle a_1
\rangle,a'_1)\ar[d]^-{(\langle a_1\rangle,\epsilon,1)}
\\ (a_0,f_0,a'_0)\ar[r]_{(u,\beta,u')} 
 & (a_1,f_1,a'_1).}
$$

Therefore, by Lemma \ref{fact1}, there are homotopies $\BB J\,\BB
L\Rightarrow 1_{\BB(F\downarrow F')}$ and 
$1_{\BB(F\text{{\tiny 0}}\downarrow F')}\Rightarrow \BB L\,\BB J$ making $\BB J$ a homotopy equivalence.
\end{proof}

As we will see below, the following result is the key for proving
Theorem \ref{mainTheorem}.

\begin{lemma}\label{lemma2}

$(i)$ If an oplax functor $F':\A'\to \B$ has the property {\em
B}$_o$, then, for any lax functor $F:\A\to \B$, the commutative
square
    \begin{equation}\label{hbarsquarel}\begin{array}{c}
\xymatrix@C=20pt@R=20pt{\BB(F\!\downarrow\!F')\ar[r]^-{\BB\bar{F}}\ar[d]_{\BB P}
  &\BB(\B\!\downarrow\!F')\ar[d]^{\BB P}\\
\BB\A\ar[r]^{\BB F}&\BB \B,}
\end{array}
\end{equation}
induced by the first square in $(\ref{pullbackSquares})$ on
classifying spaces, is a homotopy pullback.

$(ii)$ If a lax functor $F:\A\to \B$ has the property {\em B}$_l$,
then, for any oplax functor $F'\!:\A'\to \B$, the commutative square
    \begin{equation}\label{hbarsquareo}\begin{array}{c}
\xymatrix@C=20pt@R=20pt{\BB(F\!\downarrow\!F')\ar[r]^-{\BB P'}\ar[d]_{\BB\bar{F}'}
&\BB\A'\ar[d]^{\BB F'}\\
\BB (F\!\downarrow\!\B)\ar[r]^-{\BB P'}&\BB\B,}
\end{array}
\end{equation}
induced by the second square in $(\ref{pullbackSquares})$ on
classifying spaces, is a homotopy pullback.
\end{lemma}
\begin{proof}
  Suppose that $F':\A'\to \B$ is any given oplax functor having
  the property B$_o$. We will prove that the simplicial map $\Delta
  P:\Delta(\B\!\downarrow\!F')\to \Delta\B$, induced on geometric
  nerves by the projection 2-functor $P:\B\!\downarrow\!F'\to \B$ in
  $(\ref{PP'})$, satisfies the condition $(i)$ of Lemma \ref{simlem}. To do so, let
  $\x:[n]\to \B$ be any geometric $n$-simplex  of $\B$. Thanks to
  Lemma \ref{pullbackLemma} $(i)$, the square
$$
\xymatrix@C=20pt@R=20pt{\x\!\downarrow\!F'\ar[r]^{\bar{\x}}\ar[d]_{P}
& \B\!\downarrow\!F'
\ar[d]^{P}\\
          [n]\ar[r]^\x&\B}
$$
is a pullback in the category of bicategories and lax functors,
whence the square induced by taking geometric nerves
$$
\xymatrix@C=20pt@R=20pt{
\Delta(\x\!\downarrow\!F')\ar[r]^{\Delta\bar{\x}}\ar[d]_{\Delta P}
&\Delta(\B\!\downarrow\!F')
\ar[d]^{\Delta P}\\
          \Delta[n]\ar[r]^{\Delta \x}&\Delta\B}
$$
is a pullback in the category of simplicial sets. Therefore,
$\xymatrix{\Delta P^{-1}(\Delta\x)\cong\Delta(\x\!\downarrow\!F')}$.
Furthermore, for any map $\sigma:[m]\to [n]$ in the simplicial
category, the diagram of lax functors
$$
\xymatrix@R=8pt{
\x\sigma\!\downarrow\!F'\ar[rd]_{\bar{\sigma}}\ar[rrd]^{\overline{\x\sigma}}
\ar[dd]_{P}&& \\
&\x\!\downarrow\!F'\ar[dd]^{P}\ar[r]_{\bar{\x}}
&\B\!\downarrow\!F' \ar[dd]^{P}\\
 [m]\ar[rd]_{\sigma}\ar[rrd]^(.27){\x\sigma}|!{[ru];[rd]}\hole&& \\
 & [n]\ar[r]_{\x}&\B}
$$
is commutative, whence the induced diagram of simplicial maps
$$
\xymatrix@R=8pt{
\Delta(\x\sigma\!\downarrow\!F')\ar[rd]_{\Delta\bar{\sigma}}
\ar[rrd]^{\Delta\overline{\x\sigma}}
\ar[dd]_{\Delta P}&& \\
&\Delta(\x\!\downarrow\!F')\ar[dd]^{\Delta P}\ar[r]_{\Delta\bar{\x}}
&\Delta(\B\!\downarrow\!F')
 \ar[dd]^{\Delta P}\\
 \Delta[m]\ar[rd]_{\Delta\sigma}
          \ar[rrd]^(.27){\Delta(\x\sigma)}|!{[ru];[rd]}\hole&& \\
 & \Delta[n]\ar[r]_{\Delta\x}&\Delta\B}
$$
is also commutative. Consequently,  the diagram below commutes.
$$
\xymatrix{
\Delta P^{-1}(\Delta\x\Delta \sigma)\ar[d]_{\cong}\ar[r]
&\Delta P^{-1}(\Delta\x)\ar[d]^{\cong}\\
\Delta(\x\sigma\!\downarrow\!F')\ar[r]^{\Delta\bar{\sigma}}
&\Delta(\x\!\downarrow\!F')
}
$$

Therefore, it suffices to prove that the lax functor
$\bar{\sigma}:\x\sigma\!\downarrow\!F'\to \x\!\downarrow\!F'$
induces a homotopy equivalence on classifying spaces,
$\BB(\x\sigma\!\downarrow\!F')\simeq \BB(\x\!\downarrow\!F')$. But
note that we have the diagram
$$
\xymatrix@R=20pt{
\x\sigma 0\!\downarrow\!F'\ar[r]^{J}\ar[d]_{\x(0,\sigma 0)^*}
         \ar@{}@<30pt>[d]|(.52){\Rightarrow}|(.38){\theta}
 & \x\sigma\!\downarrow\!F'\ar[d]^{\bar{\sigma}}
\\\x 0\!\downarrow\!F'\ar[r]^{J} & \x\!\downarrow\!F'}
$$
where the homomorphisms $J$ are given as in $(\ref{JF})$, and
$\theta$ is the pseudo-transformation that assigns to every object
$(f,a')$ of $\x\sigma 0\!\downarrow\!F'$ the 1-cell of
$\x\!\downarrow\!F'$
$$
\big((0,\sigma 0), \theta(f,a'), 1_{a'}\big):
(0,f\circ \x(0,\sigma 0),a')\to (\sigma 0,f,a'),
$$
where  the 2-cell of $\B$
$$
\xymatrix@C=35pt@R=20pt{
\x 0\ar[r]^{\x(0,\sigma 0)}\ar[d]_{f\circ \x(0,\sigma 0)}
\ar@{}@<30pt>[d]|(.58){\Rightarrow}|(.38){\theta(f,a')}
&\x\sigma 0\ar[d]^{f}\\
F'\!a'\ar[r]_{F'1_{a'}}&F'\!a'
}
$$
is the composite $\theta(f,a')=\big(F'1_{a'}\circ (f\circ
\x(0,\sigma 0))\overset{\widehat{F}'\circ 1}\Longrightarrow
1_{F'a'}\circ (f\circ \x(0,\sigma 0)) \overset{\bl}\Longrightarrow
f\circ \x(0,\sigma 0)\big) $, and its naturality component at any
1-cell $(\beta,u'):(f_0,a'_0)\to (f_1,a'_1)$
$$
\xymatrix@C=8pc{
(0,f_0\circ \x(0,\sigma 0),a'_0)
    \ar[r]^{((0,0),\imath(\beta\circledcirc \x(0,\sigma 0),u'),u')}
\ar[d]_{((0,\sigma 0),\theta(f_0,a'_0),1_{a'_0})}
\ar@{}@<90pt>[d]|(.56){\cong}|(.42){(1,\bl^{-1}\cdot \br)}
  & (0,f_1\circ \x(0,\sigma 0),a'_1)
  \ar[d]^{((0,\sigma 0),\theta(f_1,a'_1),1_{a'_1})}
\\ (\sigma 0,f_0,a'_0)
\ar[r]_{((\sigma 0,\sigma 0),\imath(\beta,u'),u')}
 & (\sigma 0,f_1,a'_1)
}
$$
is given by the canonical isomorphism $\bl^{-1}\cdot \br:u'\circ
1_{a'_1}\cong 1_{a'_2}\circ u'$ in $\A'$. Therefore, by Lemma
\ref{fact1}, the induced square on classifying spaces
$$
\xymatrix{
\BB(\x\sigma 0\!\downarrow\!F')\ar@{}@<35pt>[d]|(.5){\Rightarrow}|(.35){\BB\theta}\ar[r]^{\BB J}
\ar[d]_{\BB\x(0,\sigma 0)^*}
 & \BB(\x\sigma\!\downarrow\!F')\ar[d]^{\BB\bar{\sigma}}
\\\BB(\x 0\!\downarrow\!F')\ar[r]^{\BB J} & \BB(\x\!\downarrow\!F')}
$$
is homotopy commutative. Moreover, by Lemma \ref{lemma1}$(i)$, both
maps $\BB J$ in the square are homotopy equivalences and, since the
oplax functor $F'$ has the property B$_o$, the map $\BB \x(0,\sigma
0)^*:\BB(\x\sigma 0\!\downarrow\!F')\to \BB(\x 0\!\downarrow\!F')$
is also a homotopy equivalence. It follows that the remaining map in
the square has the same property, that is, the map $\BB\bar{\sigma}:
\BB(\x\sigma\!\downarrow\!F')\simeq \BB(\x\!\downarrow\!F') $ is a
homotopy equivalence, as required.

Suppose now that $F:\A\to \B$ is any lax functor. Again, by Lemma
\ref{pullbackLemma}$(i)$, the first square in
$(\ref{pullbackSquares})$ is a pullback in the category of
bicategories and lax functors, whence the square induced by taking
geometric nerves
\begin{equation}\label{dsqp}\begin{array}{c}
\xymatrix@R=20pt{
\Delta(F\!\downarrow\!F')\ar[r]^{\Delta\bar{F}}\ar[d]_{\Delta P}
&\Delta(\B\!\downarrow\!F')
\ar[d]^{\Delta P}\\
          \Delta\A\ar[r]^{\Delta F}&\Delta\B}
          \end{array}
\end{equation}
is a pullback in the category of simplicial sets. By what has been
already proven above, it follows from Lemma \ref{simlem} $(iii)$
that the commutative square
$$
\xymatrix{
|\Delta(F\!\downarrow\!F')|\ar[r]^-{|\Delta\bar{F}|}\ar[d]_{| P|}
&|\Delta(\B\!\downarrow\!F')|\ar[d]^{|P|}
   \ar@{}@<40pt>[d]|{\textstyle \overset{(\ref{4h})}\simeq}\\
|\Delta\A|\ar[r]^{|\Delta F|}&|\Delta \B|}\hspace{0.5cm}
\xymatrix@R=25pt{
\BB(F\!\downarrow\!F')\ar[r]^-{\BB\bar{F}}\ar[d]_{\BB P}
&\BB(\B\!\downarrow\!F')\ar[d]^{\BB P}\\
\BB\A\ar[r]^{\BB F}&\BB \B}
$$
is a homotopy pullback. This completes the proof of part $(i)$ of
the lemma.

The proof of part $(ii)$ follows similar lines, but using the
geometric nerve functor $\nabla$ instead of $\Delta$ as above. Thus,
for example, given any lax functor $F:\A\to \B$ having the property
B$_l$, we start by proving that the simplicial map $\nabla
P':\nabla(F\!\downarrow\!\B)\to \nabla \B$ satisfies the condition
$(i)$ in Lemma \ref{simlem}, which we do by first getting natural
simplicial isomorphisms $\nabla P'^{-1}(\nabla\x')\cong
\nabla(F\!\downarrow\!\x')$, for the different oplax functors
$\x':[n]\to \B$ (i.e., the simplices of $\nabla\B$), and then by
proving that any simplicial map $\sigma:[m]\to [n]$ induces a
homotopy equivalence $\BB(F\!\downarrow\!\x'\sigma)\simeq
\BB(F\!\downarrow\!\x')$. Here, we  need to use the homomorphisms
$J':F\!\downarrow\!\x' n\to F\!\downarrow\!\x'$ in $(\ref{JF})$,
which induce homotopy equivalences on classifying spaces by Lemma
\ref{lemma1} $(ii)$, and the existence of a pseudo-transformation
$\theta':\bar{\sigma}\, J'\Rightarrow J'\,\x'(\sigma m,n)_*$,
which assigns to every object $(a,f)$ of $F\!\downarrow\!\x'\sigma
m$ the 1-cell $ \big(1_a, \theta'(a,f), (\sigma m, n)\big):(a,
f,\sigma m)\to (a, \x'(\sigma m,n)\circ f, n) $, where
$$\theta'(a,f)=\big(\x'(\sigma m,n)\circ
f\overset{\br^{-1}}\Longrightarrow (\x'(\sigma m,n)\circ f)\circ
1_{Fa} \overset{1\circ \widehat{F}}\Longrightarrow (\x'(\sigma
  m,n)\circ f)\circ F1_{a}\big) .$$ Using  Lemma
\ref{pullbackLemma} $(ii)$ therefore, we deduce that, for any lax
functor $F':\A'\to \B$, the square
$$
\xymatrix@R=20pt{
\nabla(F\!\downarrow\!F')\ar[r]^-{\nabla P'}\ar[d]_{\nabla \bar{F}'}
&\nabla \A'\ar[d]^{\nabla F'}\\
\nabla (F\!\downarrow\!\B)\ar[r]^{\nabla P'}&\nabla \B,}
 $$
 is a  pullback in the  category of simplicial
 sets which, by Lemma \ref{simlem}, induces a homotopy pullback square on geometric realizations.
It follows that  $(\ref{hbarsquareo})$ is a homotopy pullback.
\end{proof}

With the corollary below we will be ready to complete the proof of
Theorem \ref{mainTheorem}.

\begin{corollary}\label{bacor}
$(i)$ For any lax functor $F:\A \to \B$, the projection $2$-functor
$P:F\!\downarrow\!\B \to \A$ induces a homotopy equivalence on
classifying spaces,  $\BB(F\!\downarrow\!\B)\simeq\BB\A$.

\vspace{0.2cm} $(ii)$ For any oplax functor $F':\A'\to \B$, the
projection $2$-functor $P':\B\!\downarrow\!F'\to \A'$ induces a
homotopy equivalence on classifying spaces,
$\BB(\B\!\downarrow\!F')\simeq \BB\A'$.
\end{corollary}
\begin{proof} Once again we limit ourselves to proving $(i)$. Let
  $F:\A\to \B$ be a lax functor.

  The identity
  homomorphism $1_\B:\B\to \B$ has the property B$_o$ since,
  for any object $b\in \Ob\B$,   the classifying space of the comma
  bicategory $b\!\downarrow\! \B$ is contractible, by Lemma \ref{cont}. Therefore, Lemma
  \ref{lemma2} $(i)$ applies to the case when $F'=1_\B$, and tells us
  that the induced commutative square
$$
\xymatrix@R=20pt{
\BB(F\!\downarrow\!\B)\ar[r]^-{\BB\bar{F}}\ar[d]_{\BB P}
&\BB(\B\!\downarrow\!\B)\ar[d]^{\BB P}\\
\BB\A\ar[r]^{\BB F}&\BB \B,}
$$
is a homotopy pullback. So, it is enough to prove that the map $\BB
P:\BB(\B\!\downarrow\! \B) \to \BB \B$ is a homotopy equivalence. To
do so, let $b$ be any object of $\B$, and let us particularize the
square above to the case where $F=b:[0]\to \B$. Then, we find the
commutative homotopy pullback square
$$
\xymatrix@R=20pt{\BB(b\!\downarrow\!\B)\ar[r]^-{\BB\bar{b}}\ar[d]_{\BB P} &
\BB(\B\!\downarrow\!\B)\ar[d]^{\BB P}\\ \mathrm{pt}\ar[r]^{\BB b}&\BB \B,}
$$
where, by Lemma \ref{cont}, the left vertical map is a homotopy
equivalence. This tells us that the different homotopy fibres of the
map $\BB P:\BB(\B\!\downarrow\!\B)\to \BB\B$ over the 0-cells of
$\BB\B$ are all contractible, and consequently $\BB P$ is actually a
homotopy equivalence.
\end{proof}

We can now complete the proof of Theorem \ref{mainTheorem}:

For any diagram $\xymatrix@C=18pt{\A\ar[r]^-{F} & \B
&\A'\ar[l]_(.4){F'}}$,  where $F$ is a lax functor and $F'$ is an
oplax functor, the square $(\ref{hpqsquare})$ occurs as the
 outside region in both of the following two diagrams:
\begin{equation}\label{2dia}\begin{array}{cc}
\xymatrix{
\BB(F\!\downarrow\!F')
\ar@{}[rd]|(.46){\mathrm{(\ref{hbarsquarel})}}
\ar@{}@<13pt>[rr]|{=}
\ar[d]_{\BB P}
\ar@/^2pc/[rr]^{\BB P'}\ar[r]^{\BB\bar{F}}& \BB(\B\!\downarrow\!F')
\ar@{}@<30pt>[d]|(.5){\Rightarrow}|(.35){\BB\omega}
\ar[d]_{\BB P}\ar[r]^{\BB P'}&\BB\A'\ar[d]^{\BB F'}\\
\BB\A \ar@{}@<-12pt>[rr]|{=}\ar@/_1.7pc/[rr]_{\BB F}\ar[r]^{\BB F}
&\BB\B\ar[r]^{\BB 1_\B}&\BB\B
}
&
\xymatrix{\BB(F\!\downarrow\!F')
\ar@/_3pc/[dd]_{\BB P}
\ar@{}@<-25pt>[dd]|(.45){=}
\ar[r]^{\BB P'}\ar[d]_{\BB \bar{F}'}
\ar@{}[rd]|(.46){\mathrm{(\ref{hbarsquareo})}}
& \BB \A' \ar[d]^{\BB F'}
\ar@/^2.5pc/[dd]^{\BB F'}
\ar@{}@<18pt>[dd]|(.45){=}
\\
\BB(F\!\downarrow\!\B)
\ar@{}@<30pt>[d]|(.5){\Rightarrow}|(.35){\BB\omega'}
\ar[d]_{\BB P}\ar[r]^{\BB P'}&\BB\B\ar[d]^{\BB 1_\B}\\
\BB\A\ar[r]_{\BB F}&\BB\B
}
\end{array}
\end{equation}
where the inner squares with the homotopies labelled  $\BB\omega$
and $\BB\omega'$ are the particular cases of the squares
$(\ref{hpqsquare})$ obtained when $F=1_\B$ and when $F'=1_\B$,
respectively. The homotopies are respectively induced, by Lemma
\ref{fact1}, by the lax transformations
$$
\xymatrix@R=20pt{\B\!\downarrow\!F'\ar[r]^{P'}\ar[d]_{P}
&\A'\ar[d]^{F'}\\
  \B\ar@{}@<-25pt>[u]|(.7){\omega}|(.55){\Rightarrow}\ar[r]^{1_\B}
&\B}\hspace{0.6cm}
\xymatrix@R=20pt{F\!\downarrow\!\B\ar[r]^{P'}\ar[d]_{P}&\B\ar[d]^{1_\B}\\
  \A\ar@{}@<-25pt>[u]|(.7){\omega'}|(.52){\Rightarrow}\ar[r]^{F}&\B}
$$
which are defined as follows: The lax transformation $\omega$
associates to any object $(b,f,a')$ of $\B\!\downarrow\!F'$ the
1-cell $f:b\to F'a'$, and its naturality component at any 1-cell
$(p,\beta,u'):(b_0,f_0,a'_0)\to (b_1,f_1,a'_1)$ is the 2-cell
$\beta:F'u'\circ f_0\Rightarrow f_1\circ p$. Similarly, $\omega'$
associates to any object $(a,f,b)$ of $F\!\downarrow\!\B$ the 1-cell
$f:Fa\to b$, and its naturality component at any 1-cell
$(u,\beta,p):(a_0,f_0,b_0)\to (a_1,f_1,b_1)$ is $\beta:p\circ
f_0\Rightarrow f_1\circ Fu$. Since, by Corollary \ref{bacor}, both
maps $\BB P':\BB (\B\!\downarrow\!  F')\to \BB\A'$ and $\BB P:\BB
(F\!\downarrow\! \B)\to \BB\A$ are homotopy equivalences, both
squares are homotopy pullbacks. The other inner squares are those
referred to therein.

 The above implies the part $(i)$ of Theorem \ref{mainTheorem} and, furthermore, it follows that the square $(\ref{hpqsquare})$ is a homotopy
pullback whenever one of the inner squares $(\ref{hbarsquarel})$ or
$(\ref{hbarsquareo})$ is a homotopy pullback.  Therefore, Lemma
\ref{lemma2} implies part $(ii)$.

 For proving part $(iii)$, suppose a lax functor $F:\A\to \B$ is given
 such that the square $(\ref{hpqsquare})$ is a homotopy pullback for
 any $F'=b:[0]\to \B$, $b\in\Ob\B$. It follows from the diagram on the
 left in $(\ref{2dia})$ that the inner square (\ref{hbarsquarel})
 $$
 \xymatrix{
\BB(F\!\downarrow\!b)\ar[r]^{\BB\bar{F}}\ar[d]_{\BB P} &\BB(\B\!\downarrow\!b)
\ar[d]^{\BB P}\\
          \BB\A\ar[r]^{\BB F}&\BB\B}
 $$
 is a homotopy pullback for any object $b\in\B$. Then, if $p:b_0\to
 b_1$ is any 1-cell of $\B$, since we have the commutative
 diagram
 $$
\xymatrix@R=12pt@C=14pt{\BB(F\!\downarrow\!b_0)\ar[dd]_{\BB P}
  \ar[rr]^{\BB\bar{F}}\ar[rd]^{\BB p_*}&&
\BB(\B\!\downarrow\!b_0)
\ar[dd]^(.3){\BB P}
\ar[rd]^{\BB p_*}& \\
&\BB(F\!\downarrow\!b_1)\ar[rr]^(.3){\BB\bar{F}}|!{[ru];[dr]}\hole
\ar[ld]^{\BB P}
&&\BB(\B\!\downarrow\!b_1)\ar[ld]^{\BB P} \\
\BB\A\ar[rr]^{\BB F}&&\BB\B}
$$
we deduce that the square
$$
\xymatrix{
\BB(F\!\downarrow\!b_0)\ar[r]^-{\BB\bar{F}}\ar[d]_{\BB p_*}
&\BB(\B\!\downarrow\!b_0)\ar[d]^{\BB p_*}\\
\BB(F\!\downarrow\!b_1)\ar[r]^-{\BB\bar{F}}
&\BB(\B\!\downarrow\!b_1)} \hspace{0.6cm}
$$
is also a homotopy pullback. Therefore, as
$\BB(\B\!\downarrow\!b_0)\simeq \mathrm{pt} \simeq
\BB(\B\!\downarrow\!b_1)$, by Lemma \ref{cont}, we conclude that the
induced map $\BB p_*:\BB(F\!\downarrow\!b_0)\simeq
\BB(F\!\downarrow\!b_1)$ is a homotopy equivalence. That is, the lax
functor $F$ has the property B$_l$.
\endproof

As a corollary, we obtain the following theorem, which is just the
well-known Quillen's Theorem B ~\cite{Quillen1973} when the lax or
oplax functor $F$ in the hypothesis is an ordinary functor between
small categories. The generalization  of Theorem B to lax
functors between bicategories was originally stated and proven by
Calvo, Cegarra, and Heredia in ~\cite[Theorem 5.4]{CCH2013}, and it
also generalizes a similar result by the first author in
~\cite[Theorem 3.2]{Cegarra2011} for the case when $F$ is a
2-functor between 2-categories.
\begin{corollary}\label{hfth} $(i)$ If a lax functor $F:\A\to \B$
 has the property  $\mathrm{B}_l$
  then, for every object $b\in \B$, there is an induced homotopy fibre
  sequence
$$
 \xymatrix{ \BB(F\!\downarrow\!b)\ar[r]^-{\BB P}&\BB\A\ar[r]^-{\BB F}&\BB \B.}
$$

$(ii)$ If an oplax functor $F':\A'\to \B$ has the property
$\mathrm{B}_o$ then, for every object $b\in \B$, there is an induced
homotopy fibre sequence
$$
\xymatrix{ \BB(b\!\downarrow\!F')\ar[r]^-{\BB P'}&\BB\A'\ar[r]^-{\BB
    F'}&\BB \B.}
$$
\end{corollary}
\begin{proof} It follows from Theorem \ref{mainTheorem}, by taking
  $F'=b:[0]\to \B$ to obtain part $(i)$ and $F=b:[0]\to \B$ for part
  $(ii)$.
\end{proof}

By the above result in ~\cite{CCH2013,Cegarra2011}, the bicategories
$F\!\downarrow\!b$ and $b\!\downarrow\!F'$ are called {\em
homotopy-fibre bicategories}. The following consequence was proven
in ~\cite[Theorem 5.6]{CCH2013}, and it shows a generalization of
Quillen's Theorem A ~\cite{Quillen1973}.

\begin{corollary}\label{tagen}
  $(i)$ Let $F:\A\to \B$ be a lax functor such that the classifying
  spaces of its homotopy-fibre categories are contractible, that is,
  $\BB(F\!\downarrow\!b)\simeq \mathrm{pt}$ for every object $b\in \B$. Then,
  the induced map on classifying spaces $\BB F:\BB\A\to \BB \B$ is a
  homotopy equivalence.

\vspace{0.1cm}
  $(ii)$ Let $F':\A'\to \B$ be an oplax functor such that the
  classifying spaces of its homotopy-fibre categories are
  contractible, that is, $\BB(b\!\downarrow\!F')\simeq \mathrm{pt}$ for every
  object $b\in \B$. Then, the induced map on classifying spaces $\BB
  F':\BB\A'\to \BB \B$ is a homotopy equivalence.
\end{corollary}

Particular cases of the above results have also been stated by
Bullejos and Cegarra in ~\cite[Theorem 1.2]{BC2003}, for the case
when $F:\A\to \B$ is any 2-functor between 2-categories, and by del
Hoyo in ~\cite[Theorem 6.4]{Hoyo2012}, for the case when $F$ is a
lax functor from a category $\A$ to a 2-category $\B$. In
~\cite[Th\'{e}or\`{e}me 6.5]{Chiche2012}, Chiche proved a relative
Theorem A for lax functors between 2-categories, which also
specializes by giving the particular case of Theorem \ref{tagen}
when $F$ is any lax functor between 2-categories.

Next we study conditions on a bicategory $\B$ in order for the
square $(\ref{hpqsquare})$ to always be a homotopy pullback. We use
that, for any two objects $b$, $b'$ of a bicategory $\B$, there is a
diagram
\begin{equation}\label{ssq}\begin{array}{c}
\xymatrix@R=18pt{
\B(b,b')\ar@{}@<25pt>[d]|(0.35){\gamma}|{\Rightarrow}
   \ar[r]\ar[d]&[0]\ar[d]^{b'}\\
[0]\ar[r]^{b}&\B,}
\end{array}
\end{equation}
in which $\gamma$ is the lax transformation defined by $\gamma f=f$,
for any 1-cell $f:b\to b'$ in $\B$, and whose naturality component
at a 2-cell $\beta:f_0\Rightarrow f_1$, for any $f_0,f_1:b\to b'$,
is the composite 2-cell $\widehat{\gamma}_\beta=\big(1_{b'}\circ
f_0\overset{\bl} \cong f_0\overset{\beta}\Rightarrow
f_1\overset{\br^{-1}}\cong f_1\circ 1_b\big)$.

\begin{theorem}\label{xbs} The following properties of a bicategory
  $\B$ are equivalent:
\begin{itemize}
\item[$(i)$] For any diagram of bicategories
$\xymatrix@C=18pt{\A\ar[r]^-{F} & \B &\A'\ar[l]_(.4){F'}}$,
where $F$ is a lax functor and $F'$ is an oplax functor, the
induced square $(\ref{hpqsquare})$
$$
\xymatrix{
\BB(F\!\downarrow\!F')\ar@{}@<30pt>[d]|(.45){\Rightarrow}\ar[r]^-{\BB P'}\ar[d]_{\BB P}
&\BB \A'\ar[d]^{\BB F'}\\
\BB \A\ar[r]^{\BB F}&\BB\B}
$$
is a homotopy pullback.

\item[$(ii)$] Any lax functor $F:\A\to \B$ has the property
$\mathrm{B}_l$.

\item[$(iii)$] Any oplax functor $F':\A'\to \B$ has the property
$\mathrm{B}_o$.

\item[$(iv)$] For any object $b$ and $1$-cell $p:b_0\to b_1$ in
$\B$, the functor $p_*:\B(b,b_0) \to \B(b, b_1)$ induces a
homotopy equivalence on classifying spaces, $\BB\B(b,b_0) \simeq
\BB\B(b, b_1)$.

\item[$(v)$]  For any object $b$ and $1$-cell $p:b_0\to b_1$ in
$\B$, the functor $p^*:\B(b_1,b) \to \B(b_0, b)$ induces a
homotopy equivalence on classifying spaces, $\BB\B(b_1,b) \simeq
\BB\B(b_0, b)$.

\item[$(vi)$] For any two objects $b,b'\in\B$, the homotopy
  commutative square
$$
\xymatrix@C=20pt@R=20pt{
\BB \B(b,b')\ar@{}@<30pt>[d]|(.4){\BB\gamma}|(.55){\Rightarrow}
\ar[r]\ar[d]&\mathrm{pt}\ar[d]^{\BB b'}\\\mathrm{pt}\ar[r]^{\BB b}&\BB\B,
}
$$
induced by $(\ref{ssq})$, is a homotopy pullback. That is, the
whisker map
$$
\BB\B(b,b')\to
\{\gamma:I\to \BB\B \mid \gamma(0)=\BB b,\gamma(1)=\BB b'\}\subseteq \BB\B^{I}
$$
is a homotopy equivalence.
\end{itemize}
\end{theorem}
\begin{proof} The implications $(i)\Leftrightarrow (ii)\Leftrightarrow
  (iii)$ are all direct consequences of Theorem \ref{mainTheorem}. For
  the remaining implications, let us take into account that, for any
  objects $b,b'\in\B$ there is quite an obvious isomorphism of
  categories $b\!\downarrow\!b'\cong \B(b,b')$. With this
  identification in mind, we see that the homomorphism $b:[0]\to \B$ has
  the property B$_l$ (resp. B$_o$) if and only if, for any $1$-cell
  $p:b_0\to b_1$ in $\B$, the functor $p_*:\B(b,b_0) \to \B(b, b_1)$
  (resp. $p^*:\B(b_1,b) \to \B(b_0, b)$) induces a homotopy
  equivalence on classifying spaces. Therefore, the implications
  $(ii)\Rightarrow (iv)$ and $(iii)\Rightarrow (v)$ are clear.

  Furthermore, we see that the square in $(vi)$ identifies the
   square
   $$
\xymatrix@C=20pt@R=20pt{
\BB(b\!\downarrow\!b')\ar@{}@<30pt>[d]|(.45){\Rightarrow}\ar[r]^-{\BB P'}\ar[d]_{\BB P}
&\BB [0]\ar[d]^{\BB b'}\\
\BB [0]\ar[r]^{\BB b}&\BB\B.}
$$
Then, for $b$ fixed, it follows from Theorem \ref{mainTheorem} that
the square in $(vi)$ is a homotopy pullback for any $b'$ if and only
if $b:[0]\to \B$ has the property B$_l$, that is, the equivalence of
statements $(vi)\Leftrightarrow (iv)$ holds.

Finally, to complete the proof, we are going to prove that
$(iv)\Rightarrow (iii)$ and we shall leave it to the reader the
proof that $(v)\Rightarrow (ii)$ since it is parallel. By
hypothesis, for any object $b\in\Ob\B$, the normal homomorphism  $b:[0]\to\B$
has the property B$_l$. Then, by Theorem \ref{mainTheorem} $(ii)$, for
any oplax functor $F':\A'\to \B$ the square
$$
\xymatrix@C=20pt@R=20pt{
\BB(b\!\downarrow\!F')\ar[r]^-{\BB P'}\ar[d]_{\BB P}\ar@{}@<25pt>[d]|{\Rightarrow}
&\BB\A'\ar[d]^{\BB F'}\\
\BB [0]\ar[r]^-{\BB b}&\BB\B}
$$
is a homotopy pullback for any object $b\in\B$. Therefore, by Theorem \ref{mainTheorem} $(iii)$, $F'$ has the property B$_o$.
\end{proof}

We can state that

\begin{itemize}
\item[(B)]  {\em a bicategory $\B$ has the property } B if it has the
    properties in Theorem {\em \ref{xbs}}.
\end{itemize}

\vspace{0.2cm}For example, {\em bigroupoids}, that is, bicategories
whose 1-cells are invertible up to a 2-cell, and whose 2-cells are
strictly invertible, have the property B: If $\B$ is any bigroupoid,
for any object $b$ and $1$-cell $p:b_0\to b_1$ in $\B$, the functor
$p^*:\B(b_1,b) \to \B(b_0, b)$ is actually an equivalence of
categories and, therefore, induces a homotopy equivalence on
classifying spaces $\BB p^*:\BB\B(b_1,b)\simeq \BB\B(b_0,b)$. Recall
that, by the correspondence $\B\mapsto \BB\B$, bigroupoids
correspond to homotopy 2-types, that is, CW-complexes whose
$n^{\mathrm{th}}$ homotopy groups at any base point vanish for
$n\geq 3$ (see Duskin ~\cite[Theorem 8.6]{Duskin2002}).

\begin{corollary}\label{corome}
If a bicategory  $\B$ has the property $\mathrm{B}$, then, for any
object $b\in \B$, there is a  homotopy equivalence
\begin{equation}\label{loo1}
\Omega(\BB \B,\BB b)\simeq \BB\B(b,b)
\end{equation}
between the loop space of the classifying space of the bicategory
with base point $\BB b$ and the classifying space of the category of
endomorphisms of $b$ in $\B$.
\end{corollary}

The above homotopy equivalence is already known when the bicategory
is strict, that is, when $\B$ is a 2-category. It appears as a main
result in the paper by Del Hoyo ~\cite[Theorem 8.5]{Hoyo2012}, and
it was also stated at the same time by the first author in
~\cite[Example 4.4]{Cegarra2011}. Indeed, that homotopy equivalence
$(\ref{loo1})$, for the case when $\B$ is a 2-category, can be
deduced from a result by Tillmann about simplicial categories in
~\cite[Lemma 3.3]{Tillmann1997}.

\subsection{The case when both functors are lax.} For a diagram
$\xymatrix@C=18pt{\A\ar[r]^-{F} & \B
  &\C\ar[l]_(.4){G}}$, where both $F$ and $G$ are lax
functors, the comma bicategory $F\!\downarrow\!G$ is not defined
(unless $G$ is a homomorphism). However, we can obtain a
bicategorical model for the homotopy pullback of the induced maps
$\xymatrix@C=18pt{\BB\A\ar[r]^-{\BB F} & \BB\B &\BB\C\ar[l]_(.4){\BB
    G}}$ as follows: Let
$$
F\!\downarrow_{_2}\!G := F\!\downarrow\!P'
$$
be the comma bicategory defined as in $(\ref{ff'})$ by the diagram
$\A\overset{F}\longrightarrow\B\overset{P'}\longleftarrow
G\!\downarrow\!\B$, where $P'$ is the projection 2-functor
$(\ref{PP'})$ (the notation is taken from Dwyer, Kan, and Smith in
~\cite{DKS1989} and Barwick and Kan in ~\cite{BK2011,BK2013}). Thus,
$F\!\downarrow_{_2}\!G$ has 0-cells tuples $(a,f,b,g,c)$, where
$Fa\overset{f}\to b\overset{\,g}\leftarrow Gc$ are 1-cells of $\B$.
A 1-cell $$(u,\beta,p,\beta',v): (a_0,f_0,b_0,g_0,c_0)\to
(a_1,f_1,b_1,g_1,c_1)$$ in $F\!\downarrow_{_2}\!G$ consists of
1-cells $u:a_0\to a_1$, $p:b_0\to b_1$, and $v:c_0\to c_1$, in $\A$,
$\B$, and $\C$, respectively, together with 2-cells $\beta$ and
$\beta'$ of $\B$ as in the diagram
$$
\xymatrix{
Fa_0\ar[r]^{f_0}\ar[d]_{Fu}\ar@{}@<20pt>[d]|(.35){\beta}|(.5){\Leftarrow}
&b_0 \ar@{}@<20pt>[d]|(.35){\beta'}|(.5){\Rightarrow}
     \ar[d]_{p}
&\ar[l]_{g_0}Gc_0\ar[d]^{Gv}\\
Fa_1\ar[r]_{f_1}&b_1&\ar[l]^{g_1}Gc_1,}
$$
and a $2$-cell
$$\xymatrix@C=40pt{
(a_0,f_0,b_0,g_0,c_0)\ar@/^1.2pc/[r]^{(u,\beta,p,\beta',v)}
                    \ar@{}[r]|(.5){\Downarrow(\alpha,\delta,\rho)}
\ar@/_1.2pc/[r]_{(\bar{u},\bar{\beta},\bar{p},\bar{\beta}',\bar{v})}
&(a_1,f_1,b_1,g_1,c_1),}
$$
is given by $2$-cells $\alpha:u\Rightarrow \bar{u}$ in $\A$,
$\delta:p\Rightarrow \bar{p}$ in $\B$, and $\rho:v\Rightarrow
\bar{v}$ in $\C$, such that the diagrams below commute.
$$
\begin{array}{cc}
\xymatrix@C=35pt{
p\circ f_0\ar@2[r]^{\delta \circ 1}\ar@2[d]_{\beta}
  & \bar{p}\circ f_0\ar@2[d]^{\bar{\beta}}
\\ f_1\circ Fu\ar@2[r]^{1\circ\, F\!\alpha} & f_1\circ F\bar{u}}&
\xymatrix@C=35pt{p\circ g_0\ar@2[r]^{\delta \circ 1}\ar@2[d]_{\beta'}
  & \bar{p}\circ g_0\ar@2[d]^{\bar{\beta}'}
\\ g_1\circ Gv\ar@2[r]^{1\circ\, G\!\rho} & g_1\circ G\bar{v}}
\end{array}
$$

There is a (non-commutative) square
\begin{equation}\label{pqfgsquare}\begin{array}{c}
\xymatrix{
F\!\downarrow_{_2}\!G\ar[r]^-{Q}\ar[d]_{P} &\C\ar[d]^{G}\\ \A\ar[r]^{F}&\B}
\end{array}
\end{equation}
where $P$ and $Q$ are projection 2-functors, which act on cells of
$F\!\downarrow_{_2}\!G$ by
$$
\xymatrix{
a_0\ar@/^8pt/[r]^{u}\ar@/_8pt/[r]_{\bar{u}}
       \ar@{}[r]|{\Downarrow\alpha}&a_1}
\xymatrix@C=12pt{&\ar@{|->}[l]_{P}}
 \xymatrix@C=3pc{
(a_0,f_0,b_0,g_0,c_0)\ar@/^1.2pc/[r]^{(u,\beta,p,\beta',v)}
          \ar@{}[r]|(.5){\Downarrow (\alpha,\delta,\rho)}
       \ar@/_1.2pc/[r]_{(\bar{u},\bar{\beta},\bar{p},\bar{\beta}',\bar{v})}
&(a_1,f_1,b_1,g_1,c_1)}
\xymatrix@C=12pt{\ar@{|->}[r]^{Q}&}
\xymatrix{
c_0\ar@/^8pt/[r]^{v}\ar@/_8pt/[r]_{\bar{v}}\ar@{}[r]|{\Downarrow\rho}&c_1,}
$$
and we have the result given below.
\begin{theorem} Let $\xymatrix@C=18pt{\A\ar[r]^-{F} & \B
    &\C\ar[l]_(.4){G}}$ be a diagram where $F$ and $G$ are lax
  functors.

$(i)$ There is a homotopy $\BB F\, \BB P\Rightarrow \BB G\,\BB Q$ so
that the square below, which is induced by
   $(\ref{pqfgsquare})$ on classifying spaces, is homotopy commutative.
$$
\xymatrix{
\BB(F\!\downarrow_{_2}\!G)\ar@{}@<30pt>[d]|(.45){\Rightarrow}\ar[r]^-{\BB Q}\ar[d]_{\BB P}
&\BB \C\ar[d]^{\BB G}\\ \BB \A\ar[r]^{\BB F}&\BB \B}
$$

$(ii)$ The square above is a homotopy pullback whenever $F$ or $G$
has property $\mathrm{B}_l$.
\end{theorem}
\begin{proof} The part $(i)$ follows from  Theorem \ref{mainTheorem} $(i)$ and the definition of $F\!\downarrow_{_2}\!G$. For the part $(ii)$, since $F\!\downarrow_{_2}\!G \cong
  G\!\downarrow_{_2}\!F$, it is enough, by symmetry, to prove the
  theorem when $F$ has the property B$_l$. In this case, we have the
  homotopy commutative diagram
$$\xymatrix@R=25pt@C=20pt{
\BB(F\!\downarrow_{_2}\!G)
\ar@{}[rd]|(.46){\mathrm{(\ref{hpqsquare})}}
\ar@{}@<13pt>[rr]|{=}
\ar[d]_{\BB P}
\ar@/^2pc/[rr]^{\BB Q}\ar[r]^{\BB P'}& \BB(G\!\downarrow\!\B)
\ar@{}[rd]|(.46){\mathrm{(\ref{hpqsquare})}}
\ar[d]_{\BB P'}\ar[r]^{\BB P}&\BB\C\ar[d]^{\BB G}\\
\BB\A \ar@{}@<-12pt>[rr]|{=}\ar@/_1.7pc/[rr]_{\BB F}\ar[r]^{\BB F}
&\BB\B\ar[r]^{\BB 1_\B}&\BB\B
}
$$
where, by Theorem \ref{mainTheorem}, the inner squares
$(\ref{hpqsquare})$ are both homotopy pullback. Then, the outside
square is also a homotopy pullback, as claimed.
\end{proof}

\section{Homotopy pullbacks of monoidal categories.}
\label{monoidalCategories}

Recall ~\cite{Saavedra1972,MacLane1998} that a {\em monoidal
category} $(\M,\otimes)=
(\M,\otimes,\text{I},\boldsymbol{a},\boldsymbol{l},\boldsymbol{r})$
consists of a category $\M$ equipped with a 
tensor product $\otimes:\M\times\M\to\M$, a unit object $\text{I}$,
and natural and coherent isomorphisms $\aso:(m_3\otimes m_2)\otimes
m_1\cong m_3\otimes (m_2\otimes m_1)$,
$\boldsymbol{l}:\text{I}\otimes m\cong m$, and ${\br:m\otimes
\text{I}\cong m}$. Any monoidal category $(\M,\otimes)$ can be
viewed as a bicategory $\Sigma\M$ with only one object, say $*$, the
objects $m$ of $\M$ as 1-cells $m:*\rightarrow *$, and the morphisms
of $\M$ as 2-cells. Thus, $\Sigma\M(*,*)=\M$, and the horizontal
composition of cells is given by the tensor functor. The identity at
the object is $1_*=\text{I}$, the unit object of the monoidal
category, and the associativity, left unit and right unit
constraints for $\Sigma\M$ are precisely those of the monoidal
category, that is, $\boldsymbol{a}$, $\boldsymbol{l}$, and
$\boldsymbol{r}$, respectively. Furthermore, a monoidal functor
$F=(F,\widehat{F}):(\N,\otimes)\to (\M,\otimes)$ amounts precisely
to a homomorphism $\Sigma F:\Sigma\N\to \Sigma\M$.

 For any monoidal category $(\M,\otimes)$, the Grothendieck nerve
 $(\ref{GrothendieckNerve})$ of the bicategory $\Sigma\M$ is exactly
 the pseudo-simplicial category that the monoidal category defines by
 the reduced bar construction (see Jardine ~\cite[Corollary
 1.7]{Jardine1991}), whose category of $p$-simplices is $\M^p$, the
 $p$-fold power of the underlying category $\M$. Therefore, the {\em
   classifying space of the monoidal category} $\BB(\M,\otimes)$
 ~\cite[\S 3]{Jardine1991} is the same as the classifying space
 $\BB\Sigma\M$ of the one-object bicategory it defines ~\cite{BC2004},
 and thus the bicategorical results  obtained above are
 applicable to monoidal functors between monoidal categories. This,
 briefly, can be done as follows:

Given any diagram  $\xymatrix@C=18pt{(\N,\otimes)\ar[r]^-{F} &
(\M,\otimes) &(\N',\otimes)\ar[l]_(.4){F'}}$, where $F$ and $F'$ are
monoidal functors between monoidal categories,  the  ``{\em homotopy- fibre product bicategory}"
\begin{equation}
F\!\overset{_{\otimes}}\downarrow\!F'
\end{equation}
(the notation $\overset{_{\otimes}}\downarrow$ is to avoid confusion with the comma category $F\!\downarrow\!F'$ of the underlying functors) has as $0$-cells the
objects $m\in \M$. A 1-cell $(n,f,n'):m_0\to m_1$ of
$F\!\overset{_{\otimes}}\downarrow\!F'$ consists of objects $n\in \N$ and $n'\in \N'$,
and a morphism $f: F'n'\otimes m_0 \to m_1\otimes Fn$ in $\B$. A
$2$-cell in $F\!\overset{_{\otimes}}\downarrow\!F'$,
$$
\xymatrix@C=40pt{m_0\ar@/^1pc/[r]^{(n,f,n')}
    \ar@{}[r]|(.5){\Downarrow(u,u')}
   \ar@/_1pc/[r]_{(\bar{n},\bar{f},\bar{n}')}&m_1,}
$$
is given by a pair of morphisms, $u:n\to \bar{n}$ in $\N$ and
$u':n'\to \bar{n}'$ in $\N'$, such that the
 diagram below commutes.
$$
\xymatrix@C=35pt@R=20pt{F'n'\otimes m_0\ar[r]^{F'\!u'\otimes\, 1}\ar[d]_{f}
  & F'\bar{n}'\otimes m_0\ar[d]^{\bar{f}}
\\ m_1\otimes Fn\ar[r]^{1\,\otimes F\!u} & m_1\circ F\bar{n}}
$$
The vertical composition of $2$-cells is given by the composition of morphisms in
$\N$ and $\N'$. The horizontal composition of the  1-cells 
$\xymatrix@C=40pt{m_0\ar[r]^{(n_1,f_1,n'_1)}&m_1\ar[r]^{(n_2,f_2,n'_2)}
& m_2}$ is the 1-cell
$$
(n_2\otimes n_1, f_2\circledcirc f_1,n'_2\otimes n'_1):m_0\to m_2,
$$
$$
\begin{array}{cl}
f_2\circledcirc f_1=\big(&\hspace{-0.3cm}F'(n'_2\otimes n'_1)\otimes m_0
\overset{\widehat{F}'^{-1}\otimes 1}\cong
 (F'n'_2\otimes F'n'_1)\otimes m_0 \overset{\aso} \cong
 F'n'_2\otimes (F'n'_1\otimes m_0) \overset{1\otimes f_1}\longrightarrow\\
 &\hspace{-0.3cm}F'n'_2\otimes (m_1\otimes Fn_1)\overset{\aso^{-1}}\cong
 (F'n'_2\otimes m_1)\otimes Fn_1  \overset{f_2\otimes 1}\longrightarrow
 (m_2\otimes Fn_2)\otimes Fn_1 \overset{\aso}\cong\\
 &\hspace{-0.3cm}m_2\otimes (Fn_2\circ Fn_1)\overset{1\otimes
   \widehat{F}}
\cong   m_2\otimes F(n_2\otimes n_1)\big),
\end{array}
$$
and the horizontal composition of $2$-cells is given by the tensor
product of morphisms in $\N$ and $\N'$.  The identity $1$-cell, at
any 0-cell $m$, is $(\text{I},\overset{_\circ}{1}_{m},\text{I}):m\to
m$, where
$$ \overset{_\circ}{1}_{m}=\big(F'\text{I}\otimes
m \overset{\widehat{F}'^{-1}\otimes 1}\cong \text{I}\otimes
m\overset{\bl}\cong m \overset{\br^{-1}}\cong m\otimes
\text{I}\overset{1\otimes \widehat{F}}\cong m\otimes
F\text{I}\big).$$ The associativity, right, and left unit constraints
of the bicategory $F\!\overset{_{\otimes}}\downarrow \!F'$ are provided by those of $\N$
and $\N'$ by the formulas
$$
\aso_{(n_3,f_3,n'_3),(n_2,f_2, n'_2),(n_1,f_1,n'_1)}\!=
\!(\aso_{n_3,n_2,n_1},\aso_{n'_3,n'_2,n'_1})
,\,
\bl_{(n,f,n')}\!=\!(\bl_n,\bl_{n'}),\ \br_{(n,f,n')}\!=\!(\br_n,\br_{n'}).
$$

\begin{remark} {\em
  Let us stress that $F\!\overset{_{\otimes}}\downarrow \!F'$ is not a monoidal category
  but a genuine bicategory, since it generally has more than one
  object.
  }
\end{remark}

In particular, for any monoidal functor $F:(\N,\otimes)\to
(\M,\otimes)$, we have the {\em homotopy-fibre bicategories} (cf.
~\cite{BC2003})
\begin{equation}F\!\overset{_{\otimes}}\downarrow \!\text{I},
  \hspace{0.3cm}\text{I}\!\overset{_{\otimes}}\downarrow \!F
\end{equation}
where we denote by $\text{I}:([0],\otimes)\to (\M,\otimes)$ the
monoidal functor that carries $0$ to the unit object $\text{I}$, and
whose structure isomorphism is
$\bl_{\text{I}}=\br_{\text{I}}:\text{I}\otimes\text{I}\cong
\text{I}$. Every object $m\in \M$ determines $2$-endofunctors
$$
m\otimes -:F\!\overset{_{\otimes}}\downarrow \!\text{I} \to F\!\overset{_{\otimes}}\downarrow \!\text{I},
 \hspace{0.4cm} -\otimes m:
\text{I}\!\overset{_{\otimes}}\downarrow \!F\to \text{I}\!\overset{_{\otimes}}\downarrow \!F,
$$
respectively given on cells by
$$
\xymatrix@C=1.8pc{
m_0\ar@/^10pt/[r]^{(n,f)}\ar@/_10pt/[r]_{(\bar{n},\bar{f})}
\ar@{}[r]|{\Downarrow u}
&m_1}
\xymatrix@C=16pt{\ar@{|->}[r]^{m\otimes -}&}
\xymatrix@C=1.8pc{m\!\otimes m_0\ar@/^11pt/[r]^{(n, m\odot f)}
\ar@/_11pt/[r]_{(\bar{n},m\odot \bar{f})}
\ar@{}[r]|{\Downarrow u}&
m\!\otimes m_1,}\
\xymatrix@C=2pc{
m_0\ar@/^11pt/[r]^{(g,n')}\ar@/_11pt/[r]_{(\bar{g},\bar{n}')}
\ar@{}[r]|{\Downarrow u'}
&m_1}
\xymatrix@C=16pt{\ar@{|->}[r]^{-\otimes m}&}
\xymatrix@C=1.8pc{m_0\!\otimes m\ar@/^10pt/[r]^{(g\odot m, n')}
\ar@/_10pt/[r]_{(\bar{g}\odot m,\bar{n}')}
\ar@{}[r]|{\Downarrow u'}&
m_1\!\otimes m,}
$$
where, for any $(n,f):m_0\to m_1$ in $F\!\overset{_{\otimes}}\downarrow\!\text{I}$ and
$(g,n):m_0\to m_1$ in $\text{I}\!\overset{_{\otimes}}\downarrow\!F$,

\noindent $$ m\odot f\!=\!\big(\text{I}\otimes (m\otimes
m_0)\overset{\bl}\cong m\otimes m_0\! \overset{1\otimes
\bl^{-1}}\cong\! m\otimes (\text{I}\otimes m_0) \overset{1\otimes
f}\longrightarrow m\otimes (m_1\otimes
Fn)\!\overset{\aso^{-1}}\cong\! (m\otimes m_1)\otimes Fn\big),$$

\noindent $$ g\odot m=\big(Fn\otimes (m_0\otimes
m)\!\overset{\aso^{-1}}\cong \!(Fn\otimes m_0)\otimes m\!
\overset{g\otimes 1}\longrightarrow (m_1\otimes \text{I} )\otimes m
\overset{\br\otimes 1}\cong m_1\otimes m\!\overset{\br^{-1}}\cong
(m_1\otimes m)\otimes \text{I}\big).$$

\vspace{0.2cm} We state that
\begin{itemize}
\item[(B$_l$)] {\em the monoidal functor $F$ has the property} B$_l$
    if, for any object $m\in \M $, the induced map $\BB(m\otimes
  -): \BB(F\!\overset{_{\otimes}}\downarrow\!\text{I}) \to
  \BB(F\!\overset{_{\otimes}}\downarrow\!\text{I})$ is a homotopy
    autoequivalence.

\item[(B$_o$)] {\em the monoidal functor $F$ has the property} B$_o$
    if, for any object $m\in \M $, the induced map $\BB(-\otimes
  m): \BB(\text{I}\!\overset{_{\otimes}}\downarrow\!F) \to
  \BB(\text{I}\!\overset{_{\otimes}}\downarrow\!F)$ is a homotopy autoequivalence.
\end{itemize}

Our main result here is a direct consequence of Theorem
\ref{mainTheorem}, after taking into account the identifications
$\BB(\M,\otimes)=\BB\Sigma\M$, $\BB F=\BB\Sigma F$,
$F\!\overset{_{\otimes}}\downarrow\!F'= \Sigma F\!\downarrow\!\Sigma F'$,
$F\!\overset{_{\otimes}}\downarrow\!\text{I}= \Sigma F\!\downarrow\!*$, and
$\text{I}\!\overset{_{\otimes}}\downarrow\!F= *\!\downarrow\!\Sigma F$, and the fact
that a monoidal functor has the property B$_l$ or B$_o$ if and only
if the homomorphism $\Sigma F$ has that property. This result is as
given below.

\begin{theorem}\label{maintheomon}
$(i)$ Suppose  $\xymatrix@C=18pt{(\N,\otimes)\ar[r]^-{F} &
(\M,\otimes) &(\N',\otimes)\ar[l]_(.4){F'}}$ are monoidal functors
between monoidal categories, such that  $F$ has the property
$\mathrm{B}_l$ or $F'$ has the property $\mathrm{ B}_o$. Then, there
is an induced homotopy pullback square
\begin{equation}\label{hpqsquaremon}\begin{array}{c}
\xymatrix@C=20pt@R=20pt{
\BB(F\!\overset{_{\otimes}}\downarrow\!F')\ar@{}@<35pt>[d]|(.4){\Rightarrow}\ar[r]^-{\BB P'}\ar[d]_{\BB P}
&\BB(\N',\otimes)\ar[d]^{\BB F'}\\
\BB (\N,\otimes)\ar[r]^{\BB F}&\BB(\M,\otimes).}
\end{array}
\end{equation}
Therefore, there is an induced Mayer-Vietoris type long exact
sequence on homotopy groups, based at the $0$-cells $\BB *$ of
$\BB(\M,\otimes)$, $\BB(\N,\otimes)$, and $\BB(\N',\otimes)$
respectively, and the $0$-cell $\BB\mathrm{I}\in
\BB(F\!\overset{_{\otimes}}\downarrow\!F')$,
$$
\xymatrix@C=12pt{\cdots\to \pi_{n+1}\BB(\M,\otimes)\ar[r]&
\pi_n\BB(F\!\overset{_{\otimes}}\downarrow\!F')\ar[r]&
\pi_n\BB(\N,\otimes)\times\pi_n\BB(\N',\otimes)
\ar[r]&\pi_n\BB(\M,\otimes) \to}
$$
$$
\xymatrix@C=12pt{
\cdots\to \pi_1\BB(F\!\overset{_{\otimes}}\downarrow\!F')\ar[r]
&\pi_1\BB(\N,\otimes)\times\pi_1\BB
(\N',\otimes)
\ar[r]&\pi_1\BB(\M,\otimes)\ar[r]&\pi_0\BB(F\!\overset{_{\otimes}}\downarrow\!F')\ar[r]&0.}
$$

$(ii)$  Given a monoidal functor $F:(\N,\otimes)\to (\M,\otimes)$,
if the square $(\ref{hpqsquaremon})$ is a homotopy pullback for
every monoidal functor $F':(\N',\otimes)\to (\M,\otimes)$, then $F$
has the property $\mathrm{B}_l$. Similarly, if $F'$ is a  monoidal
functor  such that the square $(\ref{hpqsquaremon})$ is a homotopy
pullback for any monoidal functor $F$, as above, then $F'$  has the
property $\mathrm{B}_o$.
\end{theorem}

Similarly, from Corollaries \ref{hfth} and \ref{tagen}, we get the
following extensions of Quillen's Theorems A and B to monoidal
functors:

\begin{theorem}\label{qmab} Let $F:(\N,\otimes)\to (\M,\otimes)$ be any monoidal
  functor.

$(i)$ If $F$  has the property $\mathrm{B}_l$, then there is an
induced homotopy fibre sequence {\em $$ \xymatrix@C=15pt{
\BB(F\!\overset{_{\otimes}}\downarrow\!\text{I})\ar[r] &\BB(\N,\otimes)\ar[r]&\BB
(\M,\otimes).}
$$}
$(ii)$ If $F$ has the property $\mathrm{B}_o$, then there is an
induced homotopy fibre sequence {\em $$ \xymatrix@C=15pt{
\BB(\text{I}\!\overset{_{\otimes}}\downarrow\!F)\ar[r] &\BB(\N,\otimes)\ar[r]&\BB
(\M,\otimes).}
$$}
$(iii)$ If the classifying space of any of the two homotopy-fibre
bicategories of $F$ is contractible, that is, if
{\em $\BB(F\!\overset{_{\otimes}}\downarrow\!\text{I})\simeq \mathrm{pt}$} or
{\em $\BB(\text{I}\!\overset{_{\otimes}}\downarrow\!F)\simeq \mathrm{pt}$},  then  the induced map on
classifying spaces $\BB F:\BB(\N,\otimes)\simeq \BB
(\M,\otimes)$ is a homotopy
equivalence.
\end{theorem}

For the last statement in the following theorem, let us note that
there is a diagram of bicategories
\begin{equation}\label{ssqm}\begin{array}{c}
    \xymatrix@C=15pt@R=15pt{
\M\ar@{}@<20pt>[d]|(0.3){\gamma}|{\Rightarrow}\ar[r]\ar[d]
&[0]\ar[d]^{*}\\
      [0]\ar[r]^{*}&\Sigma \M}
\end{array}
\end{equation}
in which $\gamma$ is the lax transformation defined by $\gamma
m=m:*\to*$, for any object $m\in\M$, and whose naturality component
at a morphism $f:m_0\to m_1$, is the composite 2-cell
$\widehat{\gamma}_f=\big(\text{I}\otimes m_0\overset{\bl} \cong
m_0\overset{f}\Rightarrow m_1\overset{\br^{-1}}\cong m_1\circ
\text{I}\big)$. Then, we have an induced homotopy commutative square
on classifying spaces
$$
\xymatrix@C=15pt@R=15pt{
\BB \M \ar@{}@<25pt>[d]|(0.35){\BB\gamma}|{\Rightarrow}
\ar[r]\ar[d]&\mathrm{pt}\ar[d]^{*}\\ \mathrm{pt}\ar[r]^-{*}&\BB(\M,\otimes)
}
$$
and a corresponding whisker map
\begin{equation}\label{whism}
\BB\M\to \Omega(\BB(\M,\otimes),*).
\end{equation}

Theorem \ref{xbs} particularizes by giving

\begin{theorem} \label{xbsmon} The following properties of a monoidal
  category $(\M,\otimes)$ are equivalent:
\begin{itemize}
\item[$(i)$] For any diagram of monoidal functors
 $\xymatrix@C=18pt{(\N,\otimes)\ar[r]^-{F} & (\M,\otimes)
&(\N',\otimes)\ar[l]_(.4){F'}}$, the induced square
$(\ref{hpqsquaremon})$
$$
\xymatrix{
\BB(F\!\overset{_{\otimes}}\downarrow\!F')\ar@{}@<35pt>[d]|(.45){\Rightarrow}\ar[r]^-{\BB P'}\ar[d]_{\BB P}
&\BB(\N',\otimes)\ar[d]^{\BB F'}\\
\BB (\N,\otimes)\ar[r]^{\BB F}&\BB(\M,\otimes).}
$$
is a homotopy pullback.

\item[$(ii)$] Any monoidal funtor $F:(\N,\otimes)\to (\M,\otimes)$ has
  property $\mathrm{B}_l$.

\item[$(iii)$]Any monoidal functor $F:(\N,\otimes)\to (\M,\otimes)$
  has property $\mathrm{B}_o$.

\item[$(iv)$] For any object $m\in \M$, the functor $m\otimes-:\M \to
  \M$ induces a homotopy autoequivalence on the classifying
  space $\BB\M$.

\item[$(v)$] For any object $m\in \M$, the functor $-\otimes m:\M \to
  \M$ induces a homotopy autoequivalence on the classifying
  space $\BB\M$.

\item[$(vi)$] The whisker map $(\ref{whism})$ is a homotopy equivalence
$$
\BB\M\simeq \Omega(\BB(\M,\otimes),*)
$$
between the classifying space of the underlying category and the
loop space of the classifying space of the monoidal category.
\end{itemize}
\end{theorem}

The implications $(iv)\Rightarrow (vi)$ and $(v)\Rightarrow (vi)$
in the above theorem are essentially due to Stasheff
~\cite{Stasheff1963}, but several other proofs can be found in the
literature (see Jardine ~\cite[Propositions 3.5 and
3.8]{Jardine1991}, for example). When the equivalent properties in
Theorem \ref{xbsmon} hold, we  say that {\em the monoidal category
  is homotopy regular}. For example, {\em regular monoidal categories}
(as termed by Saavedra ~\cite[Chap. I, (0.1.3)]{Saavedra1972}), that
is, monoidal categories $(\M,\otimes)$ where, for every object $m\in
\M$, the functor $m\otimes -:\M\to \M$ is an autoequivalence of the
underlying category $\M$, and, in particular, {\em categorical
groups}
 (so named by Joyal and Street in ~\cite[Definition 3.1]{JS1993} and also termed {\em
   Gr-categories} by Breen in ~\cite[\S 2, 2.1]{Breen1992}),
that is, monoidal categories whose objects are invertible up to an
isomorphism, and whose morphisms are all invertible, are homotopy
regular.

\section{Homotopy pullback of crossed module morphisms}
\label{crossedModules}

Thanks to the equivalence between the category of crossed modules
and the category of 2-groupoids, the results in Section
\ref{mainSection} can be applied to crossed modules. To do so in
some detail, we shall start by briefly reviewing crossed modules and
their classifying spaces.

Recall that, if $\PP$ is any (small) groupoid, then the category of
(left) {\em $\PP$-groups} has objects the functors $\PP\to
\mathbf{Gp}$, from $\PP$ into the category of groups, and its
morphisms, called {\em $\PP$-group homomorphisms}, are natural
transformations. If $\GG$ is a $\PP$-group, then, for any arrow
$p:a\to b$ in $\PP$, we write the associated group homomorphism
$\GG(a)\to \GG(b)$ by $g\mapsto {}^pg$, so that the equalities $
^1g=g$ ,${}^{(q\circ p)}g={}^q({}^pg)$, and ${}^p(g\cdot
g')={}^{p}g\cdot{}^pg'$ hold whenever they make sense. Here, the
symbol $ \circ $ denotes composition in the groupoid $\PP$, whereas
$ \cdot $ denotes multiplication in $\GG$.  For instance, the
assignment to each object of $\PP$ its isotropy group, $a\mapsto
\mathrm{Aut}_\PP(a)$, is the function on objects of a $\PP$-group $
\mathrm{Aut}_\PP:\PP\to\mathbf{Gp} $ such that $^pq=p\circ q\circ
p^{-1}$, for any $p:a\to b$ in $\PP$ and $q\in \mathrm{Aut}_\PP(a)$.
Then, a {\em crossed module} (of groupoids) is a triplet
$$(\GG,\PP,\partial)$$
consisting of a groupoid $\PP$, a $\PP$-group $\GG$, and a
$\PP$-group homomorphism $\partial:\GG\to \mathrm{Aut}_\PP$, called
the {\em
  boundary map}, such that the Peiffer identity $ ^{\partial
  g}g'=g\cdot g'\cdot g^{-1}$ holds, for any $g,g'\in \GG(a)$, $a\in
\Ob{\PP}$.

When a group $P$ is regarded as a groupoid $\PP$ with exactly one
object, the above definition by Brown and Higgins
~\cite{BH1981algebra} recovers the more classic notion of crossed
module $(G,P,\partial)$ due to Whitehead and Mac Lane
~\cite{Whitehead1949,MW1950}, now called {\em crossed modules
  of groups}.  In fact, if $(\GG,\PP,\partial)$ is any crossed
module, then, for any object $a$ of $\PP$, the triplet $(\GG(a),
\mathrm{Aut}_\PP(a),\partial_a)$ is precisely a crossed module of
groups.

Composition with any given functor $F:\PP\to \QQ$ defines a functor
from the category of $\QQ$-groups to the category of $\PP$-groups: $
(\varphi:\GG\to \mathcal{H})\mapsto (\varphi F:\GG F \to
\mathcal{H}F) $. For the particular case of the $\QQ$-group of
automorphisms $\mathrm{Aut}_{\QQ}$, we have the $\PP$-group
homomorphism $ F:\mathrm{Aut}_{\PP}\to \mathrm{Aut}_{\!\QQ}\,F $,
which, at any $a\in\PP$, is given by the map
$\mathrm{Aut}_{\PP}(a)\to \mathrm{Aut}_{\QQ}(F a)$, $q\mapsto F q$,
defined by the functor $F$. Then, a {\em morphism of crossed
modules}
$$(\varphi,F):(\GG,\PP,\partial)\to (\mathcal{H},\QQ,\partial)$$
consists of a functor $F:\PP\to \QQ$ together with a $\PP$-group
homomorphism $\varphi:\GG\to \HH F$ such that the square below
commutes.
$$
\xymatrix@C=25pt@R=20pt{
\GG\ar[r]^{\partial}\ar[d]_{\varphi}&\mathrm{Aut}_\PP\ar[d]^{F}\\
\HH F\ar[r]^-{\partial F}&\mathrm{Aut}_\QQ F.
}
$$

The category of crossed modules, where compositions and identities
are defined in the natural way, is denoted by $\mathbf{Xmod}$. Let
us  now recall from Brown and Higgins ~\cite[Theorem
4.1]{BH1981equivalence} that there is an equivalence between the
category of crossed modules and the category of 2-groupoids
\begin{equation}\label{equ2}
\beta: \mathbf{Xmod}\overset{_\sim}\longrightarrow
  2\text{-}\mathbf{Gpd},\end{equation}
which is as follows: Given any crossed module $(\GG,\PP,\partial)$,
$\PP$ is the underlying groupoid of the 2-groupoid
$\beta(\GG,\PP,\partial)$, whose 2-cells
$$
\xymatrix{
a_0\ar@/^0.7pc/[r]^{p}\ar@/_0.7pc/[r]_{\bar{p}}\ar@{}[r]|{\Downarrow\,g} &a_1 }
$$
are those elements $g\in \GG(a_0)$ such that $\bar{p}\circ \partial
g=p$. The vertical and horizontal composition of 2-cells are,
respectively, given by
$$
\begin{array}{cc}
\xymatrix@C=40pt{
a_0\ar[r]|{\bar{p}}\ar@/^1.2pc/[r]^{p}
             \ar@{}@<9pt>[r]|(.55){\Downarrow g}
             \ar@{}@<-8pt>[r]|(.55){\Downarrow \bar{g}}
  \ar@/_1.2pc/[r]_{\bar{\bar{p}}}&a_1}\
\overset{\cdot}\mapsto
\xymatrix@C=30pt{a_0\ar@/^1pc/[r]^{p}
        \ar@{}[r]|(.55){\Downarrow \bar{g}\cdot g}
        \ar@/_1pc/[r]_{\bar{\bar{p}}}&a_1},
&
\xymatrix@C=25pt{
 a_0\ar@{}[r]|(.55){\Downarrow g_1}\ar@/^0.7pc/[r]^{p_1}
  \ar@/_0.7pc/[r]_{\bar{p}_1}&a_1 \ar@{}[r]|(.55){\Downarrow
    g_2}\ar@/^0.7pc/[r]^{p_2}\ar@/_0.7pc/[r]_{\bar{p}_2}&a_2} \
\overset{\circ}\mapsto \ \xymatrix@C=45pt{a_0\ar@{}[r]|{\Downarrow
    {}^{\bar{p}_1^{-1}}\!\!\!g_2\cdot g_1} \ar@/^0.8pc/[r]^{p_2\circ
    p_1}\ar@/_0.9pc/[r]_{\bar{p}_2\circ \bar{p}_1}&a_2.  }
\end{array}
$$

A morphism of crossed modules $(\varphi,F):(\GG,\PP,\partial)\to
(\mathcal{H},\QQ,\partial)$ is carried by the equivalence to the
2-functor $\beta(\varphi,F):\beta(\GG,\PP,\partial)\to
\beta(\mathcal{H},\QQ,\partial)$ acting on cells by
$$
\xymatrix{a_0\ar@/^0.7pc/[r]^{p}\ar@/_0.7pc/[r]_{\bar{p}}
             \ar@{}[r]|{\Downarrow g} & a_1}
 \mapsto
\xymatrix{F a_0\ar@/^0.7pc/[r]^{F p}\ar@/_0.7pc/[r]_{F \bar{p}}
     \ar@{}[r]|{\Downarrow \varphi g} & F a_1.}
$$

\begin{example}\label{becm}{\em A striking  example of crossed
module is $\Pi(X,A,S)=(\pi_2(X,A),\pi(A,S),\partial)$, which comes
associated to any triple $(X,A,S)$, where $X$ is any topological
space, $A\subseteq X$ a subspace, and $S\subseteq A$ a set of (base)
points. Here, $\pi(A,S)$ is the fundamental groupoid of homotopy
classes of paths in $A$ between points in $S$, $\pi_2(X,A):
\pi(A,S)\to\mathbf{Gp}$ is the functor associating to each
 $a\in S$ the relative homotopy group
$\pi_2(X,A,a)$, and, at any $a\in S$, the boundary map $
\partial:\pi_2(X,A,a)\to \pi_1(A,a)$
is the usual boundary homomorphism in the exact sequence of homotopy
groups based at $a$ of the pair $(X,A)$:
$$\begin{bmatrix}\xymatrix@C=12pt@R=10pt{
a\ar@{=}[d]\ar[r]^{u}
\ar@{}@<12pt>[d]|{g}
&a\ar@{=}[d]\\
a\ar@{=}[r]&a
}\end{bmatrix}\overset{\partial}\mapsto\hspace{4pt}
\xymatrix@C=15pt{a\ar[r]^{[u]}&a.}
$$

 Furthermore, $\pi(A,S)$ is the underlying groupoid of the {\em Whitehead
$2$-groupoid} $W(X,A,S)$ presented by Moerdijk and Svensson
~\cite{MS1993}, whose 2-cells
$$\begin{bmatrix}\xymatrix@C=12pt@R=8pt{
a\ar@{=}[d]\ar[r]^{v}
\ar@{}@<12pt>[d]|{g}
&b\ar@{=}[d]\\
a\ar[r]_{w}&b
}\end{bmatrix}: [v]\Rightarrow [w]:a\to b,
$$
are equivalence classes of maps $g:I\times I\to X$, from the square
$I\times I$ into $X$, which are constant along the vertical edges
with values in $S$, and map the horizontal edges into $A$; two such
maps
 are equivalent if they are homotopic by a homotopy that
is constant along the vertical edges and deforms the horizontal
edges within $A$.

Both constructions  $\Pi(X,A,S)$ and
 $W(X,A,S)$ correspond to each other by the equivalence of categories
$(\ref{equ2})$.  More precisely, there is  a natural isomorphism
 \begin{equation}\label{betpiw} \beta \Pi(X,A,S)\cong
 W(X,A,S),\end{equation}
 which is the identity on 0- and 1-cells, and carries a
2-cell $[g]:[v]\Rightarrow [w]$ of $\beta\Pi(X,A,S)$ to the 2-cell
$1_{[w]}\circ [g]:[v]\Rightarrow [w]$ of $ W(X,A,S)$:
$$\begin{pmatrix}\begin{bmatrix}\xymatrix@R=15pt@C=15pt{
a\ar@{=}[d]\ar[r]^{u}
\ar@{}@<15pt>[d]|{_{g(s,t)}}
&a\ar@{=}[d]\\
a\ar@{=}[r]&a
}\end{bmatrix}:
\xymatrix@C=11pt{[v]\ar@2[r]&[w]}\hspace{-5pt}\end{pmatrix}\hspace{4pt}
\mapsto
\hspace{4pt}
\begin{pmatrix}\begin{bmatrix}\xymatrix@C=19pt@R=15pt{
a\ar@{=}[d]\ar[r]^{u}
\ar@{}@<15pt>[d]|{_{g(2s,t)}}
&
a\ar[r]^{w}\ar@{=}[d]
\ar@{}@<15pt>[d]|{_{w(2s\text{-}1)}}
&
b\ar@{=}[d]\\
a\ar@{=}[r]&a\ar[r]_{w}&b
}\end{bmatrix}:
\xymatrix@C=11pt{[v]\ar@2[r]&[w]}\hspace{-5pt}
\end{pmatrix}.
$$
}\end{example}

For a simplicial set $K$, its  {\em fundamental}, or {\em homotopy},
{\em crossed module} $\Pi(K)$ is defined as the crossed module
\begin{equation}
\Pi(K)=\Pi\big(|K|, |K^{(1)}|,|K^{(0)}|\big)
\end{equation}
constructed in Example \ref{becm} (here, $K^{(n)}$ denotes the
$n$-skeleton, as usual). The construction $K\mapsto \Pi(K)$ gives
rise to a functor $\Pi:\mathbf{SimpSet}\to\mathbf{Xmod}$, from the
category of simplicial sets to the category of crossed modules. To
go in the other direction, we have the notion of {\em nerve of a
crossed
  module}, which is actually a special case of the definition of nerve
for crossed complexes by Brown and Higgins ~\cite{BH1991}. Thus, the
{\em nerve} $N(\GG,\PP,\partial)$ of a crossed module
$(\GG,\PP,\partial)$ is the simplicial set
\begin{equation}\label{nervcm}
N(\GG,\PP,\partial):\Delta^{op}\longrightarrow \Set,\hspace{0.4cm}
 [n]\mapsto \mathbf{Xmod}\big(\Pi(\Delta[n]),(\GG,\PP,\partial)\big),
\end{equation}
whose $n$-simplices are all morphisms of crossed
 modules $\Pi(\Delta[n])\to (\GG,\PP,\partial)$.

The {\em classifying space} $\BB(\GG,\PP,\partial)$ of a crossed
module $(\GG,\PP,\partial)$ is the  geometric realization of its
nerve, that is,
\begin{equation}\BB(\GG,\PP,\partial)=|N(\GG,\PP,\partial)|.\end{equation}
By ~\cite[Proposition 2.6]{BH1991}, $\BB(\GG,\PP,\partial)$ is a
CW-complex whose 0-cells identify with the objects of the groupoid
$\PP$ and whose homotopy groups, at any $a\in \Ob\PP$, can be
algebraically computed as
\begin{equation}\label{deshog}
\pi_i\big(\BB(\GG,\PP,\partial),a\big)=\left\{\begin{array}{l}
\text{the set of connected components of } \PP,
\text{ if } i=0,\\
\mathrm{Coker } \ \partial:\GG(a)\to \mathrm{Aut}_\PP(a),
\text{ if } i=1,\\
\mathrm{Ker }\  \partial:\GG(a)\to \mathrm{Aut}_\PP(a),
\text{ if } i=2,\\
0, \text{ if } i\geq 3.\\
\end{array}\right.
\end{equation}
Therefore, classifying spaces of crossed modules are homotopy
2-types. Furthermore, it is a consequence of ~\cite[Theorem
4.1]{BH1991} that, for any CW-complex $X$ with $\pi_i(X,a)=0$ for
all $i>2$ and base 0-cell $a$, there is a homotopy equivalence
$X\simeq \BB\Pi(X,X^{(1)},X^{(0)})$. Therefore, crossed modules are
algebraic models for homotopy 2-types.

\begin{lemma}\label{lehebgn} For any crossed module
  $(\GG,\PP,\partial)$, there is a homotopy natural homotopy
  equivalence
\begin{equation}\label{hnhe} \BB(\GG,\PP,\partial)\simeq
\BB\beta(\GG,\PP,\partial).\end{equation}
\end{lemma}
\begin{proof}
  By ~\cite[Theorem 2.4]{BH1991}, the functor
  $\Pi:\mathbf{SimpSet}\to\mathbf{Xmod}$ is left adjoint to the nerve
  functor $N:\mathbf{Xmod}\to \mathbf{SimpSet}$. Furthermore, in
  ~\cite[Theorem 2.3]{MS1993} Moerdijk and Svensson show that the {\em
    Whitehead $2$-groupoid functor} $W:\mathbf{SimpSet}\to
  2\text{-}\mathbf{Gpd}$, $K\mapsto W(K)=W\big(|K|,
  |K^{(1)}|,|K^{(0)}|\big)$ (see Example \ref{becm}) is left adjoint
  to the unitary geometric nerve functor
  $\Delta^{\!\mathrm{u}}:2\text{-}\mathbf{Gpd}\to
  \mathbf{SimpSet}$. Since, owing to the isomorphisms
  $(\ref{betpiw})$, there is a natural isomorphism $\beta\,\Pi\cong
  W$, we conclude that $\Delta^{\!\mathrm{u}}\beta\cong N$. Therefore,
  for $(\GG,\PP,\partial)$ any crossed module,
  $\BB(\GG,\PP,\partial)=|N(\GG,\PP,\partial)|\cong
  |\Delta^{\!\mathrm{u}}\beta(\GG,\PP,\partial)|\overset{(\ref{4h})}\simeq
  \BB\beta(\GG,\PP,\partial)$.
\end{proof}
\begin{remark}{\em For any crossed module $(\GG,\PP,\partial)$, the
$n$-simplices of $\Delta^{\!\mathrm{u}}\beta(\GG,\PP,\partial)$,
that is, the normal lax functors $[n]\to \beta(\GG,\PP,\partial)$,
are precisely systems of data
$$(g,p,a)=\big(g_{i,j,k},p_{i,j},a_i\big)_{0\leq i\leq j\leq k\leq
  n}$$
consisting of objects $a_i$ of $\PP$,  arrows $p_{i,j}:a_i\to a_j$
of $\PP$, with $p_{i,i}=1$,  and elements $g_{i,j,k}\in \GG(a_i)$,
with  $g_{i,i,j}=g_{i,j,j}=1$, such that the following conditions
hold:
$$
\begin{array}{ll}
\partial(g_{i,j,k})= p_{i,k}^{-1}\circ p_{j,k}\circ p_{i,j}
& \text{for } i\leq j\leq k,\\[4pt]
g_{i,j,k}^{-1}\cdot g_{i,k,l}^{-1}\cdot g_{i,j,l}\cdot
{}^{p_{i,j}^{-1}}{g}_{j,k,l}= 1& \text{for }i\leq j\leq k\leq l.
\end{array}
$$
Thus, the unitary geometric nerve
$\Delta^{\!\mathrm{u}}\beta(\GG,\PP,\partial)$ coincides with the
simplicial set called by Dakin ~\cite[Chapter 5, \S 3]{Dakin1977}
the {\em nerve of the crossed module} $(\GG,\PP,\partial)$ (cf.
~\cite[page 99]{BH1991} and ~\cite[Chapter 1, \S 11]{Ashley1983}).
From the above explicit description, it is easily proven that the
nerve of a crossed module is a Kan complex whose homotopy groups are
given as in $(\ref{deshog})$. }
\end{remark}

Thanks to Lemma \ref{lehebgn}, the bicategorical results obtained in
Section \ref{mainSection} are transferable to the setting of crossed
modules. To do so, if
$$\xymatrix@C=35pt{(\GG,\PP,\partial)\ar[r]^{(\varphi,F)}
  & (\HH,\QQ,\partial) &
  (\GG',\PP',\partial)\ar[l]_{(\varphi',F')}}$$ is any diagram
in $\mathbf{Xmod}$, then its {\em ``homotopy pullback crossed
module"}
\begin{equation}\label{hpulcros}
(\varphi,F)\!\downarrow\!(\varphi',F')=\big(\GG_{^{\varphi\mspace{-1mu},\mspace{-2mu}F\downarrow \mspace{-2mu}
\varphi'\mspace{-3mu},\mspace{-2mu}F'}},
\PP_{^{\!\varphi\mspace{-1mu},\mspace{-2mu}F\downarrow \mspace{-2mu}
\varphi'\mspace{-3mu},\mspace{-2mu}F'}},\partial\big)
\end{equation}
is constructed as follows:

{\em - The groupoid
$\PP_{^{\!\varphi\mspace{-1mu},\mspace{-2mu}F\downarrow
\mspace{-2mu} \varphi'\mspace{-3mu},\mspace{-2mu}F'}}$} has objects
the triples $(a,q,a')$, with $a\in \Ob \PP$, $a'\in \Ob \PP'$, and
$q:F a\to F'a'$ a morphism in $\QQ$. A morphism
$(p,h,p'):(a_0,q_0,a'_0)\to (a_1,q_1,a'_1)$  consists of a morphism
$p:a_0\to a_1$ in $\PP$, a morphism $p':a'_0\to a'_1$ in $\PP'$, and
an element $h\in \HH(Fa_0)$, which measures the lack of
commutativity of the square
$$
\xymatrix@C=20pt@R=20pt{Fa_0\ar[r]^{q_0}\ar[d]_{F p}&F'a'_0\ar[d]^{F'\!p'}\\
F a_1\ar[r]^{q_1}&F'a'_1}
$$
in the sense that the following equation holds: $
\partial h=F p^{-1}\circ q_1^{-1}\circ F'\!p'\circ q_0
$. The composition of two morphisms $ \xymatrix@C=35pt{
(a_0,q_0,a'_0)\ar[r]^{(p_1,h_1,p'_1)} &
(a_1,q_1,a'_1)\ar[r]^{(p_2,h_2,p'_2)}& (a_2,q_2,a'_2)} $ is given by
the formula
$$
(p_2,h_2,p'_2)\circ (p_1,h_1,p'_1)=
(p_2\circ p_1, {}^{F p_{1}^{-1}}\!h_2\cdot h_1, p'_2\circ p'_1).
$$
For every object $(a,q,a')$, its identity is $
1_{(a,q,a')}=(1_a,1,1_{a'})$, and the inverse of any morphism
$(p,h,p')$ as above is $
(p,h,p')^{-1}=(p^{-1},{}^{Fp}h^{-1},p'^{-1}). $ \vspace{0.2cm}

{\em - The functor
$\GG_{^{\varphi\mspace{-1mu},\mspace{-2mu}F\downarrow \mspace{-2mu}
\varphi'\mspace{-3mu},\mspace{-2mu}F'}}:\PP_{^{\!\varphi\mspace{-1mu},\mspace{-2mu}F\downarrow
\mspace{-2mu} \varphi'\mspace{-3mu},\mspace{-2mu}F'}}\to
\mathbf{Gp}$} is defined on objects by
$$\GG_{^{\varphi\mspace{-1mu},\mspace{-2mu}F\downarrow \mspace{-2mu}
\varphi'\mspace{-3mu},\mspace{-2mu}F'}}(a,q,a')= \GG(a)\times
\GG'\!(a') ,$$ and, for any morphism $(p,h,p'):(a_0,q_0,a'_0)\to
(a_1,q_1,a'_1)$, the associated homomorphism is given by $
^{(p,h,p')}(g,g')=({}^pg,{}^{p'}\!g') $.

 \vspace{0.2cm}{\em - The boundary map $\partial:\GG_{^{\varphi\mspace{-1mu},\mspace{-2mu}F\downarrow \mspace{-2mu}
\varphi'\mspace{-3mu},\mspace{-2mu}F'}}\to
\mathrm{Aut}_{\PP_{{\!\varphi\mspace{-1mu},\mspace{-2mu}F\downarrow
\mspace{-2mu} \varphi'\mspace{-3mu},\mspace{-2mu}F'}}}$}, at any
object $(a,q,a')$ of the groupoid $\PP_{^{\!\varphi,F\downarrow \varphi',F'}}$,
is  given by the formula
$$\partial(g,g')=
(\partial g, \varphi g^{-1}\cdot{}^{q^{-1}}\!\!\varphi'g',\partial g').$$

\vspace{0.2cm}
For any crossed module $(\HH,\QQ,\partial)$, we identify any object
$b\in \QQ$ with the morphism from the trivial crossed module
$b:(1,1,1)\to (\HH,\QQ,\partial)$ such that $b(1)=b$, so that, for
any morphism of crossed modules $(\varphi, F):(\GG,\PP,\partial)\to
(\HH,\QQ,\partial)$, we have defined the {\em ``homotopy-fibre
crossed
  module"} $$(\varphi, F)\!\downarrow\!b.$$

Next, we summarize our results in this setting of crossed modules.
The crossed module $(\ref{hpulcros})$ comes with a (non-commutative) square
\begin{equation}\label{funsqucro}\begin{array}{c}
    \xymatrix@R=25pt@C=40pt{
(\varphi,F)\!\downarrow\!(\varphi',F')\ar[r]^(.53){(\pi',P')}\ar[d]_{(\pi,P)}
& (\GG',\PP',\partial)
      \ar[d]^{(\varphi',F')}\\
      (\GG,\PP,\partial)\ar[r]^{(\varphi,F)}&(\HH,\QQ,\partial),
    }\end{array}
\end{equation}
where
$$
\begin{array}{c}
\xymatrix{
(a_0\overset{p}\to a_1)& \big((a_0,q_0,a'_0)\ar@{|->}[l]_{P}
                        \ar[r]^{(p,h,p')}
&(a_1,q_1,a'_1)\big)\ar@{|->}[r]^{P'}&(a'_0\overset{p'}\to a'_1)
}\\
\xymatrix{g&\ar@{|->}[l]_{\pi} (g,g')\ar@{|->}[r]^{\pi'}&g'}
\end{array}
$$

\begin{theorem}\label{mthpcm} The following statements hold:

$(i)$ For any morphisms of crossed modules
$\xymatrix@C=35pt{(\GG,\PP,\partial)\ar[r]^{(\varphi,F)}
  & (\HH,\QQ,\partial) &
  (\GG',\PP',\partial),\ar[l]_{(\varphi',F')}}$ there is
a homotopy $\BB(\varphi,F)\,\BB(\pi,P)\Rightarrow \BB(\varphi',F')\,
\BB(\pi',P')$ making the homotopy commutative square
 \begin{equation}\label{ifunsqucro}\begin{array}{c}
\xymatrix@C=40pt{
\BB((\varphi,F)\!\downarrow\!(\varphi',F')) \ar@{}@<60pt>[d]|(.45){\Rightarrow}
\ar[r]^-{\BB(\pi',P')}\ar[d]_{\BB(\pi,P)}&
 \BB(\GG',\PP',\partial)
\ar[d]^{\BB(\varphi',F')}\\
\BB(\GG,\PP,\partial)\ar[r]^{\BB (\varphi,F)}&\BB(\HH,\QQ,\partial),
}\end{array}
\end{equation}
induced by $(\ref{funsqucro})$ on classifying spaces, a homotopy
pullback square.

\vspace{0.1cm} $(ii)$ For any morphism of crossed modules
$(\varphi,F):(\GG,\PP,\partial)\to (\HH,\QQ,\partial)$ and every
object $b\in\QQ$, there is an induced homotopy fibre sequence
$$
\xymatrix@C=30pt{
\BB((\varphi,F)\!\downarrow\!b)\ar[r]^{\BB(\pi,P)}
&\BB(\GG,\PP,\partial)\ar[r]^{\BB(\varphi,F)} &\BB(\HH,\QQ,\partial).}
$$

\vspace{0.1cm} $(iii)$ A morphism of crossed modules
$(\varphi,F):(\GG,\PP,\partial)\to (\HH,\QQ,\partial)$ induces a
homotopy equivalence on classifying spaces, $\BB(\varphi,F):
\BB(\GG,\PP,\partial)\simeq \BB(\HH,\QQ,\partial)$, if and only if,
for every object $b\in\QQ$, the space
$\BB((\varphi,F)\!\downarrow\!b)$ is contractible.

\vspace{0.1cm} $(iv)$ For any crossed module $(\GG,\PP,\partial)$
and object $a\in \PP$, there is a homotopy equivalence
$$
\BB\big((\GG,\PP,\partial)(a)\big)\simeq \Omega(\BB(\GG,\PP,\partial),a),
$$
where $(\GG,\PP,\partial)(a)$ is the groupoid whose objects are the
automorphisms $p:a\to a$ in $\PP$, and whose arrows $g:p\to q$ are
those elements $g\in \GG(a)$ such that $p= q\circ \partial g$.
\end{theorem}
\begin{proof} $(i)$ Let us apply the equivalence of categories
  $(\ref{equ2})$ to the square of crossed modules
  $(\ref{funsqucro})$. Then, by direct comparison, we see that the
  equation between squares of 2-groupoids
$$
\xymatrix@R=20pt@C=30pt{
\beta\big((\varphi,F)\!\downarrow\!(\varphi',F')\big)\ar[r]^(.57){\beta(\pi',P')}\ar[d]_{\beta(\pi,P)}
& \beta(\GG',\PP',\partial)
      \ar[d]^{\beta(\varphi',F')}
\ar@{}@<45pt>[d]|{\textstyle =}
\\
      \beta(\GG,\PP,\partial)\ar[r]^{\beta(\varphi,F)}&\beta(\HH,\QQ,\partial)
    }
\xymatrix@R=20pt@C=30pt{
\big(\beta(\varphi,F)\!\downarrow\!\beta(\varphi',F')\big)\ar[r]^(.57){P'}\ar[d]_{P}
& \beta(\GG',\PP',\partial)
      \ar[d]^{\beta(\varphi',F')}\\
      \beta(\GG,\PP,\partial)\ar[r]^{\beta(\varphi,F)}&\beta(\HH,\QQ,\partial)
}
$$
holds, where the square on the right is $(\ref{pqsquare})$ for the
2-functors $\beta(\varphi,F)$ and $\beta(\varphi',F')$. As any
2-groupoid has property $(iv)$ in Theorem \ref{xbs} (see the comment
before Corollary \ref{corome}), that theorem gives a homotopy
$\BB\beta(\varphi,F)\, \BB \beta(\pi,P)\Rightarrow
\BB\beta(\varphi',F')\ \BB\beta(\pi',P')$ such that the induced
square
$$
\xymatrix@R=20pt@C=40pt{
\BB\beta\big((\varphi,F)\!\downarrow\!(\varphi',F')\big)\ar[r]^(.57){\BB\beta(\pi',P')}
\ar@{}@<60pt>[d]|(.4){\Rightarrow}
\ar[d]_{\BB\beta(\pi,P)}
& \BB\beta(\GG',\PP',\partial)
      \ar[d]^{\BB\beta(\varphi',F')}
\\
      \BB\beta(\GG,\PP,\partial)\ar[r]^{\BB\beta(\varphi,F)}&\BB\beta(\HH,\QQ,\partial)
    }
$$
is a homotopy pullback. It follows that the square
$(\ref{ifunsqucro})$ is also a homotopy pullback since, by Lemma
\ref{lehebgn}, it is homotopy equivalent to the square above.

The implications $(i)\Rightarrow (ii)\Rightarrow (iii)$ are clear,
and $(iv)$ follows from Corollary \ref{corome}, as
$(\GG,\PP,\partial)(a)=\beta(\GG,\PP,\partial)(a,a)$ and
$\BB(\GG,\PP,\partial)\simeq \BB\beta(\GG,\PP,\partial)$.
\end{proof}

We can easily show how the construction
$(\varphi,F)\!\downarrow\!(\varphi',F')$ works on basic examples
(see below).

\begin{example}\label{exfin} $(i)$ Let
  $\xymatrix@C=15pt{P\ar[r]^-{F}&Q&P'\ar[l]_(.4){F'}} $ be homomorphisms
  of groups. These induce homomorphisms of crossed modules of groups
  $\xymatrix{(1,P,1)\ar[r]^{(1,F)}&(1,Q,1)&(1,P',1)\ar[l]_{(1,F')},} $
  whose homotopy pullback crossed module is
  $(1,F)\!\downarrow\!(1,F')=(1,F\!\downarrow\!F',1)$, where $F\!\downarrow\!F'$ is the groupoid
  having as objects the elements $q\in Q$ and as morphisms
  $(p,p'):q_0\to q_1$ those pairs $(p,p')\in P\times P'$ such that
  $q_1\cdot Fp=F'p'\cdot q_0$. Thus, $(\ref{ifunsqucro})$
  particularizes by giving a homotopy pullback square
$$
\xymatrix@R=20pt@C=20pt{
\BB(F\!\downarrow\!F')\ar[r]\ar[d]&K(P',1)\ar[d]\\K(P,1)\ar[r]&K(Q,1).}
$$

$(ii)$ Let
$\xymatrix@C=15pt{A\ar[r]^-{\varphi}&B&A'\ar[l]_-{\varphi'}}$ be
homomorphisms of abelian groups. These induce homomorphisms of
crossed modules of groups $\xymatrix{(A,1,1)\ar[r]^-{(\varphi,1)}
&(B,1,1)&(A',1,1)\ar[l]_-{(\varphi',1)},} $ whose homotopy pullback
is the abelian crossed module of groups
$(\varphi,1)\!\downarrow\!(\varphi',1)=(A\times A',B,\partial)$,
where the coboundary map is given by
$\partial(a,a')=\varphi'a'-\varphi a$. Thus, $(\ref{ifunsqucro})$
particularizes by giving a homotopy pullback square
$$
\xymatrix@R=20pt@C=20pt{\BB(A\times
  A',B,\partial)\ar[r]\ar[d]&K(A',2)\ar[d]\\K(A,2)\ar[r]&K(B,2).}
$$
\end{example}

Let us stress that, as Example \ref{exfin}$(i)$ shows, the homotopy
pullback crossed module $(\varphi,F)\!\downarrow\!(\varphi',F')$ may
be a genuine crossed module of groupoids even in the case when both
$(\varphi,F)$ and $(\varphi',F')$ are morphisms between crossed
modules of groups. The reader can find in this fact a good reason to
be interested in the study of general crossed modules over
groupoids.

To finish, recall that the category of crossed complexes has a
closed model structure as described by Brown and Golasinski
~\cite{BG1989}. In this homotopy structure, a morphism of crossed
modules $(\varphi,F):(\GG,\PP,\partial)\to (\HH,\QQ,\partial)$ is a
weak equivalence if the induced map on classifying spaces
$\BB(\varphi,F)$ is a homotopy equivalence, and it is a fibration
(see Howie ~\cite{Howie1979}) whenever the following conditions
hold: $(i)$ $F:\PP\to\QQ$ is a fibration of groupoids, that is, for
every object $a\in \PP$ and every morphism $q:Fa\to b$ in $\QQ$,
there is a morphism $p:a\to a'$ in $\PP$ such that $Fp=q$, and
$(ii)$ for any object $a\in \PP$, the homomorphism
$\varphi:\GG(a)\to\HH(Fa)$ is surjective. Then, it is natural to ask
whether the constructed homotopy pullback crossed module
$(\varphi,F)\!\downarrow\!(\varphi',F')$ is actually a homotopy
pullback in the model category of crossed complexes. The answer is
positive as a consequence of the theorem below, and this fact
implies that the classifying space functor
$(\GG,\PP,\partial)\mapsto \BB(\GG,\PP,\partial)$ preserves homotopy
pullbacks.

\begin{theorem}\label{weakEquivalenceofHPB}
Let $\xymatrix@C=30pt{(\GG,\PP,\partial)\ar[r]^{(\varphi,F)}
  & (\HH,\QQ,\partial) &
  (\GG',\PP',\partial)\ar[l]_{(\varphi',F')}}$ be
  morphisms of crossed modules. If one of them is a fibration, then
  the canonical morphism $$
(\GG,\PP,\partial)\times_{(\HH,\QQ,\partial)}(\GG',\PP',\partial)\to
(\varphi,F)\!\downarrow\!(\varphi',F')$$ induces a homotopy
  equivalence
$
\BB\big((\GG,\PP,\partial)\times_{(\HH,\QQ,\partial)}(\GG',\PP',\partial)\big)
\simeq
  \BB\big((\varphi,F)\!\downarrow\!(\varphi',F')\big)$.
\end{theorem}
\begin{proof}

Let us observe that the pullback crossed module of $(\varphi,F)$ and
$(\varphi',F')$ is
$$(\GG,\PP,\partial)\times_{(\HH,\QQ,\partial)}(\GG',\PP',\partial)=
(\GG\times_{\HH F}\!\GG',\PP\times_\QQ\PP', \partial),$$
where  $\PP\times_\QQ \PP'$ is the pullback groupoid of $F:\PP\to
\QQ$ and $F':\PP'\to \QQ$. The functor  $\GG\times_{_ F}\!\GG'
:\PP\times_\QQ \PP'\to \mathbf{Gp}$ is defined on objects by
$$
(\GG\times_{\HH F}\!\GG')(a,a')= \GG(a)\times_{\HH(Fa)}\GG'(a')=\{(g,g')\in \GG(a)\times\GG'(a')\mid \varphi_a(g)=\varphi'_{a'}(g')\},
$$
and the homomorphism associated to any morphism $(p,p'):(a,a')\to
(b,b')$ in $\PP\times_\QQ\PP'$ is given by
$^{(p,p')}(g,g')=({}^pg,{}^{p'}\!g') $. The boundary map $\partial:
\GG\times_{\HH F}\!\GG'\to \mathrm{Aut}_{\PP\times_\QQ\PP'}$, at any
object  of the groupoid $\PP\times_\QQ\PP'$, is given by the formula
$\partial(g,g')=(\partial g,\partial g')$.

The canonical morphism
\begin{equation}\label{canmor} (\jmath,J):
(\GG\times_{\HH F}\!\GG',\PP\times_\QQ\PP', \partial)\to
\big(\GG_{^{\varphi\mspace{-1mu},\mspace{-2mu}F\downarrow \mspace{-2mu}
\varphi'\mspace{-3mu},\mspace{-2mu}F'}},
\PP_{^{\!\varphi\mspace{-1mu},\mspace{-2mu}F\downarrow \mspace{-2mu}
\varphi'\mspace{-3mu},\mspace{-2mu}F'}},\partial\big)\end{equation} is as follows: The functor
$J\!:\PP\!\times_\QQ\PP'\to
\PP_{^{\!\varphi\mspace{-1mu},\mspace{-2mu}F\downarrow \mspace{-2mu}
\varphi'\mspace{-3mu},\mspace{-2mu}F'}}$ sends a morphism
$(p,p')\!:(a,a')\to \!(b,b')$ to the morphism $(p,1_{\HH(F
a)},p'):(a,1_{F a},a')\to (b,1_{F b},b')$, and the
$\PP\times_\QQ\PP'$-group homomorphism $\jmath:\GG\times_{\HH
F}\!\GG'\to
 \PP_{^{\!\varphi\mspace{-1mu},\mspace{-2mu}F\downarrow \mspace{-2mu}
\varphi'\mspace{-3mu},\mspace{-2mu}F'}} J$ is  given at any object
$(a,a')\in \PP\times_\QQ\PP'$ by the inclusion map
$\GG(a)\times_{\HH(F a)} \GG'(a')\hookrightarrow
\GG(a)\times\GG'(a') $.

Next, we assume that  $(\varphi, F)$ is a fibration. Then, we
verify that the canonical morphism $(\ref{canmor})$ induces
isomorphisms between the corresponding homotopy groups. Recall from
$(\ref{deshog})$ how to compute the homotopy groups of the
classifying space of a crossed module.

\noindent $\bullet$ {\em The map $\pi_0(\jmath,J)$ is a bijection.}

Injectivity: Suppose  objects $(a,a'),(b,b')\in \PP\times_\QQ\PP'$,
such that there is a morphism
$(p,h,p'):(a,1_{Fa},a')\to(b,1_{Fb},b')$ in
$\PP_{^{\!\varphi\mspace{-1mu},\mspace{-2mu}F\downarrow
\mspace{-2mu} \varphi'\mspace{-3mu},\mspace{-2mu}F'}}$. Then, as
 $\varphi: \GG(a)\to \HH (F a)$ is surjective, there is
$g\in\GG(a)$ such that $\varphi(g)=h$, whence $(p\circ \partial
g,p'):(a,a')\to (b,b')$ is a morphism in $\PP\times_\QQ\PP'$.

Surjectivity: Let $(a,q,a')$ be an object of
$\PP_{^{\!\varphi\mspace{-1mu}, \mspace{-2mu}F\downarrow
\mspace{-2mu} \varphi'\mspace{-3mu},\mspace{-2mu}F'}}$.  As
 $F:\PP\to \QQ$ is a fibration of groupoids, there is a morphism
$p:a\to b$ in $\PP$ such that $Fp=q$. Then, $(b,a')$ is an object of
the groupoid $\PP\times_\QQ\PP'$ with $J(b,a')=(b,1_{Fb},a')$ in the
same connected component of $(a,q,a')$, since we have the morphism
$(p, 1_{\HH(Fb)},1_{a'}):(a,q,a')\to (b,1_{F{b}},a')$.

\noindent $\bullet$ {\em The homomorphisms $\pi_1(\jmath,J)$ are
isomorphisms.} Let $(a,a')$ be any object of $\PP\times_\QQ\PP'$.

Injectivity: Let $[(p,p')]$ be an element in the kernel of the
homomorphism $\pi_1(\jmath,J)$ at $(a,a')$, that is, such that
$[(p,1_{\HH(F
  a)} ,p')]=[(1_a,1_{\HH(F a)} ,1_{a'})]$. This means that
there is $(g,g')\in\GG(a)\times\GG'(a')$ with $\partial g = p$,
$\partial (g') = p'$ and $\varphi (g)^{-1}\cdot \varphi'(g')=1$. The
last equation says that $(g,g')$ is an element of
$\GG(a)\times_{\HH(Fa)}\GG'(a)$ which, by the first two, satisfies
that $\partial(g,g')=(p,p')$. Hence,  $[(p,p')]=[(1_a,1_{a'})]$.

Surjectivity: Let $(p,h,p'):(a,1_{Fa},a')\to (a,1_{Fa},a')$ be an
automorphism of
$\PP_{^{\!\varphi\mspace{-1mu},\mspace{-2mu}F\downarrow
\mspace{-2mu} \varphi'\mspace{-3mu},\mspace{-2mu}F'}}$. As
 $\varphi: \GG(a)\to \HH (F a)$ is surjective, there is a $g\in\GG(a)$ such
that $\varphi(g)=h$. Then, we have
\begin{align}\nonumber
(p,h,p')^{-1}\circ J(p\circ \partial g,p')&=(p,h,p')^{-1}\circ (p\circ\partial g,1_{\HH(Fa)},p')\\
\nonumber
 &=(p^{-1}\circ p\circ \partial g, h^{-1},p'^{-1}\circ p')\\ \nonumber
&=(\partial g,\varphi(g)^{-1}\cdot 1_{\HH(Fa)},1_{a'})=\partial (g,1_{\GG'(a')})
\end{align}
and therefore $[(p,h,p')]=[J(p\circ \partial g,p')]$.

\noindent $\bullet$ {\em The homomorphisms $\pi_2(\jmath,J)$ are
isomorphisms. } At any object $(a,a')\in\PP\times_\QQ\PP'$, the
homomorphism $\pi_2(\jmath,J)$ is the restriction to the kernels of
the boundary maps of the inclusion
$\GG(a)\times_{\HH(Fa)}\GG'(a')\hookrightarrow
\GG(a)\times\GG'(a')$. Then, it is clearly injective.  To see the
surjectivity,  let
 $(g,g')\in \GG(a)\times \GG'(a')$ with $\partial
(g,g')=(1_a,1_{\HH (Fa)},1_{a'})$.  Then, we have $\partial g=1_a$,
$\partial g'=1_{a'}$ and $\varphi( g)^{-1}\cdot \varphi'(
g')=1_{\HH(Fa)}$. That is, that $(g,g')\in
\GG(a)\times_{\HH(Fa)}\GG'(a')$ and $(\partial g,\partial
g')=(1_a,1_{a'})$.
\end{proof}

\section{Appendix: Proofs of Lemmas \ref{fact1} and \ref{facts2}}\label{appendix}

We shall only address lax functors below, but the discussions are
easily dualized in order to obtain the corresponding results for
oplax functors.

Our first goal is to accurately determine the functorial behaviour
of the Grothendieck nerve construction $\B\mapsto \ner\B$
(\ref{GrothendieckNerve}) on lax functors between bicategories by
means of the theorem below. The result in the first part of it is
already known: see ~\cite[\S 3, (21)]{CCG2010}, where a proof is
given using Jardine's Supercoherence Theorem in ~\cite{Jardine1991}.
However, for the second part, we need a new proof of
 the existence of the  pseudo-simplical category $\ner\B$, with a
more explicit and detailed construction of it.

\begin{theorem}\label{groner}  $(i)$ Any bicategory $\B$ defines a normal
  pseudo-simplicial category, that is, a unitary pseudo-functor from
  the simplicial category $\Delta^{\!{\mathrm{op}}}$ into the
  $2$-category of small categories,
$$
\ner\B=(\ner\B,\widehat{\ner}\B):\Delta^{\!{\mathrm{op}}}\ \to \Cat,
$$
which is called the {\em Grothendieck} or {\em pseudo-simplicial
nerve} of the bicategory, whose category of $p$-simplices, for
$p\geq 0$, is {\em $$ \ner\B_{p}:=
\hspace{-0.2cm}\bigsqcup_{(x_p,\ldots,x_0)\in \text{\scriptsize
    Ob}\B^{p+1}}\hspace{-0.3cm}
\B(x_{p-1},x_p)\times\B(x_{p-2},x_{p-1})\times\cdots
  \times\B(x_0,x_1).
$$
}
$(ii)$ Any lax functor between bicategories
  $F:\B\to\B'$ induces a lax transformation (i.e., a lax simplicial
  functor)
$$
\ner F=(\ner F,\widehat{\ner}F):\ner\B\to\ner \B'.
$$

For any pair of composable lax functors $F:\B\to \B'$ and $F':\B'\to
\B''$, the equality $\ner F'\, \ner F=\ner (F'F)$ holds, and, for
any bicategory $\B$, $ \ner 1_\B= 1_{\ner\B}$.
\end{theorem}

Before starting with the proof, we shall describe some needed
constructions and a few auxiliary facts.
 Given a category $\mathcal{I}$ and a bicategory $\B$, we denote by
$$
\brep(\mathcal{I},\B)
$$
the category whose objects are lax functors $F:\mathcal{I}\to \B$,
and whose morphisms are {\em relative to object lax
transformations}, as termed by Bullejos and Cegarra in
~\cite{BC2003}, but also called {\em  icons} by Lack in
~\cite{Lack2010}. That is, for any two lax functors
$F,G:\mathcal{I}\to \B$, a morphism $\Phi:F\Rightarrow G$ may exist
only if $F$ and $G$ agree on objects, and it is then given by
2-cells in $\B$, $\Phi a:Fa\Rightarrow Ga$, for every arrow $a:i\to
j$ in $\mathcal{I}$, such that the diagrams
$$\begin{array}{ll}
  \xymatrix@R=20pt@C=15pt{Fa\circ  Fb\ar@{=>}[r]^{\widehat{F}_{a,b}}
             \ar@{=>}[d]_{\Phi
      a\circ  \Phi b}&F(ab) \ar@{=>}[d]^{\
      \Phi(ab)}\\
    Ga\circ  Gb\ar@{=>}[r]^{\widehat{G}_{a,b}}&G(ab),}&
  \xymatrix@R=20pt@C=5pt{&1_{Fi=Gi}\ar@{=>}[ld]_{\widehat{F}_i}
    \ar@{=>}[rd]^{\widehat{G}_i}& \\
    F1_i\ar@{=>}[rr]^{\Phi{1_i}}&&G1_i,
  }
\end{array}
$$
commute for each pair of composable arrows $i\overset{ b}\to
j\overset{ a}\rightarrow k$ and each object $i$. The composition of
morphisms $\Phi:F\Rightarrow G$ and $\Psi:G\Rightarrow H$, for
$F,G,H:\mathcal{I}\to \B$ lax functors, is
$\Psi\cdot\Phi:F\Rightarrow H$, where $(\Psi\cdot\Phi) a=\Psi
a\cdot\Phi a:Fa\Rightarrow Ha$, for each arrow $a:i\to j$ in
$\mathcal{I}$. The identity morphism of a lax functor
$F:\mathcal{I}\to \B$ is $1_F:F\Rightarrow F$, where
$(1_F)a=1_{Fa}$, the identity of $Fa$ in the category $\B(Fi,Fj)$,
for each $a:i\to j$ in $\mathcal{I}$.

Let us now replace the category $\mathcal{I}$ above by a (directed)
graph $\mathcal{G}$. For any bicategory $\B$, there is a category
$$
\brepg(\mathcal{G},\B),
$$
where an object  $f:\mathcal{G}\to \B$ consists of a pair of maps
that assign an object $fi$ to each vertex $i\in \mathcal{G}$ and a
1-cell $fa:fi\to fj$ to each edge $a:i\to j$ in $\mathcal{G}$,
respectively. A morphism $\phi:f\Rightarrow g$ may exist only if $f$
and $g$ agree on vertices, that is, $fi=gi$ for all $i\in
\mathcal{G}$; and then it consists of a map that assigns to each
edge  $a:i\to j$ in the graph a 2-cell $\phi a:fa\Rightarrow ga$ of
$\B$. Compositions in $\brepg(\mathcal{G},\B)$ are defined in the
natural way by the same rules as those stated above for the category
$\brep(\mathcal{I},\B)$.

\begin{lemma}\label{rufc} Let $\mathcal{I}=\mathcal{I}(\mathcal{G})$ be the free category
  generated by a graph $\mathcal{G}$, let $\B$ be a bicategory, and let
$$
R: \brep(\mathcal{I}(\mathcal{G}),\B)\to \brepg(\mathcal{G},\B)
$$
be the functor defined by restriction to the basic graph. Then,
there is a functor
$$
J:\brepg(\mathcal{G},\B)\to   \brep(\mathcal{I},\B),
$$
and a natural transformation
\begin{equation}\label{counit}
  \v:JR\Rightarrow 1_{\brep(\mathcal{I},\B)},
\end{equation}
such that $RJ=1_{\brepg(\mathcal{G},\B)},\ \v J=1_{J},\ R\v=1_{R}$.
Thus, the functor $R$ is right adjoint to the functor $J$.
\end{lemma}
\begin{proof} To describe the functor $J$, we use the following
  useful construction: For any list $(x_0,\dots,x_p)$ of objects in
  the bicategory $\B$, let
$$
  \xymatrix{
    \overset{_\mathrm{or}}\circ:\B(x_{p-1},x_p)\times\B(x_{p-2},x_{p-1})
    \times\cdots\times\B(x_0,x_1)
    \longrightarrow \B(x_0,x_p)}
$$
denote the functor obtained by iterating horizontal composition in
the bicategory, which acts on objects and arrows of the product
category by the recursive formula
$$
\overset{_\mathrm{or}}\circ(u_p,\dots,u_1)=
\left\{\begin{array}{lll}u_1&\text{if}& p=1,\\[3pt]
    u_p\circ\big(\overset{_\mathrm{or}}\circ(u_{p-1},\dots,u_1)\big)&
    \text{if}&p\geq 2.\end{array}\right.
$$

Then, the homomorphism $J$ takes a graph map, say $f:\mathcal{G}\to
\B$, to the unitary pseudo-functor from the free category
$$
J(f)=F:\mathcal{I}\to \B,
$$
such that $ Fi=fi$, for any vertex $i$ of $\mathcal{G}$ (= objects
of $\mathcal{I}$), and associates to strings $ a:a(0)\overset{
a_{1}}\to \cdots \overset{ a_{p}}\to a(p) $ in $\mathcal{G}$ the
1-cells $Fa= \overset{_\mathrm{or}}\circ(fa_p,\dots,fa_1):fa(0)\to
fa(p)$.  The structure 2-cells $\widehat{F}_{a,b}:Fa\circ Fb
\Rightarrow F(ab)$, for any pair of strings in the graph,
$a=a_p\cdots a_1$ as above and $ b=b_q\cdots b_1$ with $b(q)=a(0)$,
are canonically obtained from the associativity constraints in the
bicategory: first by taking $\widehat{F}_{a_1,b}=1_{F(a_1b)}$ when
$p=1$ and then, recursively for $p>1$, defining $\widehat{F}_{a,b}$
as the composite
$$
\xymatrix@C=35pt{\widehat{F}_{a,b}:\ Fa\circ Fb \overset{
    \boldsymbol{a}}\Longrightarrow Fa_p\circ ( Fa'\circ
  Fb)\ar@{=>}[r]^-{1\circ \widehat{F}_{a',b} } &F(ab),}
$$
where $a'=a_{p-1}\cdots a_1$ (whence $Fa=Fa_p\circ  Fa'$). The
coherence conditions for $F$ are easily verified by using the
coherence and naturality of the associativity constraint $\aso$ of
the bicategory.

Any morphism $\phi:f \Rightarrow g$ in $\brepg(\mathcal{G},\B)$ is
taken by $J$ to the morphism $J(\phi):F\Rightarrow G$ of
$\brep(\mathcal{I},\B)$, consisting of the 2-cells in the bicategory
$ \overset{_\mathrm{or}}\circ (\phi{a_p},\dots,\phi{a_1}):Fa
\Rightarrow Ga$, attached to the strings of adjacent edges in the
graph $ a=a_p\cdots a_1$. The coherence conditions of $J(\phi)$ are
 consequence of the naturality
of the associativity constraint $\aso$ of the bicategory. If
$\phi:f\Rightarrow g$ and $\psi: g\Rightarrow h$ are 1-cells in
$\brepg(\mathcal{G},\B)$, then  $J(\psi)\cdot J(\phi) = J(\psi\cdot
\phi)$ follows from the functoriality of the composition $\circ$,
and so $J$ is a functor.

The lax transformation $\v$ is defined as follows: The component of
this lax transformation at a lax functor $F:\mathcal{I}\to \B$,
$\v:JR(F)\Rightarrow F$, is defined on identities by
$\v1_i=\widehat{F}_i:1_{Fi}\Rightarrow F1_i$, for any vertex $i$ of
$\mathcal{G}$, and it associates to each string of adjacent edges in
the graph $a=a_p\cdots a_1$ the 2-cell $ \v
a:\overset{_\mathrm{or}}\circ (Fa_p,\dots,Fa_1)\Rightarrow Fa $,
which is given by taking $\v{ a_1}=1_{Fa_1}$ if $p=1$, and then,
recursively for $p>1$, by taking $\v{a}$ as the composite
$$
\v a=\xymatrix@C=25pt{\big(\overset{_\mathrm{or}}\circ
  (Fa_p,\dots,Fa_1) \ar@{=>}[r]^-{1\circ \v{a'} } &Fa_p\circ
  Fa'\ar@{=>}[r]^-{\widehat{F}_{a_p,a'}}&Fa\big),}$$
where $a'=a_{p-1}\cdots a_1$. The naturality condition
$\widehat{F}_{a,b}\circ (\v a\circ \v b)= \v(ab)\circ
\widehat{JR(F)}_{a,b}$, for any pair of composable morphisms in
$\mathcal{I}$, can be checked as follows: when $a=1_i$ or $b=1_j$
are identities, then it is a consequence of the commutativity of the
diagrams
$$
\xymatrix@C=0pt@R=12pt{ &1_{Fi}\circ
  JR(F)b\ar@{=>}[rr]^{\boldsymbol{l}} \ar@{=>}[rd]^(.55){1\circ \v b}
  \ar@{=>}[dd]_{\widehat{F}_i\circ \v b}&&JR(F)b\ar@{=>}[dd]^{\v b}\\
  &\ar@{}@<4pt>[r]|(.4){(A)} &1_{Fi}\circ
  Fb\ar@{=>}[rd]^{\boldsymbol{l}}
  \ar@{=>}[ld]_{\widehat{F}_i\circ  1}\ar@{}[ru]|{(B)}\ar@{}[d]|{(C)}& \\
  &F1_i\circ Fb\ar@{=>}[rr]_{\widehat{F}_{1_i,b}}&&Fb, } \hspace{0.3cm}
\xymatrix@C=0pt@R=12pt{ &JR(F)a\circ
  1_{Fj}\ar@{=>}[rr]^{\boldsymbol{r}} \ar@{=>}[rd]^(.55){\v a\circ 1}
  \ar@{=>}[dd]_{\v a\circ \widehat{F}_j}&&JR(F)a
  \ar@{=>}[dd]^{\v a}\\
  &\ar@{}@<4pt>[r]|(.4){(A)}&Fa\circ
  1_{Fj}\ar@{=>}[rd]^{\boldsymbol{r}}
  \ar@{=>}[ld]_{1\circ  \widehat{F}_j}\ar@{}[ru]|{(B)}\ar@{}[d]|{(C)}& \\
  &Fa\circ F1_j\ar@{=>}[rr]_{\widehat{F}_{a,1_i}}&&Fa, }
$$
where the regions labelled with $(A)$ commute by the functoriality
of $\circ$, those with $(B)$ by the naturality of $\boldsymbol{l}$
and $\boldsymbol{r}$, and those with $(C)$ by the coherence of $F$.
Now, for arbitrary strings $a$ and $b$ in the graph with
$b(q)=a(0)$, we study the coherence recursively on the length of
$a$. The case when $p=1$ is the obvious commutative diagram
$$
\xymatrix@C=15pt@R=20pt{Fa_1\circ JR(F)b\ar@{=>}[d]_-{1\circ \v
    b}\ar@{=>}[r]^-{1}&
  JR(F)(a_1b)\ar@{=>}[d]^{\v(a_1b)=\widehat{F}_{a_1\!,b}\cdot (1\circ  \v b)}\\
  Fa_1\circ Fb\ar@{=>}[r]_-{\widehat{F}_{a_1,b}}&F(a_1b),}
$$
and then, for $p>1$, the result is a consequence of the diagram
$$
\xymatrix@C=40pt@R=20pt{
 JR(F)a\circ JR(F)b \ar@{=>}[r]^-{\boldsymbol{a}}
                   \ar@{=>}[d]_{(1\circ \v a')\circ \v b}
&Fa_p\!\circ \!(JR(F)a'\!\circ \! JR(F)b) \ar@{}[rd]|{ (B)}
             \ar@{=>}[d]|{1\circ (\v a'\circ \v b)}
             \ar@{=>}[r]^-{1\circ \widehat{JR(F)}_{a'\!,b}}
             \ar@{}[ld]|{(A)}
&JR(F)(ab)  \ar@{=>}[d]^{1\circ  \v (a'b)}
\\(Fa_p\!\circ \! Fa')\!\circ \! Fb\ar@{}[rrd]|{ (C)}
       \ar@{=>}[r]^-{\boldsymbol{a}}
       \ar@{=>}[d]_{\widehat{F}_{a_p\!,a'}\circ 1}
&Fa_p\!\circ \!(Fa'\!\circ \!Fb)\ar@{=>}[r]_{1\circ\widehat{F}_{a'\!,b}}
&Fa_p\circ F(a'b)  \ar@{=>}[d]^{\widehat{F}_{a_p\!,a'b}}
\\  Fa\circ Fb\ar@{=>}[rr]^(.4){\widehat{F}_{a,b}}
&& F(ab) }
$$
where $(A)$ commutes by the naturality of $\boldsymbol{a}$, $(B)$ by
induction, and $(C)$ by the coherence of $F$.

To verify the equalities $RJ=1$, $\v J=1$, and $R\v=1$ is
 straightforward.\end{proof}

Let  $\mathcal{I}=\mathcal{I}(\mathcal{G})$ again be the free
category generated by a graph $\mathcal{G}$, as in Lemma \ref{rufc}
above, and suppose now that $F:\B\to\B'$ is a lax functor. Then, the
square
\begin{equation}\label{twosquares}
\begin{array}{c}\xymatrix{
\brep(\mathcal{I},\B) \ar[d]_{F_*}\ar[r]^{R}&\brepg(\mathcal{G},\B)\ar[d]^{F_*}\\
\brep(\mathcal{I},\B')\ar[r]^{R'}&\brepg(\mathcal{G},\B') }
\end{array}
\end{equation}
commutes and, since $RJ=1$, we have the equalities
\begin{equation}\label{newiqu}
R'F_*JR=F_*RJR =F_*R=R'F_*.
\end{equation}
Furthermore, the naturality of $\v:JR\Rightarrow 1$ and
$\v':J'R'\Rightarrow 1$ means that the square
$$
\xymatrix@C=50pt{
J'R'F_*JR\ar@2[r]^{\v'F_*JR}\ar@{=>}[d]_{J'R'F_*\v}&F_*JR\ar@2[d]^{F_*\v}\\
J'R'F_*\ar@2[r]^{\v'F_*}&F_*
}
$$
commutes. As $R\v=1_R$ and then
$J'R'F_*\v=J'F_*R\v=J'F_*1_R=1_{J'R'F_*}$,  we have the equality
\begin{equation}\label{commutativeDiagram1}
F_*\v\circ \v'F_*JR =\v'F_*.
\end{equation}

\subsection{Proof of Theorem \ref{groner}.} $(i)$:
Let us note that, for any integer $p\geq 0$, the category $[p]$ is
free on the graph
$$
\mathcal{G}_p=(0\to 1 \cdots\to p).
$$
Then, for any given bicategory $\B$, the existence of an adjunction
\begin{equation}\label{plrg}
J_p\dashv R_p: \ner\B_p=\brepg(\mathcal{G}_p,\B) \rightleftarrows \brep([p],\B)
\end{equation}
follows from Lemma \ref{rufc}, where $R_p$ is the functor defined
by restricting to the basic graph $\mathcal{G}_p$ of the category
$[p]$, where $R_pJ_p=1$, whose unity is the identity, and whose
counit $\v_p:J_pR_p\Rightarrow 1$ satisfies the equalities $\v_p
J_p=1$ and $R_p\v_p=1$.

If $a:[q]\to [p]$ is any map in the simplicial category, then the
associated functor $\ner\B_{a}:\ner\B_{p}\to\ner\B_{q}$ is the
composite
$$
\begin{array}{l}
\xymatrix@R=16pt{ \ner\B_{p}\ar@{.>}[r]^{\ner\B_{a}}
\ar[d]_{J_p}
&\ner\B_{q}\\
\brep([p],\B)\ar[r]^{a^*}&\brep([q],\B). \ar[u]_{R_q}}
\end{array}
$$
Thus, $\ner\B_{a}$ maps the component category of $\ner\B_{p}$ at
$(x_p,\dots,x_0)$ into the component at $(x_{a(q)},\dots,x_{a(0)})$
of $\ner\B_{q}$, and it acts both on objects and morphisms of
$\ner\B_{p}$ by the formula
$\ner\B_{p}(u_p,\dots,u_1)=(v_q,\dots,v_1)$, where, for $0\leq k<
q$,
$$
v_{k+1}=\left\{\begin{array}{lll}
\xymatrix{\hspace{-4pt}\overset{_\mathrm{or}}\circ (u_{a(k+1)},
\ldots, u_{a(k)+1})}& \text{\em if} & a(k)<a(k+1),\\[6pt]
1& \text{\em if} & a(k)=a(k+1),\end{array}\right.
$$
whence, in particular, the usual formulas below for the face and
degeneracy functors.
$$
\begin{array}{l} d_i(u_p,\dots,u_1)=\left\{
\begin{array}{lcl}
  \hspace{-0.2cm}  (u_p,\dots,u_2) &\text{ if }& i=0 ,  \\[4pt]
  \hspace{-0.2cm}  (u_p,\dots,u_{i+1}\circ  u_{i},\dots,u_1)
  & \text{ if } & 0<i<p , \\[4pt]
  \hspace{-0.2cm}  (u_{p-1},\dots,u_1) & \text{ if }& i=p,
\end{array}\right.\\[20pt]
s_i(u_p,\dots,u_1)=(u_p,\dots,u_{i+1},1,u_{i},\dots, u_0).
\end{array}
$$

 The structure natural transformation
\begin{equation}\label{constraint1}
 \begin{array}{l}
 \xymatrix@C=15pt{
    \ner\B_{p} \ar@/^1.1pc/[rr]^{\ner\B_{b}\ \ner\B_{a}}
             \ar@/_1.1pc/[rr]_{ \ner\B_{ab}}
             \ar@{}[rr]|{\Downarrow\,\widehat{\ner}\B_{a,b}}
    & &\ner\B_{n},}
\end{array}
\end{equation}
for each pair of composable maps $[n]\overset{b}\to
[q]\overset{a}\to [p]$
 in $\Delta$, is
$$
\xymatrix@C=40pt{\ner\B_{b}\ \ner\B_{a}=
  R_n b^*J_qR_q a^* J_p\ar@{=>}[rr]^-{\widehat{\ner}\B_{a,b}=R_nb^*\v_q a^*J_p}
  &&R_n b^*a^*J_p= R_n (ab)^*J_p=\ner\B_{ab}.}$$

Let us stress that, in spite of the natural transformation $\v$ in
$(\ref{counit})$ not being invertible, the natural transformation
$\widehat{\ner}\B_{a,b}$ in $(\ref{constraint1})$ is invertible
since, for any $\x\in\ner\B_{p}$, the lax functor $a^*J_p\x$ is
actually a homomorphism and therefore $\v_qa^*J_p\x$ is an
isomorphism. Consquently, we only need to prove that these
constraints $\widehat{\ner}\B_{a,b}$
  verify the coherence conditions for lax functors:

  If $a=1_{[p]}$,
  then $\widehat{\ner}\B_{1,b}=R_nb^*\v_pJ_p=R_nb^*1_{J_p}=1_{\ner\B_b}$. Similarly,
  $\widehat{\ner}\B_{a,1}= 1_{\ner\B_a}$. Furthermore, for every
  triplet of composable arrows $[m]\overset{c}\to [n]\overset{b}\to
  [q] \overset{a}\to [p]$, the diagram
$$
\begin{array}{l}\xymatrix@C=50pt{\ner\B_c\ \ner\B_b\
    \ner\B_a\ar@{=>}[d]_{ \widehat{\ner}\B_{b,c} \,
      \ner\B_{a}}\ar@{=>}[r]^{
      \ner\B_{c}\,\widehat{\ner}\B_{a,b}}& \ner\B_c \ \ner\B_{ab}
    \ar@{=>}[d]^{ \widehat{\ner}\B_{ab,c}}\\  \ner\B_{bc}
    \  \ner\B_{a}\ar@{=>}[r]^{\widehat{\ner}\B_{a,bc}}
& \ner\B_{abc},}
\end{array}
$$
is commutative since it is obtained by applying the functors $R_m
c^*$ on the left, and $a^*J_p$ on the right, to the diagram
\begin{equation}\label{natvv}
\begin{array}{l}
  \xymatrix@C=45pt@R=18pt{J_nR_n b^*J_qR_q\ar@{=>}[r]^(.55){J_nR_n b^*\v_q}
        \ar@{=>}[d]_{\v_n b^*J_qR_q}
    &J_nR_n b^*\ar@{=>}[d]^{\v_n b^*}
    \\ b^*J_qR_q\ar@{=>}[r]_{b^*\v_q}&b^*,}
\end{array}
\end{equation}
which commutes by the naturality of $\v_n$.

\vspace{0.2cm} $(ii)$: Suppose now that $F:\B\to\B'$ is a lax
functor. Then, at any integer $p\geq 0$, the functor $\ner
F_p:\ner\B_p\to\ner\B'_p$ is the composite
$$
\xymatrix@R=20pt{\ner\B_p\ar@{.>}[r]^{\ner F_p}\ar[d]_{J_p}&\ner\B'_p\\
\brep([p],\B)\ar[r]^{F_*}&\brep([p],\B'),
\ar[u]_{R'_p}}
$$
which is explicitly given both on objects and arrows by the simple
formula $\ner F_p(u_p,\dots,u_1)=(Fu_p,\dots, Fu_1)$. The structure
natural transformation
$$
\begin{array}{l}
  \xymatrix{
    \ner\B_p \ar@{}@<30pt>[d]|(.4){\widehat{\ner}F_a}|(.6){\Rightarrow}
             \ar[r]^{\ner\B_{a}}\ar[d]_{\ner F_{p}}
    &\ner\B_q \ar[d]^{\ner F_{q}}
    \\\ner\B'_p\ar[r]_{\ner\B'_{a}}
    &\ner\B'_q,}
\end{array}
$$
at each map $a:[q]\to [p]$ in $\Delta$, is

\noindent$
\begin{array}{ll}
\xymatrix@C=70pt{\ner\B'_a\,\ner F_p=R'_qa^*J'_pR'_pF_*J_p
\ar@2[r]^-{\widehat{\ner}F_a=R'_qa^*\v'_pF_*J_p}&}&
\hspace{-0.3cm}
R'_qa^*F_*J_p=R'_qF_*a^*J_p
\overset{(\ref{newiqu})}= R'_qF_*J_qR_qa^*J_p
\\ &\hspace{-0.3cm}= \ner F_q\,\ner\B_a.
\end{array}
$

This family of natural transformations $\widehat{\ner}F_a$ verifies
the coherence conditions for lax transformations: If $a=1_{[p]}$,
then $\widehat{\ner}F_1=R'_p\v'_pF_*J_p= 1_{R'_p}F_*J_p=
1_{\widehat{\ner}F_p}$. Suppose  that   $b:[n]\to [q]$ is any other
map of $\Delta$, then the coherence diagram
$$
\xymatrix@C=45pt{
\ner\B'_b\,\ner\B'_a\,\ner F_p
     \ar@2[d]_{\widehat{\ner}\B'_{a,b}\,\ner F_p}
     \ar@2[r]^-{\ner\B'_b\,\widehat{\ner}F_a}
&\ner\B'_b\,\ner F_q\,\ner\B_a
     \ar@2[r]^-{\widehat{\ner}F_b\,\ner\B_a}
&\ner F_n\,\ner\B_b\,\ner\B_a\ar@2[d]^{\ner F_n\, \widehat{\ner}\B_{a,b}}
\\ \ner\B'_{ab}\,\ner F_p\ar@2[rr]^{\widehat{\ner}F_{ab}}
&&\ner F_n\,\ner\B_{ab}
}
$$
commutes, since
\begin{align}\nonumber
(\ner F_n\,\widehat{\ner}\B_{a,b})&\circ (\widehat{\ner}F_b\,\ner\B_a)\circ (\ner\B'_b\, \widehat{\ner}F_a)\\
\nonumber &=(R'_nF_*J_nR_nb^*\v_qa^*J_p)\circ
(R'_nb^*\v'_qF_*J_qR_qa^*J_p)\circ(R'_nb^*J'_qR'_qa^*\v'_pF_*J_p)\\
\nonumber &\overset{(\ref{newiqu})}=(R'_nb^*F_*\v_qa^*J_p)\circ
(R'_nb^*\v'_qF_*J_qR_qa^*J_p)\circ(R'_nb^*J'_qR'_qa^*\v'_pF_*J_p)\\
\nonumber &\overset{(\ref{commutativeDiagram1})}=(R'_nb^*\v'_qa^*F_*J_p)\circ(R'_nb^*J'_qR'_qa^*\v'_pF_*J_p)
\\
&\nonumber \overset{(\ref{natvv})}=(R'_nb^*a^*\v'_pF_*J_p)\circ (R'_nb^*\v'_qa^*J'_pR'_pF_*J_p)
= \widehat{\ner}F_{ab}
\circ (\widehat{\ner}\B'_{a,b}\,\ner F_p).
\end{align}

To finish, let $F:\B\to \B'$ and $F':\B'\to \B''$ be lax functors.
Then, $\ner F'\, \ner F=\ner (F'F)$ and $\ner 1_\B=1_{\ner\B}$
since, at any $[p]$  and $a:[q]\to[p]$ in $\Delta$, we have
\begin{align}\nonumber
\ner F'_p\ \ner F_p&=R''F'_*J'_pR'_pF_*J_p\overset{(\ref{newiqu})}=R''_pF'_*F_*J_p=
R''_p(F'F)_*J_p=\ner(F'F)_p,
\\ \nonumber
\widehat{\ner F'\ner F}_a&=\ner F'_q\,\widehat{\ner}F_a \circ \widehat{\ner}F'_a\,\ner F_p =
(R''_qF'_*J'_qR'_qa^*\v'_pF_*J_p)\circ (R''_qa^*\v''_pF'_*J'_pR'_pF_*J_p)\\
\nonumber  &
\overset{(\ref{newiqu})}= (R''_qa^*F'_*\v'_pF_*J_p)\circ
(R''_qa^*\v''_pF'_*J'_pR'_pF_*J_p)\overset{(\ref{commutativeDiagram1})}= R''_qa^*\v''_pF'_*F_*J_p
= \widehat{\ner}(F'F)_a,
\\ \nonumber
\ner1_p&=R_p1_*J_p\overset{(\ref{twosquares})}=R_pJ_p =1_{\ner\B_p}
,
\\ \nonumber
\widehat{\ner}1_{a}&=R_qa^*\v_p1_*J_p=R_qa^*\v_pJ_p=R_qa^*1_{J_p}=1_{R_qa^*J_p}=1_{\ner\B_a}.
\end{align}

This completes the proof of Theorem \ref{groner} and lets us prepare
to prove the first part of Lemma \ref{fact1}.

\begin{corollary}\label{pfact1}
The assignment $\B\mapsto \BB\B$ is the function on objects of a
functor
$$\BB:\mathbf{Lax}\to\mathbf{Top}.$$
\end{corollary}
\begin{proof} By Theorem \ref{groner}, any lax functor $F:\B\to\B'$ gives rise to a lax simplicial functor
$\ner F:\ner \B\to \ner \B'$, hence to a functor $\int_\Delta \ner
F:\int_\Delta \ner \B\to \int_\Delta \ner \B'$ and then to a
cellular map $\BB F:\BB\B\to \BB\B'$. For $F=1_\B$, we have
$\int_\Delta\ner 1_\B= \int_\Delta
1_{\ner\B}=1_{\int_\Delta\ner\B}$, whence $\BB 1_\B=1_{\BB\B}$. For
any other lax functor  $F':\B'\to \B''$, the equality $\ner F'\,
\ner F=\ner (F'F)$  gives that $\xymatrix{\int_\Delta\ner
(F'F)=\int_\Delta\ner F'\ner F=\int_\Delta\ner F'\int_\Delta\ner
F}$, whence $\BB(F'F)=\BB F'\, \BB F$.
\end{proof}

In ~\cite[Definition 5.2]{CCG2010}, Carrasco, Cegarra, and Garz\'on
defined the \emph{categorical geometric nerve} of a bicategory $\B$
as the simplicial category
$$
\underline{\Delta}\B: \Delta^{\!op} \to  \Cat, \hspace{0.4cm}
[p]\mapsto \brep([p],\B),
$$
whose category of $p$-simplices is the category of lax functors
$\x:[p]\to \B$, with relative to objects lax transformations (i.e.,
icons) between them as arrows.
  The proposition below shows how $\underline{\Delta}\B$ relates with
the Grothendieck nerve $\ner\B$.

\begin{proposition} For any bicategory $\B$, there is a lax simplicial functor
\begin{equation}\label{eqr}
R=(R,\widehat{R}):\underline{\Delta}\B \to \ner\B
\end{equation}
inducing a homotopy equivalence
\begin{equation}\label{eqbr} \begin{array}{l}\xymatrix{\BB\!\int_\Delta\! R :\, \BB\!\int_\Delta\!
\underline{\Delta}\B
\ar[r]^-{\sim}& \BB\!\int_\Delta\!
\ner\B=\BB\B,}
\end{array}
\end{equation}
which is natural in $\B$ on lax functors. That is, for any lax
functor $F:\B\to\B'$, the square of spaces below commutes.
\begin{equation}\label{frcomm}\begin{array}{l}
\xymatrix@C=60pt{\BB\!\int_\Delta\!\underline{\Delta}\B \ar[r]^{\BB\!\int_\Delta\! R}
\ar[d]_{\BB\!\int_\Delta\! \underline{\Delta}F}& \BB\B\ar[d]^{\BB F}\\
\BB\!\int_\Delta\!\underline{\Delta}\B' \ar[r]^{\BB\!\int_\Delta \!R'}& \BB\B'}
\end{array}
\end{equation}
\end{proposition}
\begin{proof} At any object  $[p]$ of the simplicial category, $R$ is given by the functor in $(\ref{plrg})$
$$R_p:\underline{\Delta}\B_p= \brep([p],\B)\longrightarrow \brepg(\mathcal{G}_p,\B)= \ner\B_p,$$
and, at any map $a:[q]\to [p]$, the natural transformation
$$
\xymatrix{
\underline{\Delta}\B_p \ar[r]^{a^*} \ar[d]_{R_p}
        \ar@<25pt>@{}[d]|(.37){\widehat{R}_a}|(.55){\Rightarrow}
& \underline{\Delta}\B_q\ar[d]^{R_q}
\\ \ner \B_p\ar[r]_{\ner\B_a}
& \ner\B_q,
}
$$
is defined by $ \xymatrix@C=50pt{\ner\B_a R_p=R_qa^*
J_pR_p\ar@2[r]^-{\widehat{R}_a=R_qa^*\v_p} & R_qa^*}$.  When
$a=1_{[p]}$, clearly $\widehat{R}_{1_{[p]}}=R_p\v_p=1_{R_p}$ and,
for any $b:[n]\to [q]$, the commutativity coherence condition
$$
\xymatrix@C=50pt{
\ner\B_b \,\ner\B_a\, R_p\ar@2[d]_{\widehat{\ner}\B_{a,b}R_p}\ar@2[r]^-{\ner\B_b\widehat{R}_a}&\ner\B_b\,R_q\,a^*\ar@2[d]^{\widehat{R}_ba^*}\\
\ner\B_{ab}\,R_p\ar@2[r]^-{\widehat{R}_{ab}}&R_nb^*a^*=R_n(ab)^*,
}
$$
holds since, by $(\ref{natvv})$, $R_nb^*a^*\v_p\circ
R_nb^*\v_qa^*J_pR_p=R_nb^*\v_qa^* \circ R_nb^*J_qR_qa^*\v_p $.

By ~\cite[Corollary 1]{Quillen1973}, every functor
$R_p:\underline{\Delta}\B_p\to \ner\B_p$ induces a homotopy
equivalence on classifying spaces $\BB
R_p:\BB\underline{\Delta}\B_p\overset{\sim}\to \BB\ner\B_p$ since it
has the functor $J_p$ in $(\ref{plrg})$ as a left adjoint. Then, the
induced map in $(\ref{eqbr})$ is actually a homotopy equivalence by
~\cite[Corollary 3.3.1]{Thomason1979}.

Now let $F:\B\to\B'$ be any lax functor. Then, the square
$$
\xymatrix{\underline{\Delta}\B\ar[r]^{R}\ar[d]_{\underline{\Delta} F}&\ner\B\ar[d]^{\ner F}\\
\underline{\Delta}\B'\ar[r]^{R'}&\ner\B'
}
$$
commutes since,  for any integer $p\geq 0$ and $a:[q]\to [p]$, we
have
\begin{align}\nonumber
\ner F_p\, R_p&= R'_pF_*J_pR_p \overset{(\ref{newiqu})}=R'_pF_*=R'_p\,\underline{\Delta}F_p,
\\ \nonumber
\widehat{\ner F R}_a &= \ner F_q\widehat{R}_a\circ \widehat{\ner}F_aR_p= R'_qF_*J_qR_qa^*\v_p \circ R'_qa^*\v'_pF_*
J_pR_p\\ \nonumber &\overset{(\ref{newiqu})}= R'_qa^*F_*\v_p\circ R'_qa^*\v'_pF_*
J_pR_p
\overset{(\ref{commutativeDiagram1})}=R'_qa^*\v'_pF_* =\widehat{R}_a F_*= \widehat{R'\underline{\Delta}F}_a.
\end{align}

Hence, the commutativity of the square $(\ref{frcomm})$ follows:
\begin{align} \nonumber \xymatrix{
\BB F\,\, \BB\!\int_\Delta\!R= \BB\!\int_\Delta\!\ner F\, \,\BB\!\int_\Delta\!R=
\BB( \int_\Delta\!\ner F\, \int_\Delta\!R) =
\BB\int_\Delta(\ner F\, R)
}\\
\nonumber
\xymatrix{=\BB\int_\Delta(R'\,\underline{\Delta}F)=\BB( \int_\Delta\!R'\int_\Delta\underline{\Delta}F)
=\BB\! \int_\Delta\!R' \,\,\BB\!\int_\Delta\underline{\Delta}F.
}
\end{align}
\end{proof}

We are now ready to complete the proof of Lemmas \ref{fact1} and
\ref{facts2}.

\begin{corollary}\label{pfact2}   For any bicategory $\B$, there is a homotopy equivalence
\begin{equation}\label{kappa}
\xymatrix@C=15pt{\kappa:|\Delta\B|\ar[r]^-{_\sim}& \BB\B,}
\end{equation}
which is  homotopy natural on lax functors. That is, for any lax
functor $F:\B\to \B'$, there is
 a homotopy $\kappa'\, |\Delta F|\Rightarrow \BB F\, \kappa$,
\begin{equation}\label{squkap}
\xymatrix{|\Delta\B|\ar[r]^-{\kappa}\ar[d]_{|\Delta F|}
\ar@{}@<25pt>[d]|{\Rightarrow}
& \BB\B\ar[d]^{\BB F}\\
|\Delta\B'|\ar[r]^-{\kappa'}& \BB\B'.}
\end{equation}
\end{corollary}
\begin{proof} Let $\mathrm{N}\underline{\Delta}\B:\Delta^{\!op}\to \mathbf{Simpl}\Set$ be
the bisimplicial set obtained from the simplicial category
$\underline{\Delta}\B:\Delta^{\!op}\to \Cat$ with the nerve of
categories functor $\mathrm{N}:\Cat \to \mathbf{Simpl}\Set$.

As $\Delta\B$ is the simplicial set of objects of the simplicial
category $\underline{\Delta}\B$, if we regard $\Delta\B$ as a
discrete simplicial category (i.e., with only identities as arrows),
we have a simplicial category inclusion map $\Delta\B\hookrightarrow
\underline{\Delta}\B$, whence a bisimplicial inclusion map
$\mathrm{N}\Delta\B\hookrightarrow \mathrm{N}\underline{\Delta}\B$,
where $\mathrm{N}\Delta\B$ is the bisimplicial set that is constant
the simplicial set $\Delta\B$ in the vertical direction.
 Then, we have an induced simplicial set map on diagonals
$i:\Delta\B\to \mathrm{diag}\,\mathrm{N}\underline{\Delta}\B$. This
map is clearly natural in $\B$ on lax functors and, by
~\cite[Theorem 6.2]{CCG2010}, it induces a homotopy equivalence on
geometric realizations. Furthermore, a result by Bousfield and Kan
~\cite[Chap. XII, 4.3]{BK1972} and  Thomason's Homotopy Colimit
Theorem ~\cite{Thomason1979} give us the existence  of simplicial
maps $\mu:\mathrm{hocolim}\,\mathrm{N}\underline{\Delta}\B\to
\mathrm{diag}\,\mathrm{N}\underline{\Delta}\B$ and
$\eta:\mathrm{hocolim}\,\mathrm{N}\underline{\Delta}\B \to
\mathrm{N}\!\int_\Delta\!\underline{\Delta}\B$, which are  natural
on lax functors and both induce homotopy equivalences on geometric
realizations.

We then have a chain of homotopy equivalences between spaces
$$
\xymatrix@C=35pt{|\Delta\B|\ar[r]^-{|i|}&|\mathrm{diag}\,\mathrm{N}\underline{\Delta}\B|&
|\mathrm{hocolim}\,\mathrm{N}\underline{\Delta}\B| \ar[l]_-{|\mu|} \ar[r]^-{|\eta|}&
\BB \int_\Delta\!\underline{\Delta}\B\ar[r]^-{\BB\! \int_\Delta\! R}& \BB\B,
}
$$
where the last one on the right is  the homotopy equivalence
$(\ref{eqr})$, all of them natural on lax functors $F:\B\to \B'$.
Therefore, taking $|\mu|^\bullet:
|\mathrm{diag}\,\mathrm{N}\underline{\Delta}\B| \to |
\mathrm{hocolim}\,\mathrm{N}\underline{\Delta}\B|$ to be any
homotopy inverse map of $|\mu|$, we have a homotopy equivalence
$$ \xymatrix{\kappa= \BB\! \int_\Delta\! R \cdot |\eta| \cdot |\mu|^\bullet\cdot |i|:|\Delta\B|
\ar[r]^-{_\sim}& \BB\B,}$$
which is homotopy natural on lax functors, as required.
\end{proof}

\begin{corollary}\label{proffi} If $F,F':\B\to\B'$ are two lax functors
between bicategories, then any lax or oplax transformation between
them $\alpha:F\Rightarrow F'$ determines a homotopy, $\BB\alpha:\BB
F\Rightarrow \BB F':\BB\B\to \BB\B'$, between the induced maps on
classifying spaces.
\end{corollary}
\begin{proof} In the proof of ~\cite[Proposition 7.1 (ii)]{CCG2010} it is proven that
any $\alpha:F\Rightarrow G$ gives rise to a homotopy
$H(\alpha):|\Delta F|\Rightarrow |\Delta F'|:|\Delta \B|\to |\Delta
\B'|$. Then, a homotopy $\BB\alpha:\BB F\Rightarrow \BB F'$ is
obtained as the composite of the homotopies
$$
\BB F\Longrightarrow \BB F \kappa  \kappa^\bullet\overset{(\ref{squkap})}\Longrightarrow
\kappa' |\Delta F| \kappa^\bullet \overset{\kappa' H(\alpha)\kappa^\bullet}
\Longrightarrow \kappa' |\Delta F'| \kappa^\bullet\overset{(\ref{squkap})}\Longrightarrow \BB F'\kappa\kappa^\bullet\Longrightarrow \BB F',
$$
where $\kappa^\bullet$ is a homotopy inverse of the homotopy
equivalence $\kappa:|\Delta\B|\to \BB\B$ in $(\ref{kappa})$.
\end{proof}

\bibliography{library}{}
\bibliographystyle{amsplain}

\end{document}